\documentclass[11pt, a4paper]{amsart}
\usepackage{amsmath,amsfonts,amssymb,amsthm, color, mathrsfs}
\usepackage[all]{xy}
\usepackage{hyperref}
\usepackage{bm}
\usepackage{relsize}
\allowdisplaybreaks
\DeclareMathOperator{\End}{End}

\DeclareMathOperator{\ad}{ad}

\begin{document}
\newtheorem{thm}{Theorem}[section]
\newtheorem{prop}[thm]{Proposition}
\newtheorem{coro}[thm]{Corollary}
\newtheorem{conj}[thm]{Conjecture}
\newtheorem{example}[thm]{Example}
\newtheorem{lem}[thm]{Lemma}
\newtheorem{rem}[thm]{Remark}
\newtheorem{convention}[thm]{Convention}
\newtheorem{hy}[thm]{Hypothesis}
\newtheorem*{acks}{Acknowledgements}
\theoremstyle{definition}
\newtheorem{de}[thm]{Definition}
\newtheorem{bfproof}[thm]{{\bf Proof}}

\newcommand{\C}{{\mathbb{C}}}
\newcommand{\Z}{{\mathbb{Z}}}
\newcommand{\N}{{\mathbb{N}}}
\newcommand{\Q}{{\mathbb{Q}}}
\newcommand{\te}[1]{\textnormal{{#1}}}
\newcommand{\set}[2]{{
    \left.\left\{
        {#1}
    \,\right|\,
        {#2}
    \right\}
}}
\newcommand{\sett}[2]{{
    \left\{
        {#1}
    \,\left|\,
        {#2}
    \right\}\right.
}}
\def \<{{\langle}}
\def \>{{\rangle}}

\def \({ \left( }
\def \){ \right) }
\def \:{\mathopen{\overset{\circ}{
    \mathsmaller{\mathsmaller{\circ}}}
    }}
\def \;{\mathclose{\overset{\circ}{\mathsmaller{\mathsmaller{\circ}}}}}


\newcommand{\g}{{\mathfrak g}}
\newcommand{\fg}{\g}
\newcommand{\fk}{{\frak k}}
\newcommand{\gc}{{\bar{\g}}}
\newcommand{\h}{{\frak h}}
\newcommand{\n}{{\frak n}}
\newcommand{\borel}{{\frak b}}
\newcommand{\cent}{{\frak c}}
\newcommand{\notc}{{\not c}}
\newcommand{\Loop}{{\cal L}}
\newcommand{\G}{{\cal G}}
\newcommand{\al}{\alpha}
\newcommand{\alck}{\al^\vee}
\newcommand{\be}{\beta}
\newcommand{\beck}{\be^\vee}
\newcommand{\rk}{{\mathrm{k}}}
\newcommand{\rd}{{\mathrm{d}}}
\newcommand{\ft}{{\mathfrak{t}}}
\newcommand{\rc}{{\mathrm{c}}}
\newcommand{\CH}{{\mathcal{H}}}
\newcommand{\ul}{\mathcal U_\hbar^l(\hat\g_\mu)}
\newcommand{\uf}{\mathcal U_\hbar^f(\hat\g_\mu)}
\newcommand{\us}{\mathcal U_\hbar^s(\hat\g_\mu)}
\newcommand{\hf}{\mathcal U_\hbar^f(\hat\h_\mu)}
\newcommand{\hl}{\mathcal U_\hbar^l(\hat\h_\mu)}

\newcommand{\chg}{\check{\mathfrak g}}
\newcommand{\chh}{\check{\mathfrak h}}
\newcommand{\cha}{\check{a}}
\newcommand{\chb}{\check{b}}
\newcommand{\ka}{\mathfrak{a}}
\newcommand{\chka}{\check{\mathfrak{a}}}
\newcommand{\CL}{\mathcal L}

\newcommand{\rtu}{{\xi}}
\newcommand{\period}{{N}}
\newcommand{\half}{{\frac{1}{2}}}
\newcommand{\quar}{{\frac{1}{4}}}
\newcommand{\oct}{{\frac{1}{8}}}
\newcommand{\hex}{{\frac{1}{16}}}
\newcommand{\reciprocal}[1]{{\frac{1}{#1}}}
\newcommand{\inverse}{^{-1}}
\newcommand{\inv}{\inverse}
\newcommand{\SumInZm}[2]{\sum\limits_{{#1}\in\Z_{#2}}}
\newcommand{\uce}{{\mathfrak{uce}}}
\newcommand{\choice}[2]{{
\left[
\begin{array}{c}
{#1}\\{#2}
\end{array}
\right]
}}

\newcommand{\wh}[1]{\widehat{#1}}

\newlength{\dhatheight}
\newcommand{\dwidehat}[1]{%
    \settoheight{\dhatheight}{\ensuremath{\widehat{#1}}}%
    \addtolength{\dhatheight}{-0.35ex}%
    \widehat{\vphantom{\rule{1pt}{\dhatheight}}%
    \smash{\widehat{#1}}}}
\newcommand{\dhat}[1]{%
    \settoheight{\dhatheight}{\ensuremath{\hat{#1}}}%
    \addtolength{\dhatheight}{-0.35ex}%
    \hat{\vphantom{\rule{1pt}{\dhatheight}}%
    \smash{\hat{#1}}}}

\newcommand{\dwh}[1]{\dwidehat{#1}}

\newcommand{\dis}{\displaystyle}

\newcommand{\ot}{\otimes}

\newcommand{\overit}[2]{{
    \mathop{{#1}}\limits^{{#2}}
}}
\newcommand{\belowit}[2]{{
    \mathop{{#1}}\limits_{{#2}}
}}
\xymatrixcolsep{5pc}

\newcommand{\orb}[1]{|\mathcal{O}({#1})|}
\newcommand{\up}{_{(p)}}
\newcommand{\uq}{_{(q)}}
\newcommand{\upq}{_{(p+q)}}
\newcommand{\uz}{_{(0)}}
\newcommand{\uk}{_{(k)}}
\newcommand{\nsum}{\SumInZm{n}{\period}}
\newcommand{\ksum}{\SumInZm{k}{\period}}
\newcommand{\overN}{\reciprocal{\period}}
\newcommand{\df}{\delta\left( \frac{\xi^{k}w}{z} \right)}
\newcommand{\dfl}{\delta\left( \frac{\xi^{\ell}w}{z} \right)}
\newcommand{\ddf}{\left(D\delta\right)\left( \frac{\xi^{k}w}{z} \right)}

\newcommand{\ldfn}[1]{{\left( \frac{1+\xi^{#1}w/z}{1-{\xi^{#1}w}/{z}} \right)}}
\newcommand{\rdfn}[1]{{\left( \frac{{\xi^{#1}w}/{z}+1}{{\xi^{#1}w}/{z}-1} \right)}}
\newcommand{\ldf}{{\ldfn{k}}}
\newcommand{\rdf}{{\rdfn{k}}}
\newcommand{\ldfl}{{\ldfn{\ell}}}
\newcommand{\rdfl}{{\rdfn{\ell}}}

\newcommand{\kprod}{{\prod\limits_{k\in\Z_N}}}
\newcommand{\lprod}{{\prod\limits_{\ell\in\Z_N}}}
\newcommand{\E}{{\mathcal{E}}}
\newcommand{\F}{{\mathcal{F}}}
\newcommand{\R}{{\mathcal{R}}}
\newcommand{\fsl}{{\mathfrak{sl}}}
\newcommand{\fgl}{{\mathfrak{gl}}}

\newcommand{\wt}[1]{\widetilde{#1}}

\newcommand{\tar}{{\uce\left(\Loop\left(\gc,\mu\right)\right)}}
\newcommand{\U}{{\mathcal{U}}}
\newcommand{\qtar}{\U_\hbar\left(\hat{\g}_\mu\right)}
\newcommand{\qtartemp}[1]{\U_{q^{#1}}\left(\hat{\g}_\mu\right)}
\newcommand{\qptar}{\wt{\U}_\hbar\left(\hat{\g}_\mu\right)}
\newcommand{\qhei}{{\U_\hbar( \hat{\h}_\mu )}}
\newcommand{\htart}[1]{\mathcal{#1_\hbar}(\hat\fg_\mu)}
\newcommand{\hheit}[1]{\mathcal{#1_\hbar}(\hat\h_\mu)}
\newcommand{\qptart}[1]{\wt{\mathcal{#1}}_\hbar(\wh\fg_\mu)}
\newcommand{\qphei}{{\wh\U_\hbar(\wh\h_\mu)}}
\newcommand{\qpheit}[1]{\wh{\mathcal{#1}}_\hbar(\hat\h_\mu)}
\newcommand{\symalg}{{S\left( \wh{\h}^-_\mu \right)}}

\newcommand{\comp}{{\mathfrak{comp}}}

\newcommand{\ctimes}{{\widehat{\otimes}}}

\newcommand{\bigctimes}{{\ \widehat{\bigotimes}\ }}

\newcommand{\hctvs}[1]{Hausdorff complete linear topological vector space}

\newcommand{\act}{\triangleright}

\newcommand{\T}{\mathcal{T}}

\newcommand{\varprodright}{\mathop{\underrightarrow{\prod}}\limits}


\numberwithin{equation}{section}

\title[Twisted quantum affinization]{Twisted quantum affinizations
and quantization of extended affine Lie algebras}

\author{Fulin Chen$^1$}
\address{School of Mathematical Sciences, Xiamen University,
 Xiamen, China 361005} \email{chenf@xmu.edu.cn}\thanks{$^1$Partially supported by the NSF of China (Nos.11971397, 12161141001) and the Fundamental
Research Funds for the Central Universities (No.20720200067)}

\author{Naihuan Jing$^2$}
\address{Department of Mathematics, North Carolina State University, Raleigh, NC 27695,
USA}
\email{jing@math.ncsu.edu}
\thanks{$^2$Partially supported by the NSF of China (No.12171303) and the Simons Foundation (No.523868).}

\author{Fei Kong$^3$}
\address{Key Laboratory of Computing and Stochastic Mathematics (Ministry of Education), School of Mathematics and Statistics, Hunan Normal University, Changsha, China 410081} \email{kongmath@hunnu.edu.cn}

 \author{Shaobin Tan$^4$}
 \address{School of Mathematical Sciences, Xiamen University,
 Xiamen, China 361005} \email{tans@xmu.edu.cn}
 \thanks{$^4$Partially supported by the NSF of China  (No.12131018).}


\begin{abstract}
In this paper, for an arbitrary Kac-Moody Lie algebra $\mathfrak g$ and a diagram automorphism $\mu$ of $\mathfrak g$ satisfying certain natural linking conditions, we introduce and study a $\mu$-twisted quantum affinization algebra $\mathcal U_\hbar\left(\hat{\mathfrak g}_\mu\right)$ of $\mathfrak g$.
 When $\mathfrak g$ is of finite type, $\mathcal U_\hbar\left(\hat{\mathfrak g}_\mu\right)$ is Drinfeld's current algebra realization of the twisted
 quantum affine algebra.
When $\mu=\mathrm{id}$ and $\mathfrak g$ in affine type, $\mathcal U_\hbar\left(\hat{\mathfrak g}_\mu\right)$ is the quantum toroidal
 algebra introduced by Ginzburg, Kapranov and Vasserot.
As the main results of this paper,  we first prove a triangular decomposition for $\mathcal U_\hbar\left(\hat{\mathfrak g}_\mu\right)$.
 Second, we  give a simple characterization of the affine quantum Serre relations on restricted  $\mathcal U_\hbar\left(\hat{\mathfrak g}_\mu\right)$-modules in terms of ``normal order products''.
Third, we  prove that the category of restricted $\mathcal U_\hbar\left(\hat{\mathfrak g}_\mu\right)$-modules is a monoidal category and
hence obtain a topological Hopf algebra structure on the ``restricted completion'' of $\mathcal U_\hbar\left(\hat{\mathfrak g}_\mu\right)$.
Last, we  study the classical limit of $\mathcal U_\hbar\left(\hat{\mathfrak g}_\mu\right)$ and abridge it to the quantization
theory of extended affine Lie algebras. In particular, based on a classification result of Allison-Berman-Pianzola, we obtain the $\hbar$-deformation  of all nullity $2$ extended affine Lie algebras.
\end{abstract}
\maketitle

\section{Introduction}
In this paper, we denote by $\mathcal U_\hbar(\g)$  the quantum enveloping algebra associated with a symmetrizable Kac-Moody Lie algebra $\g$
 over the ring $\C[[\hbar]]$ of complex formal series in the variable $\hbar$.

\subsection{Twisted quantum affinizations}
The  quantum affine algebras $\U_\hbar(\hat\g)$
($\hat\g$ being the affine Lie algebra of the finite dimensional Lie algebra $\g$)
are  one of the most important subclasses of the quantum enveloping algebras
introduced independently by Drinfeld \cite{Dr-hopf-alg} and Jimbo \cite{jim-start}.
Drinfeld's current algebra realization 
for $\U_\hbar(\hat\g)$ \cite{Dr-new,beck} is of fundamental importance in this area as it provides
an explicit construction of $\U_\hbar(\hat\g)$ as  quantum affinization of
$\U_\hbar(\g)$.
It is remarkable that the Drinfeld  quantum affinization process
can be extended to any symmetrizable quantum Kac-Moody
algebras $\mathcal U_\hbar(\g)$. One obtains in this way a new class of
$\C[[\hbar]]$-algebras $\mathcal U_\hbar(\hat{\g})$, called the quantum affinization algebras.
The algebras  $\mathcal U_\hbar(\hat{\g})$ ($\g$ not of finite type)
 were first introduced by Ginzburg-Kapranov-Vasserot \cite{GKV} for the untwisted affine  case,
and then in  \cite{J-KM,Naka-quiver,He-representation-coprod-proof} for the general case.
(In the case that $\g$ is symmetric but not simply-laced, the quantum affinization algebra $\U_\hbar(\hat\g)$ defined in \cite{J-KM,Naka-quiver}
is slightly different from that 
in \cite{He-representation-coprod-proof}, and in this paper we shall
exploit the latter definition).
When $\g$ is of (untwisted) affine type,
 $\mathcal U_\hbar(\hat\g)$ is often referred as the quantum toroidal algebra.
The case $\g=\hat\fsl_{\ell+1}$ is particularly interesting:
 an additional parameter $p$ can be added in the quantum affinization
process \cite{GKV} so 
one gets a two parameter quantum toroidal algebra $\mathcal U_{\hbar,p}(\hat\g)$.
One notices that such a two-parameter quantum toroidal algebra is yet to be defined for
other affine types.

The theory of general quantum affinizations has been intensively studied.
In particular, the representation theory of quantum toroidal algebras is very rich:
the Schur-Weyl duality  with the elliptic Cherednik algebras \cite{VV-schur}; the action on $q$-Fock spaces
\cite{VV-double-loop,STU,TU-level-zero}; 
the vertex representation constructions \cite{Sa,GJ};
the connection with the McKay correspondence \cite{FJW}; the quantum fusion tensor category constructions \cite{He-representation-coprod-proof,He-drinfeld-coproduct};
the geometric  representation constructions
\cite{Naka-quiver} as well as other applications \cite{miki-1999-1}-\cite{miki-2006}, \cite{FJM1}-\cite{FJMM2}, \cite{Ts}.
For a survey on the representation theory of quantum toroidal algebras, one may consult \cite{He-total}.
One of the main features of the
quantum toroidal algebra is that it contains two remarkable subalgebras
both isomorphic
to a quotient of untwisted quantum affine algebras (one comes from the  Drinfeld-Jimbo's presentation and another
from Drinfeld's new presentation).

When $\g$ is of finite $ADE$ type, Drinfeld \cite{Dr-new} also formulated a twisted quantum affinization
of $\U_\hbar(\g)$ associated with a diagram automorphism $\mu$ of $\g$.
It was announced in \cite{Dr-new} and proved in \cite{Da1,Da2} (see also \cite{ZJ, JZ1}) that
Drinfeld's twisted quantum affinization algebras gave
 new realization of the
twisted quantum affine algebras $\U_\hbar(\hat\g_\mu)$ ($\hat\g_\mu$ being the $\mu$-twisted affine Lie algebra of $\g$).
 Similar to the untwisted case,  Drinfeld's new realization theorem is of
 fundamental importance in the  theory of  twisted quantum affine algebras
\cite{CP3-twisted-quantum,J-inv,H-tw}.
It has been a question  whether the twisted quantum affinization
 process introduced by Drinfeld can be generalized to diagram automorphisms of any symmetrizable
Kac-Moody Lie algebra in a similar way to that of the untwisted  case.
In \cite{CJKT-twisted-quantum-aff-vr}, by using vertex operator calculations, we generalized the twisted quantum affinization to
a class of diagram automorphisms on simply-laced Kac-Moody Lie algebras.
One notices that
in general there exist nontrivial diagram automorphisms on non-simply-laced Kac-Moody Lie algebras that are not from finite types.
For example, when $\g$ is of affine type $B_\ell^{(1)}$ or
$C_\ell^{(1)}$, there exists such an order $2$ diagram automorphism of $\g$.
The open question is particularly interesting for such general Kac-Moody Lie algebras as one expects that
this could give rise to new type of quantized algebras.

In this paper, we give a general affirmative answer to this question by defining a {\it twisted
quantum affinization algebra} $\qtar$  for any
symmetrizable Kac-Moody
algebra $\g$ and any diagram automorphism $\mu$ of $\g$
with
two natural
linking conditions (LC1) and (LC2).
The condition (LC1), first recognized
in \cite{FSS}, ensures that the $\mu$-folded Cartan matrix of $\g$ is still a generalized
Cartan matrix. The condition (LC2) appear
naturally for establishing the Drinfeld-type {\it affine quantum Serre relations}.
One notices that, except for the transitive diagram automorphisms on $\hat\fsl_{\ell+1}$,
 all the diagram automorphisms of finite and affine Kac-Moody Lie algebras satisfy these two linking conditions.
 In the further study of 
the Drinfeld-type affine quantum Serre relations,
we also require an additional linking condition (LC3) on the diagram automorphism $\mu$. 
Every  diagram automorphism  of finite or affine Kac-Moody Lie algebra  satisfies this additional linking condition, 
except the case that it is an 
order $k+1$ diagram automorphism of $\hat\fsl_{2k+1}$.

When $\g$ is of finite type, the twisted
quantum affinization algebra $\qtar$ is nothing but Drinfeld's current algebra realization for the twisted
 quantum affine algebra.
 When $\mu=\mathrm{id}$, it coincides with the quantum affinization algebra $\U_\hbar(\hat{\g})$.
When $\g$ is of simply-laced type, $\qtar$ has been realized in \cite{CJKT-twisted-quantum-aff-vr} by vertex operators.
Especially,  in the case when $\g$ is of untwisted affine type $X_\ell^{(1)}$ and $\mu$ fixes an additional node of $\g$ \cite{Kac-book},
$\U_\hbar(\hat\g_\mu)$
 contains two subalgebras which are both isomorphic to quotients of the twisted quantum affine algebra
 of type $X_\ell^{(N)}$, where $N$ is 
 the order of $\mu$.

The theory of quantum Kac-Moody algebras
 has been a tremendously successful story.
 Certainly, the most important aspect of the structures of
 $\U_\hbar(\g)$  is that it is
 a Hopf algebra over $\C[[\hbar]]$, which specializes to the universal enveloping algebra $\U(\g)$ of $\g$ and admits a
 canonical triangular decomposition.
The next natural question is whether these fundamental algebraic properties  of
the quantum Kac-Moody algebras (triangular decomposition, Hopf algebra structure,
the specialization and etc) can be extended to general twisted quantum affinization algebras.

In Section \ref{sec:tri}, we  prove that $\qtar$ has a triangular decomposition.
When $\mu=\mathrm{id}$, this decomposition was known 
in \cite{beck} for the finite type
and in \cite{He-representation-coprod-proof} for the general case.
When $\g$ is of finite type, this decomposition was given
in \cite{Da2}.
In analogy to the untwisted case \cite{He-representation-coprod-proof,He-drinfeld-coproduct},
the triangular decomposition of $\qtar$ is essential in studying of
representation theory of $\qtar$. For example,
one can now define a notion of $l$-highest weight module for $\qtar$ and then study the properties of
(simple) integrable $l$-highest weight modules.

Unlike quantum Kac-Moody algebras, no Hopf algebra structure is known
for the general quantum affinizations.
According to an unpublished note of Drinfeld, certain completion of
the quantum affine algebra has a Hopf algebra structure and this (infinite) coproduct is conjugate to
the usual Drinfeld-Jimbo coproduct under a twist via the universal R-matrix.
For the untwisted quantum affine algebras, a  proof was given in \cite{DI-generalization-qaff} (the simply-laced  case)
and \cite{E-coprod,Gr-qshuff-qaff} (the other  cases).
One notices that Drinfeld's new coproduct involves infinite sums, so cannot be defined directly on $\qtar$.
However, it makes sense on the so-called restricted (topologically free) $\qtar$-modules.
When $\g$ is of finite type, the notion of  restricted $\qtar$-modules
coincides with the usual one   \cite{Luztig-quantum-book} and it is known that all  irreducible highest weight
$\qtar$-modules are restricted \cite{Lusztig-qdeform-simple-mod, EK-VI}.
When $\g$ is of simply-laced type, the vertex representations for $\qtar$ constructed in \cite{CJKT-twisted-quantum-aff-vr} are also restricted.
In Section 6  we prove that, if $\mu$ also satisfies  the linking condition (LC3), then 
 for any restricted $\qtar$-modules $U$ and $V$,
 Drinfeld's new coproduct affords a (restricted) $\qtar$-module structure on the $\hbar$-adically completed
 tensor product space $U\wh\ot V$.
In particular, this implies that the category $\R$ of  restricted $\qtar$-modules is a monoidal category.
Let $\mathcal F$ be the forgetful functor from $\R$ to the category of topologically free $\C[[\hbar]]$-modules.
Then   the closure  $\qptar$ of $\qtar$ in the algebra of
endomorphisms of $\mathcal F$, called the {\it restricted completion of $\qtar$},
 is naturally a topological Hopf algebra.
As usual, the key step
 is to check the compatibility between Drinfeld's new coproduct and the affine quantum Serre relations.
 For this purpose and for its own right, we present a simple characterization of affine quantum Serre relations on restricted
 $\qtar$-modules in terms of ``{\it normal order products}'', which enables us to overcome the difficulty.

More precisely, let $A=(a_{ij})_{i,j\in I}$ be the generalized Cartan matrix associated to $\g$
and  let $x_i^\pm(z), i\in I$ be the defining currents  of $\qtar$.
Then the untwisted affine quantum Serre relations in  $\U_\hbar(\hat\g)$ take the form ($i,j\in I$ with $a_{ij}<0$):
\begin{align*}
\sum_{\sigma\in S_{m}}
   \sum_{r=0}^{m}
    (-1)^r \binom{m}{r}_{q_i}x_i^\pm(z_{\sigma(1)})\cdots x_i^\pm(z_{\sigma(r)})
  x_j^\pm(w)
       x_i^\pm(z_{\sigma(r+1)})\cdots x_i^\pm(z_{\sigma(m)})=0,
\end{align*}
where $m=1-a_{ij}$ and $q_i$ an invertible element in $\C[[\hbar]]$.
Comparing with the untwisted case, one of the main features in our $\qtar$ is that
the affine quantum Serre relations can be written in a simple form in terms of  certain (Drinfeld)
polynomials
(see Definition \ref{de:tqaffine}):
\begin{align*}
&\sum_{\sigma\in S_{m}}\sum_{r=0}^{m}
    (-1)^r \binom{m}{r}_{q_i^{d_{ij}}}p_{ij,r}^\pm(z_{\sigma(1)},\dots,z_{\sigma(m)},w)
    x_i^\pm(z_{\sigma(1)})\cdots x_i^\pm(z_{\sigma(r)})\\
&\quad\qquad   \cdot x_j^\pm(w)
       x_i^\pm(z_{\sigma(r+1)})\cdots x_i^\pm(z_{\sigma(m)})\ =0,
\end{align*}
where $p_{ij,r}^\pm(z_1,\dots,z_m,w)$ are some polynomials  and $d_{ij}$ a positive integer.
This distinguishes
the theory of twisted quantum affinizations (in particular, the theory of twisted quantum affine algebras)
 from that of the untwisted one.
In Section 5, we introduce a notion of normal order product ``$\:\quad\;$'' for the currents $x_i^\pm(z)$ on restricted
$\qtar$-modules $W$. When $W$ is the vertex representation constructed in \cite{CJKT-twisted-quantum-aff-vr} so $x_i^\pm(z)$ are realized as certain
vertex operators, $\:x_i^\pm(z)x_j^\pm(z)\;$ is the usual normal order product of vertex operators defined by
moving the annihilation operators to the right.
Based on tedious OPE calculation, we
show in Theorem \ref{thm:Dr-equiv-normal-ordering} that (both twisted and untwisted) affine quantum Serre relations are equivalent to
the following simple and unified normal order product:
\begin{align}\label{newserre}
\:x_i^\pm(q_i^{a_{ij}}z)x_i^\pm(q_i^{a_{ij}+2}z)\cdots  x_i^\pm(q_i^{-a_{ij}-2}z)
  x_i^\pm(q_i^{-a_{ij}}z)x_j^\pm(z)\;=0.
\end{align}
In fact a more general and stronger version 
(see  Theorem \ref{prop:Dr-to-normal-ordering}) will be given,
and Theorem \ref{thm:Dr-equiv-normal-ordering} is a special situation. 
As another application of Theorem \ref{prop:Dr-to-normal-ordering}, we also prove that
for any finite or affine type Lie algebra $\g$,  
the following relations hold on restricted $\qtar$-modules:
\begin{align*}
\sum_{\sigma\in S_{\check{m}}}
   \sum_{r=0}^{\check{m}}
    (-1)^r \binom{\check{m}}{r}_{\check{q}_i}x_i^\pm(z_{\sigma(1)})\cdots x_i^\pm(z_{\sigma(r)})
  x_j^\pm(w)
       x_i^\pm(z_{\sigma(r+1)})\cdots x_i^\pm(z_{\sigma(\check{m})})=0,
\end{align*}
where $\check{m}=1-\check{a}_{ij}$, $\check{q}_i\in \C[[\hbar]]$ and $(\check{a}_{ij})$ is the $\mu$-folded matrix of $A$.
When $\g$ is of finite type, the above relations are proved in \cite{Da1},
which plays a key role in understanding the isomorphism between Drinfeld-Jimbo and Drinfeld's realizations
for twisted quantum affine algebras.

When $\g$ is of finite type, by taking classical limit, there is a natural vertex algebraic interpretation of the relations \eqref{newserre}:
let $W$ be a restricted module for the twisted affine Lie algebra $\hat\g_\mu$ and let
$V_\g$ be the universal affine vertex algebra associated to $\g$.
It is known \cite{Li-new-construction} that the currents $x_i^\pm(z)$ on $W$ generate a vertex algebra $V_W$
in the space $\mathrm{Hom}(W,W((z)))$ with $W$ as a quasi module.
Furthermore, there is a surjective  homomorphism from $V_\g$ to $V_W$ sending 
$x_i^\pm\mapsto x_i^\pm(z)$,
where the Chevalley generators $x_i^\pm$ of $\g$  are  viewed as elements of $V_\g$ in the usual way.
One can check directly that
$(x_i^\pm(z))_0\cdots (x_i^\pm(z))_0x_j^\pm(z)=d_{ij}^{a_{ij}-1}z^{(1-a_{ij})(1-d_{ij})}\:x_i^\pm(z)\cdots x_i^\pm(z)x_j^\pm(z)\;$ in $V_W$ and hence
\begin{align*}
(\ad x_i^\pm)^{1-a_{ij}}x_j^\pm=\((x_{i}^\pm)_{0}\)^{1-a_{ij}} x_{j}^\pm\mapsto
d_{ij}^{a_{ij}-1}z^{(1-a_{ij})(1-d_{ij})}\:x_i^\pm(z)\cdots x_i^\pm(z)x_j^\pm(z)\;.
\end{align*}
Then the relations \eqref{newserre} follows from the Serre relations on $\g$.

Finally, in Section 7 we study the classical limit of $\qtar$.
Let $\hat\g_\mu$ be the Lie algebra whose universal enveloping algebra is isomorphic to $\qtar|_{\hbar\mapsto 0}$.
We prove that there is a surjective homomorphism $\psi_{\g,\mu}$ from the Lie algebra $\hat\g_\mu$
to the twisted current Kac-Moody algebra of $\g$ associated with $\mu$.
When $\g$ is of finite type, it was proved in \cite{Da2} (see also \cite{CJKT-drin-pre}) that
$\psi_{\g,\mu}$ is an isomorphism.
In general, $\psi_{\g,\mu}$ is {\it not} an isomorphism and for the case that $\g$ is of indefinite type, the
kernel of $\psi_{\g,\mu}$ is unknown even when $\mu=\mathrm{id}$.
Based on a recent result 
\cite{CJKT-drin-pre}, for the case that $\g$ is of affine type, we
determine explicitly the kernel of $\psi_{\g,\mu}$ and establish its
ultimate connection with
quantization of nullity $2$ extended affine Lie algebras,
which leads to our second main topic explained below.

\subsection{Quantization of extended affine Lie algebras}
One of our main motivations for introducing and studying twisted quantum affinizations stems from a
fundamental problem
in the  theory of extended affine Lie algebras (EALAs for short): their quantizations.
The notion of EALAs was first introduced by 
physicists in \cite{H-KT} with applications to quantum gauge theory,
and since then it has been intensively studied
(see \cite{AABGP,BGK,N2}
and the references therein).
An EALA is   a complex Lie algebra $\E$, together with a ``Cartan subalgebra'' $\CH$ and an invariant form $(\ \mid\ )$,
that satisfies a list of conditions.
The form $(\ \mid\ )$ induces a semi-positive form on the $\mathbb R$-span of the root system $\Phi$
 of $\E$ (relative to $\CH$) and so $\Phi$  divides into a disjoint union of the sets of isotropic and nonisotropic roots.
Roughly speaking, the structure of $\E$ is determined by its core $\E_c$,
the subalgebra  generated by nonisotropic root vectors.
Following \cite{BGK}, we say that an EALA $\E$ is maximal  if $\E_c$ is centrally closed.
Set  $\bar{\E}=\E_c+\CH$, which we
 call the extended core of $\E$.
One of the conditions for $\E$ requires  that the
 group generated by isotropic roots is of finite rank and this rank is called its nullity. Nullity $0$
EALAs are exactly  finite dimensional simple Lie algebras, while
nullity $1$ EALAs   are precisely
 affine Kac-Moody algebras  \cite{ABGP}.
Meanwhile,  nullity $2$ EALAs are of particular interest: they are closely related
to the Lie algebras studied by Saito and Slodowy in the work of elliptic
singularities \cite{Sa}.
As an important achievement in the theory of EALAs, the centerless cores of nullity $2$ EALAs have been completely classified by Allison-Berman-Pianzola in \cite{ABP} (see also \cite{GP-torsors}).

We remark that quantum finite, affine and (one or two parameters) toroidal algebras can be related to EALAs in a uniform way.
More precisely,
 let $\U_\hbar$ be an arbitrary  quantum finite (resp.\,affine; resp.\,toroidal) algebra.
 Then  there is an EALA $\E$ of nullity $0$ (resp.\,$1$; resp.\,$2$)
such that the classical limit of $\U_\hbar$ is
 isomorphic to
the universal enveloping algebra  of $\bar{\E}$.
Furthermore, except for the quantum toroidal algebra of type $A_1^{(1)}$, $\E$ are always maximal.
(Note that $A_1^{(1)}$ is the unique symmetric but non-simply-laced affine GCM
and so in this case the quantum toroidal algebra defined in \cite{J-KM,Naka-quiver} is slightly different from that in
 \cite{He-representation-coprod-proof}.
In the former case, the correspond EALA  is also maximal.)
Therefore,  an eminent problem in the theory of EALAs is to establish and explore natural connections of quantum algebras with all extended cores of (maximal)  EALAs.
The theory of quantum toroidal algebras suggests that the nullity $2$ case is of particular importance. 

As an application of our 
 general twisted quantum affinization theory,
 we solve this problem for the nullity $2$ case.
More precisely, denote by $\widehat{E}_2$
the class of Lie algebras which are isomorphic to the extended core of a maximal EALA  with nullity $2$.
Let $\mu$ be a diagram automorphism of an affine Kac-Moody algebra $\g$ and
 let $\ft(\g,\mu)$ be the twisted toroidal Lie algebra associated to $(\g,\mu)$, that is, the universal central extension of the $\mu$-twisted
 loop algebra of $[\g,\g]$.
By adding two canonical derivations to $\ft(\g,\mu)$, one obtains a semi-product Lie algebra $\hat{\ft}(\g,\mu)$.
Then we have $\hat\ft(\g,\mu)\in \widehat{E}_2$ if and only if $\mu$ satisfies the linking conditions (LC1) and (LC2) \cite{ABP}.
On the other hand, let $\C_p$ be the quantum $2$-torus associated to a nonzero complex number $p$ and
 let $\tilde\fsl_{\ell+1}(\C_p)$ be the universal central extension of the special linear Lie algebra over
$\C_p$ \cite{BGK}.
Again by adding two derivations to $\tilde\fsl_{\ell+1}(\C_p)$, we have a Lie algebra $\hat\fsl_{\ell+1}(\C_p)\in \wh{E}_2$.
According to a result of  Allison-Berman-Pianzola \cite{ABP},  any algebra in $\widehat{E}_2$ is either isomorphic to  $\hat\ft(\g,\mu)$ with $\mu$ nontransitive,
or isomorphic to $\hat\fsl_{\ell+1}(\C_p)$ with  $p$ generic.
It was known that the quantization of $\hat\fsl_{\ell+1}(\C_p)$ is  $\mathcal U_{\hbar,p}({\dwh{\frak sl}}_{\ell+1})$ \cite{VV-double-loop},
while for the case $\g=\hat{\fsl}_2$, we define in \cite{CJKT-quantum-A111} a new quantum toroidal algebra $\U_\hbar^{new}$
whose classical limit is $\U(\hat\ft(\hat\fsl_2,\mathrm{id}))$.
Just like the  algebra $\U_\hbar(\dwh{\fsl}_2)$,
$\U_\hbar^{new}$  has a canonical triangular decomposition and a Hopf algebra structure \cite{CJKT-quantum-A111}.
Thus, it remains to treat the
algebras $\hat\ft(\g,\mu)$ with $\g$  not of type
$A_1^{(1)}$ and $\mu$ nontransitive.
In this case, by using the Drinfeld type presentations of $\hat\ft(\g,\mu)$ established in \cite{CJKT-drin-pre},
we obtain that the classical limit of the twisted quantum affinization algebra $\qtar$ is isomorphic to $\U(\hat\ft(\g,\mu))$.

The root system of EALAs
have been axiomatized under the name of extended affine root systems (EARSs for short) and characterized  in \cite{AABGP}.
As in the classical Lie theory,
for any EARS $\Phi$ of nullity $2$,
there exist algebras in $\wh{E}_2$ with $\Phi$ as its root system \cite{ABP}.
Such an algebra is unique (up to isomorphism) except  $\Phi$ is of type $A_\ell^{(1,1)}$, in which case
there are infinitely many nonisomorphic algebras  in $\wh{E}_2$
that are parameterized by generic numbers (i.e., the algebras $\hat\fsl_{\ell+1}(\C_p)$).
By taking classical limit, we think this gives a natural
 explanation  why two-parameter quantum toroidal algebras should only exist in type $A$ case.
 On the other hand, this shows that it is reasonable to view quantum toroidal algebras as
 $\hbar$-deformation of  algebras in $\wh{E}_2$.

The layout of the paper is as follows. In Section \ref{sec:auto}, we introduce a class of diagram automorphisms on any
symmetrizable Kac-Moody algebra $\g$ which
satisfies two linking conditions. Starting with such a diagram automorphism $\mu$, in Section \ref{sec:tqaff} we define a $\mu$-twisted
quantum affinization algebra $\qtar$ of $\U_\hbar(\g)$, and in Section \ref{sec:tri} we prove a triangular decomposition of $\qtar$.
In Section \ref{sec:qptar}, under an additional linking condition on $\mu$, we give a simple characterization of the affine quantum Serre relations on restricted $\qtar$-modules.
As an application,  in Section \ref{sec:Hopf} we prove that there is a topological Hopf algebra structure on the restrict completion of $\qtar$.
In Section \ref{sec:specialization}, we study the classical limit of $\qtar$ and also link the algebra $\qtar$ with the quantization theory of nullity $2$ EALAs.
Finally,  Section \ref{sec:pf-prop-aff-q-serre} is devoted to a proof of  Theorem \ref{prop:Dr-to-normal-ordering} on general affine quantum Serre relations.

Throughout this paper, we denote the group of non-zero complex numbers, the set of non-zero integers, the set of positive integers and the set of non-negative integers  by $\C^\times$, $\Z^\times$, $\Z_+$ and $\N$, respectively.
For any $m\in \Z_+$, we set $\xi_m=\te{e}^{2\pi\sqrt{-1}/m}$ and $\Z_m=\Z/m\Z$.

\section{Automorphisms of generalized Cartan matrices}\label{sec:auto}
In this section, we introduce a class of automorphisms on
generalized Cartan matrices (GCMs)  satisfying two linking conditions.

\subsection{Automorphisms of generalized Cartan matrices}
Here we give some general backgrounds about  automorphisms on GCMs.
One may see \cite{Kac-book,KW} for details.

Throughout this paper, let $I$ be a finite subset of $\Z$ and let
 $A=(a_{ij})_{i,j\in I}$ be  a   symmetrizable GCM. Namely, $A$ is a square matrix such that
\[a_{ij}\in \Z,\ a_{ii}=2,\ i\neq j\Rightarrow a_{ij}\le 0\ \te{and}\ a_{ij}=0\Leftrightarrow a_{ji}=0\]
for $i,j\in I$, and that  there exists an invertible diagonal matrix
\begin{align}\label{ri}
D=\te{diag}\{r_i\}_{i\in I}\ (r_i\in \Z_+)\end{align}
such that $DA$ is symmetric.
Let $(\h,\Pi,\Pi^\vee)$ be a realization of the GCM $A$ (\cite{Kac-book}).
 Explicitly,  $\h$ is a $(2 |I|-\ell)$-dimensional  $\C$-space,
$\Pi=\{\al_i\}_{i\in I}$ is a set of linearly independent elements in the
dual space $\h^\ast$ of  $\h$, $\Pi^\vee=\{\al_i^\vee\}_{i\in I}$ is
a set of linearly independent elements in $\h$ and $\al_j(\al_i^\vee)=a_{ij}$ for $i,j\in I$, where $\ell$ is the rank of $A$.
Denote by $\g=\g(A)$ the Kac-Moody Lie algebra associated with the quadruple $(A,\h,\Pi,\Pi^\vee)$, and set $\g'=[\g,\g]$.
By definition, $\fg$ is the Lie algebra generated by the elements
$h\in \h, e_i^\pm, i\in I$, and subject to
the relations
\begin{align*}
[h,h']=0,\quad [e_i^+,e_i^-]=\al_i^\vee,\quad [h,e_i^\pm]=\pm\al_i(h)e_i^\pm,\quad \mathrm{ad}(e_i^\pm)^{1-a_{ij}}(e_j^\pm)=0,
\end{align*}
where $h,h'\in \h$ and $i,j\in I$ with $i\ne j$.

Let $\te{Aut}(A)$ be the group of  automorphisms on the GCM $A$. That is,  it is the group of permutations $\mu$ on the index set $I$ such that
 $a_{ij}=a_{\mu(i)\mu(j)}$ for $i,j\in I$.
Note that $\te{Aut}(A)$ acts naturally on the  root lattice $Q=\oplus_{i\in I}\Z\al_i$ of $\g$ so that
$\mu(\al_i)=\al_{\mu(i)}$ for $\mu\in \te{Aut}(A)$ and $i\in I$. Set
\[\h'=\h\cap \g'=\oplus_{i\in I}\C\al_i^\vee\quad\te{and}\quad \mathfrak c=\{h\in \h\mid \al_i(h)=0,i\in I\}\subset\h'.\]
Then we may view $\h/\mathfrak c$ as the dual space of $\C\ot_\Z Q$. Hence,  by duality, there is
an $\te{Aut}(A)$-action on it.
Moreover, as $\te{Aut}(A)$ is a finite group,
 there exists a subspace $\h''$ of $\h$  such that (\cite[Section 4.19]{KW})
\begin{align}\label{h''}\h=\h'\oplus \h''\quad \te{and}\quad \(\h''+\mathfrak c\)/\mathfrak c\ \te{is Aut}(A)\te{-stable}.\end{align}
From now on let us fix such a choice of the complementary space $\h''$ (which is not unique in general).
Then there is a unique action of $\te{Aut}(A)$ on $\h$ such that
\[\mu(\al_i^\vee)=\al_{\mu(i)}^\vee,\quad \mu(\h'')=\h'',\quad\te{and}\quad
\al_{\mu(i)}(\mu(h''))=\al_i(h'')\]
for $\mu\in \te{Aut}(A)$, $i\in I$ and $h''\in \h''$.
Furthermore, the action of $\te{Aut}(A)$ on $\h$
extends uniquely to an action of $\te{Aut}(A)$ on $\g$ by automorphisms such that
\[\mu(e^\pm_i)=e_{\mu(i)}^\pm\quad\te{for $\mu\in \te{Aut}(A)$ and $i\in I$.}\]
As in \cite{KW}, we regard  the elements of $\te{Aut}(A)$ as automorphisms on $\g$ in this way, and  call them diagram automorphisms.
It is straightforward to see that the complementary space $\h''$ for $\h'$ in $\h$ induces
a  non-degenerate symmetric bilinear form $\<\cdot|\cdot\>$ on $\h$ such that
\[r_i\<\al_i^\vee|h\>=\al_i(h)\quad \te{and}\quad \<\h''|\h''\>=0\quad \te{for $i\in I$ and $h\in \h$.}\]
Moreover,  the bilinear form $\<\cdot|\cdot\>$ on
$\h$ induces a unique invariant non-degenerate symmetric bilinear form, still
denoted as $\<\cdot|\cdot\>$, on $\g$ \cite{Kac-book}.
It is a simple fact   \cite{KW} that  the diagram automorphisms of $\g$ preserve the bilinear form $\<\cdot|\cdot\>$.

\subsection{Linking conditions}\label{subsec:link-cond}
In this subsection, we introduce a  class of automorphisms on the GCM $A$ which satisfy certain linking conditions.

Let $\mu\in \te{Aut}(A)$ be a given diagram automorphism  of order $N$. For each $i\in I$,
we write
\[\mathcal O(i)=\{\mu^k(i)\mid k=0,\dots,N-1\}\subset I\]
 for the orbit containing $i$, and 
$N_i$ the cardinality of $\mathcal{O}(i)$.
 Following \cite{FSS}, we  define 
 the first linking condition on $\mu$:
\vspace{1.5mm}\\
\te{(LC1)}\qquad $\sum_{p\in \mathcal O(i)}a_{pi}>0\quad\te{for all}\ i\in I.$
\vspace{1mm}

The condition (LC1) can be reformulated as follows: 

\begin{lem}\label{lem:linking}\cite[Sect.\,2.2]{FSS} The automorphism $\mu$ on $A$ satisfies the condition (LC1) if and only if for every $i\in I$,
the Dynkin subdiagram of  $\mathcal O(i)$ is either

\te{(i)} a direct sum of type $A_1$, or

\te{(ii)} a direct sum of type $A_2$ with $a_{\mu^{N_i/2}(i),i}=-1$.
\end{lem}

For $i,j\in I$,  set
\begin{align}\label{caij}\check{a}_{ij}=s_i\sum_{p\in \mathcal O(i)}a_{pj}\in s_i\Z,\end{align}
where
\begin{equation}\label{si} s_i=\begin{cases} 1,\ \te{ if (i) holds in Lemma \ref{lem:linking}};\\
2,\ \te{ if (ii) holds in Lemma \ref{lem:linking}.}\end{cases}
\end{equation}
For convenience, we  fix a set of representatives for the orbits of $\mu$:
 \[\check{I}=\{ i\in I\,|\,\mu^k(i)\geq i\,\text{ for }k=0,\dots,N-1 \}.\]
It was known
\cite[Sect. 2.2]{FSS} that the folded matrix
\begin{align}
\check{A}=(\check{a}_{ij})_{i,j\in \check{I}}
\end{align}
of $A$ associated with $\mu$ is  a symmetrizable GCM. Moreover,
 $\check{A}$ is of finite (resp. affine;
 resp. indefinite) type
if and only if $A$ is of finite (resp. affine;
 resp. indefinite) type.

Denote by $\check{\h}$  the subspace of $\h$  fixed by the isometry $\mu$.
For  $i\in I$, set
\begin{align}
\check{\al}_i=\frac1{N_i}\sum_{p\in \mathcal O(i)}\al_{p}\quad \te{and}\quad \check{\al}_i^\vee=s_i \sum_{p\in \mathcal O(i)}\al_{p}^\vee.\end{align}
Then we have
$\h=\chh\oplus (1-\mu)\h,$ $\check{\al}_i|_{\chh}=\al_i|_{\chh}$ and $\check{\al}_i|_{(1-\mu)\h}=0$
 for $i\in I$.
This gives that $\check{\al}_i$'s can be canonically identified with $\al_i|_{\chh}$.
Thus,   the triple
\[(\check{\h},\ \check{\Pi}=\{\check{\al}_i\}_{i\in \check{I}},\ \check{\Pi}^\vee=\{\check{\al}_i^\vee\}_{i\in \check{I}})\] is
 a realization of the folded GCM $\check{A}$.
 Moreover, set
 \begin{align}\label{cri}
 \check{r}_i=\frac{N}{s_iN_i}r_i,\ i\in I\quad\te{and}\quad \check{D}=\te{diag}\{\check{r}_i\}_{i\in \check{I}}.
 \end{align}
Then the matrix
$\check{D}\check{A}$  is symmetric. Denote by $\check{\g}=\g(\check{A})$
the Kac-Moody Lie algebra associated with the quadruple
$(\check{A},\check{\h},\check{\Pi},\check{\Pi}^\vee)$, which is called the orbit Lie algebra of $\g$ associated with $\mu$ (\cite{FSS}).

Note that the condition (LC1) on $\mu$  controls 
the edges joining the vertices in a same  orbit
so that the correspond folded matrix is again a GCM.
We now introduce two other linking conditions on $\mu$ which control the edges joining the vertices in different orbits.
Set
\begin{align}
\mathbb I=\{(i,j)\in I\times I\mid i\notin \mathcal O(j)\ \te{and}\ a_{ij}<0\}.
\end{align}
and for $i,j\in I$, set
\begin{align}\label{gammaij}
\Gamma_{ij}=\{k\in \Z_N\mid a_{i\mu^k(j)}\ne 0\}\quad\te{and}\quad
\Gamma_{ij}^*=\{k\in \Z_N\mid a_{i\mu^k(j)}=a_{ij}\}.\end{align}
The other linking conditions on $\mu$ used in this paper are as follows:

\vspace{1.5mm}
 \te{(LC2)}\quad  $\Gamma_{ij}=\Gamma_{ij}^*$\quad for all $(i,j)\in \mathbb I$.
\vspace{1mm}

\vspace{1.5mm}
\te{(LC3)}\quad $\Gamma_{ij}$ is a subgroup of $\Z_N$\ for all $(i,j)\in \mathbb I$.
\vspace{1mm}

We say that an automorphism of $A$ is  transitive if it acts transitively on the set $I$.
Note that when $A$ is of affine type,  an automorphism of $A$ is transitive if and only if
$A$ is of type $A_\ell^{(1)}$
and it has order $\ell+1$ by rotating the Dynkin diagram.
Here and henceforth, when $A$ is of finite type or affine type,
 we will label  $A$ (or $\g$)   using the Tables  Fin and Aff 1-3 of \cite[Chap 4]{Kac-book}.
For a list of  automorphisms on affine GCMs, one may see \cite[Tables 2, 3]{ABP}. 
We have:
\begin{lem}\label{lem:LC1aff} Assume that the GCM $A$ is of finite type or affine type. Then
\begin{enumerate}
\item an automorphism $\mu$ of $A$  does not satisfy the condition (LC1) if and only if $A$ is of affine type $A_\ell^{(1)}$ and
$\mu$ is transitive;
\item  all the automorphisms of $A$ satisfy the condition (LC2);
\item an automorphism $\mu$ of $A$ does not satisfy the condition (LC3) if and only if $A$ is of affine type $A_{\ell}^{(1)}$ with $\ell=2k+1$ for some $k\ge 2$
and $\mu$ is of order $k+1$.
\end{enumerate}
\end{lem}
\begin{proof}
The first assertion was proved in \cite{FSS}. For the second one, it suffices to
check those nontrivial automorphisms on non-simply-laced affine GCMs. Namely,
$A$ is of type $B_\ell^{(1)} (\ell\ge 3), C_\ell^{(1)} (\ell\ge 2), A_{2\ell-1}^{(2)} (\ell\ge 3)$ or $D_{\ell+1}^{(2)} (\ell\ge 2)$ and $\mu$ has order $2$.
One can check directly that in each case $\mu$ satisfies the condition (LC2).

Now we turn to prove the third assertion.
Note that if $\mu$ has order $1$ or $2$, then it automatically  satisfies the condition (LC3).
So we only need to consider the following cases:
\begin{enumerate}
\item $A$ has type $D_4$ or $D_4^{(1)}$ and $\mu$ has order $3$;
\item $A$ has type $D_\ell^{(1)}$ $(\ell\ge 4)$ and $\mu$ has order $4$;
\item $A$ is of type $A_\ell^{(1)}$ $(\ell\ge 2)$ and $\mu$ is a rotation of Dynkin diagram by $a$ nodes for some $1\le a\le \lfloor \frac{\ell+1}{2}\rfloor$.
\end{enumerate}
For the first two cases, it is obvious that $\mu$ satisfies the condition (LC3).
Assume now that $\mu$ is as in Case (3). Set $I=\{0,1,\dots,\ell\}$ and
$b=\mathrm{gcd}(\ell+1,a)$.
Then $\mu=(0,1,\dots,\ell)^a$ has order $\frac{\ell+1}{b}$, and there are exactly $b$ $\mu$-orbits
\[\{i+bm\mid 0\le m\le \frac{\ell+1}{b}-1\},\quad i=0,1,\dots,b-1\]
in $I$.
If $\mu$ does not satisfy (LC3), then $b\ge 2$, $\frac{\ell+1}{b}\ge 3$, and
there exists $i\in I$ such that the vertices adjacent to $i$ lie in a common $\mu$-orbit.
This implies that $b=2$ and $\ell=2k+1$ for some $k\ge 2$. Conversely, if $b=2$ and $\ell=2k+1$ for some $k\ge 2$,
then one verifies that $\mu$ does not satisfy (LC3), as required.
\end{proof}

\section{Twisted quantum affinizations}\label{sec:tqaff}
In the rest of the paper, we will always assume that $\mu$ is an automorphism of
the symmetrizable GCM $A$, which has order $N$ and
satisfies the linking conditions (LC1) and (LC2).
In Sections 5, 6 and 8, we will further assume that the automorphism $\mu$ satisfies the linking condition (LC3).

In this section,  we introduce a notion of $\mu$-twisted quantum affinization algebras.

\subsection{Twisted quantum affinizations}\label{subsec:q-KM-alg}
We start with some conventions. In this paper, by a $\C[[\hbar]]$-algebra ,
we mean a topological algebra  over $\C[[\hbar]]$, equipped
with its canonical $\hbar$-adic topology.
For two $\C[[\hbar]]$-modules $V$ and $W$, we denote by $V\wh{\ot} W$ the $\hbar$-adically
completed tensor product of $V$ and $W$.
For any invertible element $v$  in $\C[[\hbar]]$ and  $n, k, s\in\Z$ with $0\leq k\leq s$,
we define the usual quantum numbers as follows
\begin{eqnarray*}
[n]_v=\frac{v^n-v^{-n}}{v-v\inverse},\quad
[s]_v!=[s]_v[s-1]_v\cdots [1]_v,\quad\te{and}\quad
\binom{s}{k}_v=\frac{[s]_v!}{[s-k]_v![k]_v!}.
\end{eqnarray*}
Throughout this paper, we set (see \eqref{ri} and \eqref{cri})
\begin{align}\label{eq:defqi}
q=e^{\hbar},\quad q_i=q^{r_i},\quad d_i=N/N_i,\quad\te{and}\quad \check{q}_i=q^{\check{r}_i}
=q_i^{d_i/s_i}\quad \te{for}\quad i\in I.
\end{align}

The following notion was introduced independently by Drinfeld  and Jimbo (cf. \cite{Kassel-topologically-free}).
\begin{de}
The quantum Kac-Moody algebra  $\U_\hbar(\g)$    is
the  $\C[[\hbar]]$-algebra topologically generated by the elements
$h\in \h, e_i^\pm, i\in I$, and subject to the relations $(h,h'\in \h, i,j\in I)$
\begin{eqnarray}
&&[h,h']=0,\quad [h,e_i^\pm]=\pm\al_i(h) e_i^\pm,\quad [e_i^+,e_j^-]=\delta_{i,j}
\frac{q_i^{\alck_i}-q_i^{-\alck_i}}{q_i-q_i^{-1}},\\ \label{quan-S-r}
&&\sum_{r=0}^{1-a_{ij}}(-1)^r\binom{1-a_{ij}}{r}_{q_i}(e_i^\pm)^r e_j^\pm (e_i^\pm)^{1-a_{ij}-r}=0,\quad \te{if}\ i\ne j.
\end{eqnarray}
\end{de}

For  $i,j\in I$ with $\check{a}_{ij}<0$,
denote by $\<\Gamma_{ij}^{\mathrm{d}}\>$ the subgroup of $\Z_N$ generated by the set
 \begin{align}\label{eq:gammaij-}
 \Gamma_{ij}^{\mathrm{d}}=\{k-l\mid k,l\in \Gamma_{ij}\}.\end{align}
Let us introduce the integers
\begin{align}\label{dij}d_{ij}=\mathrm{Card}\ \<\Gamma_{ij}^{\mathrm{d}}\>,\quad  \bar{a}_{ij}=\min\{a_{i\mu^k(j)}\mid k\in \Z_N\}, \quad\te{and}\quad m_{ij}=1-\bar{a}_{ij},\end{align}
and the sets
 \[\Gamma_{ij}^{\mathrm{e}}=\{k+l\mid k\in \<\Gamma_{ij}^{\mathrm{d}}\>, l\in \Gamma_{ij}\}\quad\te{and}\quad \overline{\Gamma}_{ij}=\Gamma_{ij}^{\mathrm{e}}\setminus \Gamma_{ij}.\]
Note that both $d_i$ and $d_j$ divide $d_{ij}$, and $\overline{\Gamma}_{ij}=\emptyset$ if $\mu$ satisfies
the condition  (LC3).

For $i,j\in I$, we introduce the ($\C[[\hbar]]$-valued)  polynomials:
\begin{align} \label{e:F}
&F^\pm_{ij}(z,w)=\prod_{k\in \Gamma_{ij}}\left( z-\xi^{k}q_i^{\pm a_{i\mu^k(j)}}w \right),\\ \label{e:G}
&G^\pm_{ij}(z,w)=\prod_{k\in \Gamma_{ij}}\left( q_i^{\pm a_{i\mu^k(j)}}z-\xi^{k}w \right),\\
\label{e:barF}
&\bar{F}^\pm_{ij}(z,w)=\prod_{k\in \overline{\Gamma}_{ij}}\left( z-\xi^{k}q_i^{\pm \bar{a}_{ij}}w \right),\quad\te{if}\ \check{a}_{ij}<0,\\ \label{e:barG}
&\bar{G}^\pm_{ij}(z,w)=\prod_{k\in \overline{\Gamma}_{ij}}\left( q_i^{\pm \bar{a}_{ij}}z-\xi^{k}w \right),\quad\te{if}\ \check{a}_{ij}<0,\\\label{pi}
&p_i^\pm(z_1,z_2,z_3)=q_i^{\mp \frac{3}{2}d_{i}}z_1^{d_{i}}
        -(q_i^{\frac{d_{i}}{2}}+q_i^{-\frac{d_{i}}{2}})
            z_2^{d_{i}}
        +q_i^{\pm \frac{3}{2}d_{i}}
            z_3^{d_{i}},\quad\te{if}\ s_i=2,\\
            \label{pij}
&p_{ij}^\pm(z,w)=
    \left(
        z^{d_i}+q_i^{\mp d_i}w^{d_i}
    \right)^{s_i-1}
    \frac{
        q_i^{\pm 2d_{ij}}z^{d_{ij}}-w^{d_{ij}}
    }{
        q_i^{\pm 2d_i}z^{d_{i}}-w^{d_i}
    },\quad\te{if}\ \check{a}_{ij}<0,
\end{align}
and the formal series
\begin{align}
\label{e:g}
g_{ij}(z)
=\prod_{k\in \Z_N}\frac{ q_i^{a_{i\mu^k(j)}}-\xi^k z}{1-\xi^kq_i^{a_{i\mu^k(j)}}z},
\end{align}
which is expanded for $|z|<1$,
where
\begin{align}
\xi=\xi_N=e^{2\pi\sqrt{-1}/N}.\end{align}

For $i,j\in I$ with $\check{a}_{ij}<0$ and $r=0,1,\dots,m_{ij}$, we also introduce the polynomial
\begin{align}
\label{pij-total} p_{ij,r}^\pm(z_1,\dots,z_{m_{ij}},w)=\prod_{1\le a<b\le m_{ij}}p_{ij}^\pm(z_a,z_b)
    \prod_{1\le a\le r}\bar{G}_{ij}^\pm(z_a,w)
    \prod_{r<b\le m_{ij}}\bar{F}_{ij}^\pm(z_b,w).
\end{align}
Note that if $(i,j)\in \mathbb{I}$, then
\begin{align}
\bar{F}_{ij}^\pm(z,w)=\frac{z^{d_{ij}}-q_i^{\pm d_{ij} a_{ij}}w^{d_{ij}}}{F_{ij}^{\pm}(z,w)}
\quad\te{and}\quad \bar{G}_{ij}^\pm(z,w)=\frac{q_i^{\pm d_{ij} a_{ij}}z^{d_{ij}}-w^{d_{ij}}}{G_{ij}^{\pm}(z,w)}.
\end{align}


 Now we introduce the quantum algebras concerned about in this paper.

\begin{de}\label{de:tqaffine}
The $\mu$-twisted quantum affinization $\qtar$ of  $\U_\hbar(\g)$
is   the $\C[[\hbar]]$-algebra topologically
generated by the set
\begin{eqnarray}\label{eq:tqagenerators}
\set{h,\ h_{i,m},\  x^\pm_{i,n},\ c,\ d}
{
   h\in \check{\h}, i\in I, m\in\Z^\times,n\in\Z
},
\end{eqnarray}
and subject to the relations in  terms of the generating functions
 \begin{eqnarray*}
 \phi_i^\pm(z)=q^{\pm h_{i,0}}\, \te{exp}
    \left(
        \pm (q-q\inverse)\sum\limits_{\pm m> 0}h_{i,m}z^{-m}
    \right),\quad x^\pm_i(z)=\sum\limits_{m\in \Z} x^\pm_{i,m}z^{-m},
\end{eqnarray*}
where $h_{i,0}=r_i\sum_{k\in \Z_N}\al_{\mu^k(i)}^\vee\in \check{\h}$.
The relations are ($i,j\in I$, $h,h'\in \chh$):
\begin{align*}
&\te{(Q0)  }&& \phi^\pm_{\mu(i)}(z)=\phi^\pm_i(\xi\inverse z),\quad
x^\pm_{\mu(i)}(z)=x^\pm_i(\xi\inverse z),\\
    &\te{(Q1)  }&&[d,h]=0=[d,c],\ q^d \phi_i^\pm(z) q^{-d}=\phi_i^\pm(q^{-1}z),\\
&\te{(Q2)  }&&[h,h']=0=[c,h]=[c,\phi_i^\pm(z)]=[\phi_i^\pm(z),\phi_j^\pm(w)]=[h,\phi^\pm_i(z)],\\
&\te{(Q3) }&& \phi^+_i(z)\phi^-_j(w)=\phi^-_j(w)\phi^+_i(z)
    g_{ij}(q^c w/z)\inverse g_{ij}(q^{-c}w/z),\\
    &\te{(Q4)  }&&[h, x_i^\pm(z)]=\pm \al_{i}(h)x_i^\pm(z),\
    q^d x^\pm_i(z) q^{-d}=x^\pm_{i}(q^{-1}z),\ [c,x_i^\pm(z)]=0,\\
&\te{(Q5) }&&\phi^+_i(z)x^\pm_j(w)=x^\pm_j(w)\phi^+_i(z)
    g_{ij}(q^{\mp \half c}w/z)^{\pm 1},\\
&\te{(Q6) }&& \phi^-_i(z)x^\pm_j(w)=x^\pm_j(w)\phi^-_i(z)
    g_{ji}(q^{\mp \half c}z/w)^{\mp 1},\\
&\te{(Q7)  }&&[x_i^+(z),x_j^-(w)]
=\frac{1}{q_i-q_i\inverse}
    \ksum\delta_{i,\mu^k(j)}\\
 &&&  \quad\times \Bigg(
        \phi_i^+(z q^{-\half c})\delta\left(
            \frac{\xi^kw q^c}{z}
        \right)
        -
        \phi_i^-(z q^{\half c})\delta\left(
            \frac{\xi^kw q^{-c} }{z}
        \right)
    \Bigg),\\
&\te{(Q8)  }&&F^\pm_{ij}(z,w)x^\pm_i(z)x^\pm_j(w)=
    G^\pm_{ ij}(z,w)x^\pm_j(w)x^\pm_i(z),\\
    &\te{(Q9) }&&\sum_{\sigma\in S_{3}}
    p_i^\pm(z_{\sigma(1)},z_{\sigma(2)},z_{\sigma(3)})\,
                x_i^\pm(z_{\sigma(1)})x_i^\pm(z_{\sigma(2)})x_i^\pm(z_{\sigma(3)})=0,
\quad
    \te{if}\ \ s_i=2,\\
&\te{(Q10)  }&&\sum_{\sigma\in S_{{m}_{ij}}}\sum_{r=0}^{{m}_{ij}}
    (-1)^r \binom{{m}_{ij}}{r}_{q_i^{d_{ij}}}
    p_{ij,r}^\pm(z_{\sigma(1)},\dots,z_{\sigma(m_{ij})},w)
    x_i^\pm(z_{\sigma(1)})\cdots x_i^\pm(z_{\sigma(r)})\\
 &\qquad&&\quad\qquad   \cdot x_j^\pm(w)
       x_i^\pm(z_{\sigma(r+1)})\cdots x_i^\pm(z_{\sigma({m}_{ij})})\ =0,\quad
    \te{if}\ \ \check{a}_{ij}<0,
\end{align*}
where
$\delta(z)=\sum_{n\in\Z}z^n$ is the usual delta function.
\end{de}

It is straightforward to see that the relations  (Q1)-(Q10) are compatible with
(Q0) and so the $\C[[\hbar]]$-algebra $\qtar$  is well-defined.
When $\g$ is of finite type, the algebra $\qtar$ was first introduced by Drinfeld \cite{Dr-new} for the purpose of giving
a current presentation of quantum affine algebras:

\begin{thm}\label{thm:tqaa}[\cite{Dr-new,beck,Da1,Da2,JZ1}] Assume that $\g$ is of finite type  $X_\ell$.
Then $\qtar$ is isomorphic to
the quantum affine algebra of type $X_\ell^{(N)}$
\end{thm}

When $\mu=\mathrm{id}$, $\U_\hbar(\hat\g)=\qtar$ was called the   quantum affinization of
$\U_\hbar(\g)$ (\cite{J-KM,Naka-quiver,He-total}).
And when $\g$ is of simply-laced type,
$\qtar$ was realized in \cite{CJKT-twisted-quantum-aff-vr} in terms of  vertex operators.

\begin{rem}{\em When $\g$ is of non-simply-laced type and $\mu\ne \mathrm{id}$, the affine quantum Serre relations (Q10) are new.
According to the structure theory of $\qtar$ developed later, it seems that our generalization is natural.}\end{rem}

Alternatively we can write the defining relations of $\qtar$
in terms of the generators given in  \eqref{eq:tqagenerators}.
For the defining relations (Q0)-(Q7), we have
($i,j\in I$, $h,h'\in \check\h$, $m, n\in \Z$):
\begin{eqnarray*}
&\te{(Q0$'$)} &h_{\mu(i),m}=\xi^{m}h_{i,m},\ x^\pm_{\mu(i),n}=\xi^{n}x_{i,n}^\pm,\\
&\te{(Q1$'$)} &[d,h]=0=[d,c],\ [d,h_{i,m}]=m h_{i,m},\ \\
&\te{(Q2$'$)} &[h,h']=0=[c,h]=[c,h_{i,\pm m}]=[h,h_{i,\pm m}]=[h_{i,\pm m}, h_{j,\pm n}],\ m,n\ge 0,\\
&\te{(Q3$'$)} &[h, x_{i,n}^\pm]=\pm\al_{i}(h) x_{i,n}^\pm,\ [d,x^\pm_{i,n}]=n x_{i,n}^\pm,\
[c,x^\pm_{i,n}]=0,\\
&\te{(Q4$'$)} &[h_{i,m},h_{j,-n}]=\delta_{m,n}\frac{1}{m}
    \ksum \xi^{mk}[m r_i a_{i\mu^k(j)}]_{q}\frac{q^{mc}-q^{-mc}}{q-q^{-1}},\  m,n>0,\\
&\te{(Q5$'$)} &[h_{i,m},x_{j,n}^\pm]
=\pm \frac{1}{m}\ksum \xi^{mk}[m r_i a_{i\mu^k(j)}]_{q} q^{\mp\half mc} x_{j,m+n}^\pm,\ m>0,\\
&\te{(Q6$'$)}&[h_{i,m},x_{j,n}^\pm]
=\pm \frac{1}{m}\ksum \xi^{mk}[m r_i a_{i\mu^k(j)}]_{q} q^{\pm\half mc} x_{j,m+n}^\pm,\ m<0,\\
&\te{(Q7$'$)}&[x_{i,m}^+, x_{j,n}^-]= \sum_{k\in \Z_N}\delta_{i,\mu^k(j)} \xi^{-nk}
\frac{\phi_{j,m+n}^+ q^{\frac{m-n}{2}c}-\phi_{j,m+n}^- q^{\frac{n-m}{2}c}}{q_i-q_i^{-1}},
\end{eqnarray*}
where the elements $\phi_{j,m}^+$ and $\phi_{j,-m}^-$ ($m\ge 0$) are defined respectively by
\[\phi^+_j(z)=\sum_{m\ge 0}\phi_{j,m}^+z^{-m}\quad\te{and}\quad \phi^-_j(z)=\sum_{m\ge 0}\phi_{j,-m}^-z^{m}.\]

For the relations (Q8), note that for $i\in I$ we have
\begin{eqnarray}\label{Fiiep}
F_{ii}^\pm(z,w)=(z^{d_i}+q_i^{\mp d_i}w^{d_i})^{s_i-1}(z^{d_i}-q_i^{\pm 2 d_i}w^{d_i}),\\
G_{ii}^\pm(z,w)=(q_i^{\mp d_i}z^{d_i}+w^{d_i})^{s_i-1}(q_i^{\pm 2d_i}z^{d_i}-w^{d_i}).\label{Giiep}
\end{eqnarray}
Then for the case that $i=j$,
  the relations (Q8) are equivalent to the following relations:
\begin{align}
&\sum_{\sigma\in S_2}[x_{i,m_{\sigma(1)}+d_i}^\pm,x^\pm_{i,m_{\sigma(2)}}]_{q_i^{\pm 2d_i}}=0, \quad  \te{if}\ s_i=1,\\
&\sum_{\sigma\in S_2}
\{[x_{i,m_{\sigma(1)}+2d_i},x_{i,m_{\sigma(2)}}^\pm]_{q_i^{\pm d_i}}-q_i^{\pm 2d_i}[x_{i,m_{\sigma(1)}+d_i}^\pm, x_{i,m_{\sigma(2)}+d_i}^\pm]_{q_i^{\mp 3d_i}}\}
\\\notag&\qquad\qquad
=0,\quad
\te{if}\ s_i=2,
\end{align}
where
$[a,b]_v=ab-vba$  for   $0\ne v\in \C[[\hbar]]$ and $a,b\in \qtar$.
Similarly, the relations (Q9) are equivalent to  (cf. \cite{Da1})
\begin{eqnarray*}
\te{(Q9$'$) }\qquad \sum_{\sigma\in S_3} [[x^\pm_{i,m_{\sigma(1)}\pm d_i},x_{i,m_{\sigma(2)}}^\pm]_{q_i^{d_i}}, x^\pm_{i,m_{\sigma(3)}}]_{q_i^{2d_i}}=0.
\end{eqnarray*}
Finally, for the case that $i,j\in I$ with $\check{a}_{ij}<0$,
since the  relations (Q8) and (Q10) depend on  the
polynomials $F_{ij}^{\pm}(z,w), G_{ij}^{\pm}(z,w)$ and $p^\pm_{ij,r}(z_1,\dots,z_{m_{ij}},w)$, one needs a case by case argument.
For some special cases, one may see \cite{Da1}.

\begin{rem}\label{rem:q11}{\em One of the main features in the definition of $\qtar$ is the existence of the Drinfeld polynomials
$p_{ij}^\pm(z,w)$ in the affine quantum
Serre relations (Q10).
When $\g$ is of finite type, it was proved in \cite{Da1} that the relations (Q10) in $\qtar$ are equivalent to
the following affine quantum Serre relations (without polynomials):
\begin{align*}
  \te{(Q11) }\quad&\sum_{\sigma\in S_{\check{m}_{ij}}}
   \sum_{r=0}^{\check{m}_{ij}}
    (-1)^r \binom{\check{m}_{ij}}{r}_{\check{q}_i}x_i^\pm(z_{\sigma(1)})\cdots x_i^\pm(z_{\sigma(r)})\\
 &\quad   \cdot x_j^\pm(w)
       x_i^\pm(z_{\sigma(r+1)})\cdots x_i^\pm(z_{\sigma(\check{m}_{ij})})
\,=0,\quad\te{if }\ \check a_{ij}<0,
\end{align*}
where $\check{m}_{ij}=1-\check{a}_{ij}$.
Thus, for any $\mu$-invariant subset $J$ of $I$ such that the GCM $(a_{ij})_{i,j\in J}$ is a direct sum  of
GCMs of
finite type, the relations (Q11) hold in $\qtar$ for $i,j\in J$.
We conjecture that the relations (Q11) hold in $\qtar$ for the general $\g$ and $\mu$.
}
 \end{rem}

\subsection{Twisted quantum toroidal algebras}

As in the untwisted case, we define the horizontal subalgebra $\U^h$ of $\qtar$ to be the closed subalgebra
 generated by $h, x_{i,0}^\pm$ for $h\in \h$ and $i\in I$.
When $\g$ is of affine type and $\mu$ fixes
the additional node of $\g$ (\cite{Kac-book}),
we further define the vertical subalgebra
$\U^v$ of $\qtar$ to be the closed subalgebra generated by
$h_{i,n}, x_{i,n}^\pm, c, d$ for $n\in \Z$  and $i\in I$ not equal to the additional node of $\g$.

It was known that \cite{FSS} the orbit Lie algebra $\chg$ can be realized as the subalgebra of  $\g$ generated by the elements
$h, \check{e}^\pm=\sum_{p\in \Z_N}e_{\mu^p(i)}^\pm,$ for $h\in \check{\h}$ and  $i\in \check{I}$.
However, in the quantum case,  $\U_\hbar(\chg)$ is not simply a subalgebra of $\U_\hbar(\g)$.
One may see \cite[Section 2.6]{H-tw} for details.
We expect that there is an algebra morphism from $\U_\hbar(\chg)$ to
$\qtar$.
Explicitly, noting that $\phi_{i,0}^\pm = q^{\pm h_{i,0}}$ and so from (Q7$'$) we have
    \begin{align}\label{commx0x0}
    [x_{i,0}^+, x_{j,0}^-]=\sum_{k\in \Z} \delta_{i,\mu^k(j)} \frac{q^{h_{i,0}}-q^{-h_{i,0}}}{q_i-q_i^{-1}}
    =\ksum \delta_{i,\mu^k(j)} \frac{\check{q}_i^{\check{\al}_i^\vee}-\check{q}_i^{-\check{\al}_i^\vee}}{q_i-q_i^{-1}}.
    \end{align}
   Assume now that the relations (Q11) hold on in $\qtar$. Then we have
\[\sum_{r=0}^{1-\check{a}_{ij}}(-1)^r\binom{1-\check{a}_{ij}}{r}_{\check{q}_i}(x_{i,0}^\pm)^r x_{j,0}^\pm (x_{i,0}^\pm)^{1-\check{a}_{ij}-r}=0,
\quad \te{for}\ i\ne j\in \check{I}.\] This together with (Q3$'$) and \eqref{commx0x0} gives that
there is a surjective algebra morphism
$\U_\hbar(\check{\g})\rightarrow \U^h$ defined by
\begin{align}\label{checkgmor}
h\mapsto h,\quad e_i^\pm\mapsto x_{i,0}^\pm/\sqrt{ d_i\left[\frac{d_i}{s_i} \right]_{q_i}}\quad\te{for}\quad h\in\check\h,\ i\in \check I.
\end{align} In particular, we have (see Remark \ref{rem:q11} and  Theorem \ref{thm:tqaa}):

\begin{prop} Assume that $\g$ is of untwisted affine type $X_\ell^{(1)}$ and $\mu$ fixes
the additional node of $\g$. Then $\check{A}$ is of affine type $X_\ell^{(N)}$ and
 both $\U^h$ and $\U^v$ are isomorphic to a quotient of the twisted quantum affine algebra $\U_\hbar(\check{\g})$.
\end{prop}

\section{Triangular decomposition}\label{sec:tri}
In this section,  we prove a triangular decomposition of $\qtar$.
\subsection{Triangular decomposition of $\qtar$}

\begin{de}\label{de:trian-top} Let $A$ be a completed and separated   $\C[[\hbar]]$-algebra. By a triangular decomposition of $A$, we
mean a data of three closed $\C[[\hbar]]$-subalgebras $(A^-,H,A^+)$ of $A$
such that the multiplication $x^-\otimes h\otimes x^+\mapsto x^-hx^+$ induces an isomorphism from the $\hbar$-adically completed
tensor product $\C[[\hbar]]$-module
$A^-\wh\otimes H\wh\otimes A^+$ to $A$.
\end{de}

Let $\qtar^+$ (resp. $\qtar^-$; resp. $\qhei$) be the closed subalgebra of $\qtar$ generated by  $x_{i,n}^+$
(resp. $x_{i,n}^-$; resp.
$h,h_{i,m},c,d$).
The following is the main result of this section.

\begin{thm}\label{thm:tri-decomp}
$(\qtar^-,\qhei,\qtar^+)$ is a triangular decomposition of $\qtar$.
Moreover, $\qtar^+$ (resp. $\qtar^-$; resp. $\qhei$) is isomorphic to the   $\C[[\hbar]]$-algebra topologically
generated by $x_{i,n}^+$
(resp. $x_{i,n}^-$, resp.
$h,h_{i,m},c,d$), and subject to the relations
\te{(Q0), (Q8)-(Q10)} with ``$+$'' (resp. \te{(Q0), (Q8)-(Q10)} with ``$-$''; resp. \te{(Q0)-(Q2)}).
\end{thm}

The rest of this section is devoted to proving
Theorem \ref{thm:tri-decomp}.
When $\mu=\mathrm{id}$, Theorem \ref{thm:tri-decomp} was proved in \cite{beck} for the finite type and in \cite{He-representation-coprod-proof} for the general case.
In addition, when $\g$ is of finite type, Theorem \ref{thm:tri-decomp} was proved in \cite{Da2}.
We adopt a similar method of
\cite{He-representation-coprod-proof} to show the theorem.
To prove Theorem \ref{thm:tri-decomp} and for later use,  in the following definition we collect some other $\C[[\hbar]]$-algebras related to $\qtar$.

\begin{de}\label{de:subalgebras} We denote by $\uf$  the $\C[[\hbar]]$-algebra topologically generated by the elements in \eqref{eq:tqagenerators}
with relations (Q0)-(Q7), denote by
 $\ul$  the quotient algebra of $\uf$ modulo its
closed ideal generated by relations (Q8), and denote by
$\us$  the quotient algebra of $\ul$ modulo its
closed ideal generated by relations (Q9).
Let $\uf^+$ (resp. $\uf^-$;  resp. $\hf$) be
the closed subalgebra of $\uf$ generated by
$x_{i,n}^+$ (resp.  $x_{i,n}^-$, resp.  $h,h_{i,m},c,d$).
Similarly, we have the closed subalgebras $\ul^\pm$, $\hl$ and $\us^\pm$, $\U_\hbar^{s}(\hat\h_\mu)$ of
$\ul$ and $\us$, respectively.
\end{de}

Before proving Theorem \ref{thm:tri-decomp}, we mention one of its consequences. 
For  $i\in I$, define $\U_i$ to be the closed subalgebra of
$\qtar$ generated by $h_{i,n}$, $x_{i,n}^\pm$, $c, d$ for $n\in d_i\Z$.
Denote by $\ft$ the simple finite dimensional Lie algebra of type $A_{s_i}$ and denote by
$\theta$ the diagram automorphism of $\ft$ with order $s_i$.
Then there is a $\C$-algebra morphism from $\U_{\hbar}(\hat{\ft}_\theta)$ to $\U_i$ given by
\begin{align}\label{uiiso}
h_{1,0}\mapsto \frac{h_{i,0}}{r_id_i},\
h_{1,m}\mapsto \frac{h_{i,d_im}}{[r_id_i]_q},\ x_{1,n}^\pm\mapsto \frac{x_{i,d_in}^\pm}{\sqrt{d_i[d_i]_{q_i}}} ,\
  c\mapsto \frac{c}{r_i},\ d\mapsto \frac{d}{d_i},\ \hbar\rightarrow d_ir_i \hbar
\end{align}
where $0\ne m\in\Z,n\in\Z$.
Just as untwisted quantum affinization algebras are glued by copies of quantum affine algebras of type $A_1^{(1)}$ (see (\cite[Corollary 3.3]{He-representation-coprod-proof})),
from Theorem \ref{thm:tri-decomp} we have

\begin{coro}\label{a11a22}
For $i\in I$,  $\U_i$ is isomorphic to the quantum affine algebra of type $A_{s_i}^{(s_i)}$ as $\C$-algebras with the isomorphism given by \eqref{uiiso}.
\end{coro}

\subsection{Technical  lemmas}\label{subsec:tri-tech-lemms}
In this subsection we  prove two
lemmas.

\begin{lem}\label{lem:tri-t-1}
Let $i\in I$ with $s_i=2$, $l=1,2$ or $3$ and $\eta=+$ or $-$. Then the following hold in  $\ul$:
\begin{align}
  \sum_{\sigma\in S_3} p_i^\pm(z_{\sigma(1)},z_{\sigma(2)},z_{\sigma(3)})
  \xi^{\pm,\eta}_{i,l}(z_{\sigma(1)})\xi^{\pm,\eta}_{i,l}(z_{\sigma(2)})\xi^{\pm,\eta}_{i,l}(z_{\sigma(3)})=0,\label{eq:tri-decomp-tech-3}
\end{align}
where for $k=1,2,3$,
\begin{equation*}
\xi_{i,l}^{\pm,\eta}(z_k)=\begin{cases}x_i^\pm(z_k)\ &\te{if}\ k\ne l;\\
\phi_i^\eta( q^{\mp\eta\frac{1}{2}c}z_k)\ &\te{if}\ k=l.
\end{cases}
\end{equation*}
\end{lem}

\begin{proof} We denote  by $T_{i,l}^{\pm,\eta}(z_1,z_2,z_3)$ the LHS of \eqref{eq:tri-decomp-tech-3}.
It is clear that
\begin{align*}
T_{i,3}^{\pm,\eta}(z_1,z_2,z_3)=T_{i,2}^{\pm,\eta}(z_1,z_3,z_2)=T_{i,1}^{\pm,\eta}(z_3,z_2,z_1).
\end{align*}
Thus it suffices to check the case  $l=3$.
Assume first that $\eta=+$.
By using the relations (Q5), we have
\begin{align}\label{tri-lem-1}
T_{i,3}^{\pm,+}(z_1,z_2,z_3)=\sum_{\sigma\in S_2}\frac{A^\pm_i(z_{\sigma(1)},z_{\sigma(2)},z_3)}
{F_{ii}^\pm(z_3,z_1)F_{ii}^\pm(z_3,z_2)}
x_i^\pm(z_{\sigma(1)})x_i^\pm(z_{\sigma(2)})\phi_i^+(q^{\mp\frac{1}{2}c}z_3),
\end{align}
where
\begin{equation}\begin{split}\label{defai123}
&A^\pm_i(z_1,z_2,z_3)=F_{ii}^\pm(z_3,z_1)F_{ii}^\pm(z_3,z_2) p_i^\pm(z_1,z_2,z_3)\\
&+F_{ii}^\pm(z_3,z_1)G_{ii}^\pm(z_3,z_2) p_i^\pm(z_1,z_3,z_2)
+G_{ii}^\pm(z_3,z_1)G_{ii}^\pm(z_3,z_2) p_i^\pm(z_3,z_1,z_2).
\end{split}\end{equation}

Using \eqref{ri}, \eqref{Fiiep} and \eqref{Giiep}, one gets that
\begin{align*}
&A_i^\pm(z_1,z_2,z_3)\\
=&(z_3^{d_i}+q_i^{\mp d_i}z_1^{d_i})(z_3^{d_i}-q_i^{\pm 2d_i}z_1^{d_i})
 (z_3^{d_i}+q_i^{\mp d_i}z_2^{d_i})(z_3^{d_i}-q_i^{\pm 2d_i}z_2^{d_i})\\
&\cdot (q_i^{\mp \frac{3}{2}d_{i}}z_1^{d_{i}}
        -(q_i^{\frac{d_{i}}{2}}+q_i^{-\frac{d_{i}}{2}})
            z_2^{d_{i}}
        +q_i^{\pm \frac{3}{2}d_{i}}
            z_3^{d_{i}})\\
&+(z_3^{d_i}+q_i^{\mp d_i}z_1^{d_i})(z_3^{d_i}-q_i^{\pm 2d_i}z_1^{d_i})
(q_i^{\mp d_i}z_3^{d_i}+z_2^{d_i})(q_i^{\pm 2d_i}z_3^{d_i}-z_2^{d_i})\\
&\cdot (q_i^{\mp \frac{3}{2}d_{i}}z_1^{d_{i}}
        -(q_i^{\frac{d_{i}}{2}}+q_i^{-\frac{d_{i}}{2}})
            z_3^{d_{i}}
        +q_i^{\pm \frac{3}{2}d_{i}}
            z_2^{d_{i}})\\
            &+(q_i^{\mp d_i}z_3^{d_i}+z_1^{d_i})(q_i^{\pm 2d_i}z_3^{d_i}-z_1^{d_i})
(q_i^{\mp d_i}z_3^{d_i}+z_2^{d_i})(q_i^{\pm 2d_i}z_3^{d_i}-z_2^{d_i})\\
&\cdot
(q_i^{\mp\frac{3}{2}d_{i}}z_3^{d_{i}}
        -(q_i^{\frac{d_{i}}{2}}+q_i^{-\frac{d_{i}}{2}})
            z_1^{d_{i}}
        +q_i^{\pm\frac{3}{2}d_{i}}
            z_2^{d_{i}}).
\end{align*}
We can verify that (by using Maple for example)
\begin{align}\label{tri-lem-2}
A_i^\pm(z_1,z_2,z_3)=F_{ii}^{\pm}(z_1,z_2)B_i^\pm(z_1,z_2,z_3),
\end{align}
where
\begin{align*}
B_i^\pm(z_1,z_2,z_3)=&q_i^{\mp\frac{5}{2}d_i}(q_i^{\pm 2d_i}-1)z_3^{d_i}\big((q_i^{\pm 3d_i}+q_i^{\pm 2d_i}+q_i^{\pm d_i})z_1^{d_i}z_2^{d_i}\\
&-(q_i^{\pm 4d_i}+q_i^{\pm 3d_i}+q_i^{\pm 2d_i}+q_i^{\pm d_i}+1)z_3^{2d_i}+q_i^{\pm 2d_i}(z_1^{d_i}+z_2^{d_i})z_3^{d_i}\big).
\end{align*}

Note that $B_i^\pm(z_1,z_2,z_3)=B_i^\pm(z_2,z_1,z_3)$ and $F_{ii}^\pm(z_2,z_1)=-G_{ii}^\pm(z_1,z_2)$.
Then one can conclude from \eqref{tri-lem-1}, \eqref{tri-lem-1} and the relations (Q8) that
\begin{align*}
&T_{i,3}^{\pm,+}(z_1,z_2,z_3)\\
=&
\frac{B_i^\pm(z_{1},z_{2},z_3)}{F_{ii}^\pm(z_3,z_1)F_{ii}^\pm(z_3,z_1)}
\(\sum_{\sigma\in S_2}F_{ii}^\pm(z_{\sigma(1)},z_{\sigma(2)})
x_i^\pm(z_{\sigma(1)})x_i^\pm(z_{\sigma(2)})\phi_i^+(q^{\mp \frac{1}{2}c}z_3)\)\\
=&
\frac{B_i^\pm(z_{1},z_{2},z_3)}{F_{ii}^\pm(z_3,z_1)F_{ii}^\pm(z_3,z_1)}
\(F_{ii}^\pm(z_1,z_2)x_i^\pm(z_1)x_i^\pm(z_2)-G_{ii}^\pm(z_1,z_2)x_i^\pm(z_2)
x_i^\pm(z_1)\)\phi_i^+(q^{\mp \frac{1}{2}c}z_3)\\
=&0,
\end{align*}
as desired. The proof of the case $\eta=-$ is similar and omitted.
\end{proof}

The following lemma can be viewed as  a twisted analogue of \cite[Lemma 9]{He-representation-coprod-proof}.

\begin{lem}\label{lem:tri-t-2}
Let $(i,j)\in \mathbb I$.
Then in $\ul$ we have:
\begin{equation}\begin{split}\label{eq:tri-decomp-tech-1}
&\sum_{\sigma\in S_{m}}\sum_{r=0}^m
(-1)^r\binom{m}{r}_{q_i^{d_{ij}}} p_{ij,r}^\pm(z_{\sigma(1)},\dots,z_{\sigma(m)},w)
x_i^\pm(z_{\sigma(1)})\cdots x_i^\pm(z_{\sigma(r)})\\
&\quad\times\phi^\eta_j( q^{\mp \eta\frac{1}{2}c}w)
x_i^\pm(z_{\sigma(r+1)})\cdots x_i^\pm(z_{\sigma(m)})=0,
\end{split}
\end{equation}
\begin{equation}\begin{split}\label{eq:tri-decomp-tech-2}
&\sum_{\sigma\in S_m}\sum_{r=0}^m
(-1)^r\binom{m}{r}_{q_i^{d_{ij}}}p_{ij,r}^\pm(z_{\sigma(1)},\dots,z_{\sigma(m)},w)
\xi_i(z_{\sigma(1)})\cdots \xi_i(z_{\sigma(r)})\\
&\quad\times x_j^\pm(w)
\xi_i(z_{\sigma(r+1)})\cdots \xi_i(z_{\sigma(m)})=0,
\end{split}\end{equation}
where $\eta=\pm$, $m=m_{ij}$, $\xi_i(z_p)=x_i^\pm(z_p)$ if $p\ne 1$
and $\xi(z_1)=\phi_i^\eta(q^{\mp \eta \frac{1}{2}c}z_1)$.
\end{lem}

\begin{proof}
We prove \eqref{eq:tri-decomp-tech-1} and \eqref{eq:tri-decomp-tech-2} for the case  $\eta=+$, as the proof of $-$ case is similar.
Denote by $T^\pm$ the LHS of \eqref{eq:tri-decomp-tech-1} (with $\eta=+$).
Recall from definition that
\begin{align}\label{gijwz}
g_{ij}(w/z)^{\pm 1}=\frac{G_{ij}^\pm(z,w)}{F_{ij}^\pm(z,w)}.
\end{align}
Using this, it follows from the relations (Q5) that
\begin{align}\label{Tep1}
T^\pm=\sum_{\sigma\in S_m}
B^\pm(z_{\sigma(1)},\dots,z_{\sigma(m)},w)
x_i^\pm(z_{\sigma(1)})\cdots x_i^\pm(z_{\sigma(m)})\phi_j^+(q^{-\frac{1}{2}c}w),
\end{align}
where
\begin{align*}
&B^\pm(z_1,\dots,z_m,w)=\prod_{s=1}^{m}
\frac{1}{G_{ij}^\pm(z_s,w)}\sum_{r=0}^m
\binom{m}{r}_{q_i^{d_{ij}}}(-1)^r\\
&\quad\times p_{ij,r}^\pm(z_1,\dots,z_m,w)
\prod_{1\le a\le r}
    G_{ij}^\pm(z_a,w)
\prod_{r<a\le m}
    F_{ij}^\pm(z_a,w)\\
&=\prod_{1\le a<b\le m}p_{ij}^\pm(z_a,z_b)
    \prod_{s=1}^{m}\frac{1}{G_{ij}^\pm(z_s,w)}
    \sum_{r=0}^m \binom{m}{r}_{q_i^{d_{ij}}}(-1)^r\\
&\quad\times
\prod_{1\le a\le r}(q_i^{\pm a_{ij}d_{ij}}z_a^{d_{ij}}-w^{d_{ij}})
\prod_{r<a\le m}(z_a^{d_{ij}}-q_i^{\pm a_{ij}d_{ij}}w^{d_{ij}}).
\end{align*}
It was proved in \cite[Lemma 5]{He-representation-coprod-proof} that
there exist  polynomials $f_{\pm,1},f_{\pm,2},\dots,f_{\pm,m-1}$ in $m-1$ variables such that
\begin{equation}\begin{split}\label{Tep2}
B^\pm(z_1,&\dots,z_m,w)=\prod_{1\le a<b\le m}p_{ij}^\pm(z_a,z_b)
    \prod_{s=1}^{m}\frac{1}{G_{ij}^\pm(z_s,w)}\\
&\sum_{1\le r\le m-1}
(z_{r}^{d_{ij}}-q_i^{\pm 2d_{ij}}z_{r+1}^{d_{ij}})
f_{\pm,r}(z_1,\dots,z_{r-1},z_{r+2},\dots,z_m,w).
\end{split}\end{equation}

For $r=1,\dots,m-1$, define an equivalent relation in $S_m$ by $\sigma\sim_r \sigma'$ if
and only if  $\sigma(r)=\sigma'(r+1)$,
$\sigma(r+1)=\sigma'(r)$,
and $\sigma(r'')=\sigma'(r'')$ for all $r''\ne r,r+1$.
Let $S_m^{(r)}$ be a complete set of equivalence class representatives.
Then from \eqref{Tep1} and \eqref{Tep2} we have
\begin{align*}
T^\pm=&\sum_{r=1}^{m-1}\sum_{\sigma\in S_m^{(r)}}\prod_{\substack{1\le a<b\le m,\\ (a,b)\ne (r,r+1)}}
p_{ij}^\pm(z_{\sigma(a)},z_{\sigma(b)})
    \prod_{s=1}^{m}\frac{1}{G_{ij}^\pm(z_s,w)}\\
&\times f_{\pm,r}(z_{\sigma(1)},\dots,z_{\sigma(r-1)},z_{\sigma(r+1)},\dots,z_{\sigma(m)},w)x_i^\pm(z_{\sigma(1)})\cdots x_i^\pm(z_{\sigma(r-1)})\\
&\times A^\pm(z_{\sigma(r)},z_{\sigma(r+1)})
x_i^\pm(z_{\sigma(r+2)})\cdots x_i^\pm(z_{\sigma(m)})\phi_j^+(q^{-\frac{1}{2}c}w),
\end{align*}
where
\begin{align*}
&A^\pm(z_{\sigma(r)},z_{\sigma(r+1)})=
p_{ij}^\pm(z_{\sigma(r)},z_{\sigma(r+1)})(z_{\sigma(r)}^{d_{ij}}-q_i^{\pm 2d_{ij}}z_{\sigma(r+1)}^{d_{ij}})
x_i^\pm(z_{\sigma(r)})x_i^\pm(z_{\sigma(r+1)})\\
&  +p_{ij}^\pm(z_{\sigma(r+1)},z_{\sigma(r)})(z_{\sigma(r+1)}^{d_{ij}}-q_i^{\pm 2d_{ij}}z_{\sigma(r)}^{d_{ij}})
x_i^\pm(z_{\sigma(r+1)})x_i^\pm(z_{\sigma(r)}).
\end{align*}
On the other hand, by using (Q8) we have
\begin{align*}
 & p_{ij}^\pm(z,w)(z^{d_{ij}}-q_i^{\pm 2d_{ij}}w^{d_{ij}})x_i^\pm(z)x_i^\pm(w)\\
 =&\frac{q_i^{\pm 2d_{ij}}z^{d_{ij}}-w^{d_{ij}}}{q_i^{\pm 2d_i}z^{d_i}-w^{d_i}}
 \frac{z^{d_{ij}}-q_i^{\pm 2d_{ij}}w^{d_{ij}}}{z^{d_{i}}-q_i^{\pm 2d_{i}}w^{d_i}}
 F_{ii}^\pm(z,w)x_i^\pm(z)x_i^\pm(w)\\
 =& (q_i^{\pm 2d_{ij}}z^{d_{ij}}-w^{d_{ij}})
 \(\frac{q_i^{\pm 2d_{ij}}w^{d_{ij}}-z^{d_{ij}}}{q_i^{\pm 2d_{i}}w^{d_i}-z^{d_{i}}}
 \frac{G_{ii}^\pm(z,w)}{q_i^{\pm 2d_i}z^{d_i}-w^{d_i}}\)x_i^\pm(w)x_i^\pm(z)\\
  =&-(w^{d_{ij}}-q_i^{\pm 2d_{ij}}z^{d_{ij}})p_{ij}^\pm(w,z)x_i^\pm(w)x_i^\pm(z).
\end{align*}
This gives that $A^\pm(z_{\sigma(r)},z_{\sigma(r+1)})=0$ for all $r$ and hence $T^\pm=0$, as required.

Now we turn to prove \eqref{eq:tri-decomp-tech-2}.
Denote by $R^\pm$ the LHS of \eqref{eq:tri-decomp-tech-2}.
Recall that
\begin{align}\label{giiwz}
g_{ii}(w/z)^{\pm 1}=\frac{G_{ii}^\pm(z,w)}{F_{ii}^\pm(z,w)}
=\frac{(q_i^{\mp d_i}z^{d_i}+w^{d_i})^{s_i-1}(q_i^{\pm 2d_i}z^{d_i}-w^{d_i})}
{(z^{d_i}+q_i^{\mp d_i}w^{d_i})^{s_i-1}(z^{d_i}-q_i^{\pm 2 d_i}w^{d_i})}.
\end{align}
Then it follows from \eqref{gijwz}, \eqref{giiwz} and the relation (Q5) that
\begin{align*}
R^\pm=\sum_{\pi\in S_{m-1}}
\sum_{r=1}^{m}&
    C_{ij,r}^\pm(z_{1},z_{\pi(2)},\dots,z_{\pi(m-1)},z_{\pi(m)},w)
    x_i^\pm(z_{\pi(2)})\cdots x_i^\pm(z_{\pi(r)})x_j^\pm(w)\\
    &\times x_i^\pm(z_{\pi(r+1)})\cdots x_i^\pm(z_{\pi(m)})
    \phi_i^+(q^{\mp\frac{1}{2}c}z_{1}),
\end{align*}
where   $S_{m-1}$ acts on the set $\{2,\dots,m\}$ and for $1\le r\le m$,
\begin{align*}
C_{ij,r}^{\pm}(z_1,z_2,\dots,z_m,w)&=D_{ij,r}^\pm(z_1,\dots,z_m,w)
P_{ij}^{(r)}(z_1^{d_{ij}},z_2^{d_{ij}},\dots,z_m^{d_{ij}},w^{d_{ij}},q_i^{\pm d_{ij}}),\\
D_{ij,r}^\pm(z_1,z_2,\dots,z_m,w)&=\frac{1}{F_{ij}^\pm(z_1,w)}
    \prod_{2\le a\le m}
    \frac{\(q_i^{\mp d_i}z_1^{d_i}+z_a^{d_i}\)^{s_i-1}}{z_1^{d_i}-q_i^{\pm 2 d_i}z_a^{d_i}}\\
&\times
    \prod_{2\le a<b\le m}\frac{\(z_a^{d_i}+q_i^{\mp d_i}z_b^{d_i}\)^{s_i-1}\(q_i^{\pm 2d_{ij}}z_a^{d_{ij}}-z_b^{d_{ij}}\)}{q_i^{\pm 2d_i}z_a^{d_i}-z_b^{d_i}}\\
&\times
    \prod_{2\le a\le r}\frac{q_i^{\pm a_{ij}d_{ij}}z_a^{d_{ij}}-w^{d_{ij}}}{G_{ij}^\pm(z_a,w)}
    \prod_{r<a\le m}\frac{z_a^{d_{ij}}-q_i^{\pm a_{ij}d_{ij}}w^{d_{ij}}}{F_{ij}^\pm(z_a,w)}
\end{align*}
and
\begin{align*}
&P^{(r)}_{ij}(z_1,z_2,\dots,z_m,w,q)\\
=&\binom{m}{r}_q(-1)^r
    \sum_{p=1}^r \prod_{2\le a\le p}\(z_1-q^2z_a\)
    \prod_{p<a\le m}\(q^{2}z_1-z_a\)
    \( q^{a_{ij}}z_1-w\)\\
    &\,+\binom{m}{r-1}_q(-1)^{r-1}
    \sum_{p=r}^{m} \prod_{2\le a\le p}\(z_1-q^2z_a\)
    \prod_{p<a\le m}\(q^2z_1-z_a\)
    \( z_1-q^{a_{ij}}w\).
\end{align*}

It was proved in \cite[Lemma 6]{He-representation-coprod-proof} that
\begin{align}
P_{ij}^{(1)}(z_1,&\dots,z_m,w,q)=(z_2-q^{-a_{ij}}w)
f^{(1)}_{m}(z_1,z_3,\dots,z_m,w,q)\label{eq:tri-decomp-temp-P1}\\
&+\sum_{2\le a\le m-1}(z_{a+1}-q^{-2}z_{a})
f^{(1)}_a(z_1,\dots,z_{a-1},z_{a+2},\dots,z_m,w,q),\nonumber\\
P_{ij}^{(r)}(z_1,&\dots,z_m,w,q)=(w-q^{-a_{ij}}z_r)
f^{(r)}_{r}(z_1,\dots,z_{r-1},z_{r+1},\dots,z_m,w,q)
\label{eq:tri-decomp-temp-P2}\\
&+(z_{r+1}-q^{-a_{ij}}w)
f^{(r)}_{m}(z_1,\dots,z_r,z_{r+2},\dots,z_m,w,q)\nonumber\\
&+\sum_{\substack{2\le a\le m-1\\ a\ne r}}
(z_{a+1}-q^{-2}z_{a})
f^{(r)}_{a}(z_1,\dots,z_{a-1},z_{a+2},\dots,z_m,w,q),\nonumber\\
P_{ij}^{(m)}(z_1,&\dots,z_m,w,q)=(w-q^{-a_{ij}}z_m)
f^{(m)}_{m}(z_1,\dots,z_{m-1},w,q)\label{eq:tri-decomp-temp-P3}\\
&+\sum_{2\le a\le m-1}
(z_{a+1}-q^{-2}z_{a})
f^{(m)}_a(z_1,\dots,z_{a-1},z_{a+2},\dots,z_m,w,q),\nonumber
\end{align}
where
$f_a^{(r)},a=2,\dots,m$ are some polynomials of $m-1$ variables,
of degree at most $1$ in each variable.

For $a=2,\dots,m-1$ and $\pi,\pi'\in S_{m-1}$, we define the equivalent relation $\pi\sim_a\pi'$ and
the equivalence class $S_{m-1}^{(a)}$ as above.
Then for any $\pi,\pi'\in S_{m-1}$ with $\pi\sim_a\pi'$, we have from (Q8) that
\begin{align*}
  &(z_{\pi(a)}^{d_{i}}+q_i^{\mp d_i}z_{\pi(a+1)}^{d_{i}})^{s_i-1}
  \frac{q_i^{\pm2d_{ij}}z_{\pi(a)}^{d_{ij}}-z_{\pi(a+1)}^{d_{ij}}}
  { q_i^{\pm d_i}z_{\pi(a)}^{d_i}-z_{\pi(a+1)}^{d_i} }
  (z_{\pi(a+1)}^{d_{ij}}-q_i^{\mp 2d_{ij}}z_{\pi(a)}^{d_{ij}})
  x_i^\pm(z_{\pi(a)})x_i^\pm(z_{\pi(a+1)})\\
  =&(z_{\pi'(a)}^{d_{i}}+q_i^{\mp d_i}z_{\pi'(a+1)}^{d_{i}})^{s_i-1}
  \frac{q_i^{\pm2d_{ij}}z_{\pi'(a)}^{d_{ij}}-z_{\pi'(a+1)}^{d_{ij}}}
  { q_i^{\pm d_i}z_{\pi'(a)}^{d_i}-z_{\pi'(a+1)}^{d_i} }
  (z_{\pi'(a+1)}^{d_{ij}}-q_i^{\mp 2d_{ij}}z_{\pi'(a)}^{d_{ij}})
  x_i^\pm(z_{\pi'(a)})x_i^\pm(z_{\pi'(a+1)}).
\end{align*}
This gives ($1\le r\le m$)
\begin{equation}\begin{split}\label{eq:tri-decomp-temp-ii-com}
  &D_{ij,r}^\pm(z_1,z_{\pi(2)},\dots,z_{\pi(m)},w)
  (z_{\pi(a+1)}^{d_{ij}}-q_i^{\mp 2d_{ij}}z_{\pi(a)}^{d_{ij}})
  x_i^\pm(z_{\pi(a)})x_i^\pm(z_{\pi(a+1)})\\
  =\,&D_{ij,r}^\pm(z_1,z_{\pi'(2)},\dots,z_{\pi'(m)},w)
  (z_{\pi'(a+1)}^{d_{ij}}-q_i^{\mp 2d_{ij}}z_{\pi'(a)}^{d_{ij}})
  x_i^\pm(z_{\pi'(a)})x_i^\pm(z_{\pi'(a+1)}).
\end{split}\end{equation}

In view of \eqref{eq:tri-decomp-temp-P1}, \eqref{eq:tri-decomp-temp-P2}, \eqref{eq:tri-decomp-temp-P3} and
\eqref{eq:tri-decomp-temp-ii-com}, for $r=1,\dots,m$ and $a=2,\dots,m-1$ with $a\ne r$, we have
\begin{align*}
\sum_{\pi\in S_{m-1}^{(a)}} D_{ij,r}^\pm(z_1,z_{\pi(2)},\dots,z_{\pi(m)},w)
P^{(r)}_{ij}(z_1^{d_{ij}},z_{\pi(2)}^{d_{ij}},\dots,z_{\pi(m)}^{d_{ij}},w^{d_{ij}},q_i^{\pm d_{ij}})\\
\cdot x_i^\pm(z_{\pi(2)})\cdots x_i^\pm(z_{\pi(r)})x_j^\pm(w)
    x_i^\pm(z_{\pi(r+1)})\cdots x_i^\pm(z_{\pi(m)})
    \phi_i^+(q^{\mp\frac{1}{2}c}z_{1})=0.
\end{align*}
This implies that all the terms in $R^\pm$ which contain the polynomials $f^{(r)}_a$ with $a\ne r,m$ can be erased.
Furthermore, from (Q8), it follows that
\begin{align*}
  &D_{ij,r}^\pm(z_1,z_{\pi(2)},\dots,z_{\pi(m)},w)(z_{\pi(r)}^{d_{ij}}-q_i^{\pm a_{ij}d_{ij}}w^{d_{ij}})x_i^\pm(z_{\pi(r)})x_j^\pm(w)\\
  =& D_{ij,r-1}^\pm(z_1,z_{\pi(2)},\dots,z_{\pi(m)},w)(q_i^{\pm a_{ij}d_{ij}}z_{\pi(r)}^{d_{ij}}-w^{d_{ij}})x_j^\pm(w)x_i^\pm(z_{\pi(r)}).
\end{align*}
Thus, we  obtain
\begin{align*}
R^\pm=&\sum_{\pi\in S_{m-1}}
    \sum_{r=2}^{m}D_{ij,r}^\pm(z_1,z_{\pi(2)},\dots,z_{\pi(m)},w)
(q_i^{\mp d_{ij}a_{ij}}w^{d_{ij}}-z_{\pi(r)}^{d_{ij}})
\\
\times&
    \(f^{(r)}_{r}-f^{(r-1)}_{m}\)(z_{1}^{d_{ij}},z_{\pi(2)}^{d_{ij}},\dots,
    z_{\pi(r-1)}^{d_{ij}},z_{\pi(r+1)}^{d_{ij}},\dots,
    z_{\pi(m)}^{d_{ij}},w^{d_{ij}},q_i^{\pm d_{ij}})\\
&\quad\times x_i^\pm(z_{\pi(2)})\cdots x_i^\pm(z_{\pi(r-1)})x_j^\pm(w)x_i^\pm(z_{\pi(r)})\cdots x_i^\pm(z_{\pi(m)})\phi_i^+(q^{\mp\frac{c}{2}}z_{1}).
\end{align*}
Recall from \cite[Lemma 7]{He-representation-coprod-proof} that
\begin{align*}
  f^{(r)}_{r}-f^{(r-1)}_{m}=\sum_{a=1}^{m-2}(z_{r}-q^{2}z_{r+1})g^{(r)}_a(z_1,\dots,z_{r-1},z_{r+2},\dots,z_m,w,q),
\end{align*}
where $g_a^{(r)}$ are some polynomials.
This together with  \eqref{eq:tri-decomp-temp-ii-com} gives  $R^\pm=0$.
\end{proof}

\subsection{Proof of Theorem \ref{thm:tri-decomp}}

We start with two propositions which are about the compatibility with affine quantum Serre relations (Q8), (Q9), and (Q10).

\begin{prop}\label{prop:tri-re1} For $i,j\in I$, the following hold in  $\uf$:
\begin{equation}
[F_{ij}^\pm(z,w)x_i^\pm(z)x_j^\pm(w)-G_{ij}^\pm(z,w)x_j^\pm(z)x_i^\pm(w), x_k^\mp(w_0)]=0.
\end{equation}
\end{prop}

\begin{proof}
From the relations (Q5)-(Q7), it follows that
\begin{align*}
 &\qquad \pm(q_k-q_k^{-1})[F_{ij}^\pm(z,w) x_i^\pm(z)x_j^\pm(w),x_k^\mp(w_0)]\\
&=   F_{ij}^\pm(z,w)
\sum_{p\in\Z}\delta_{i,\mu^p(k)}
    \(\phi_i^+(q^{\mp\half c}z)\delta\(\frac{q^{\pm c}\xi^pw_0}{z}\)
    -\phi_i^-(q^{\pm\half c}z)\delta\(\frac{q^{\mp c}\xi^p w_0}{z}\)\)x_j^\pm(w)\\
&+F_{ij}^\pm(z,w)x_i^\pm(z)
\sum_{p\in\Z}\delta_{j,\mu^p(k)}
\(\phi_j^+(q^{\mp\half c}w)\delta\(\frac{q^{\pm c}\xi^pw_0}{w}\)
    -\phi_j^-(q^{\pm\half c}w)\delta\(\frac{q^{\mp c}\xi^p w_0}{w}\)\)\\
&=G_{ij}^\pm(z,w)x_j^\pm(w)
 \sum_{p\in\Z}\delta_{i,\mu^p(k)}
    \(\phi_i^+(q^{\mp\half c}z)\delta\(\frac{q^{\pm c}\xi^pw_0}{z}\)
    -\phi_i^-(q^{\pm\half c}z)\delta\(\frac{q^{\mp c}\xi^p w_0}{z}\)\)\\
&+ G_{ij}^\pm(z,w)
\sum_{p\in\Z}\delta_{j,\mu^p(k)}
\(\phi_j^+(q^{\mp\half c}w)\delta\(\frac{q^{\pm c}\xi^pw_0}{w}\)
    -\phi_j^-(q^{\pm\half c}w)\delta\(\frac{q^{\mp c}\xi^p w_0}{w}\)\)x_i^\pm(z)\\
&=\pm(q_k-q_k^{-1})[G_{ij}^\pm(z,w)x_j^\pm(w)x_i^\pm(z),x_k^\mp(w_0)],
\end{align*}
where we have used the facts:
\begin{align*}
G_{ij}^\pm(z,w)=F_{ij}^{\pm}(z,w) g_{ij}(w/z)^{\pm 1}\quad\te{and}\quad
F_{ij}^{\pm}(z,w)=G_{ij}^{\pm}(z,w) g_{ji}(z/w)^{\pm 1}.
\end{align*}
\end{proof}

\begin{prop}\label{prop:tri-re2} Let $i,j,k\in I$. Then in  $\ul$ we have:
\begin{align}\label{tri-q9}
[\sum_{\sigma\in S_{3}}
    p_i^\pm(z_{\sigma(1)},z_{\sigma(2)},z_{\sigma(3)})\,
                x_i^\pm(z_{\sigma(1)})x_i^\pm(z_{\sigma(2)})x_i^\pm(z_{\sigma(3)}),x_j^{\mp}(w)]
               =0,\ \te{if}\ s_i=2,
\end{align}
\begin{equation}\begin{split}\label{tri-q10}
&\Bigg[\sum_{\sigma\in S_{{m}_{ij}}}\sum_{r=0}^{{m}_{ij}}
    (-1)^r \binom{{m}_{ij}}{r}_{q_i^{d_{ij}}}
    p_{ij,r}^\pm(z_{\sigma(1)},\dots,z_{\sigma(m)},w)
    x_i^\pm(z_{\sigma(1)})\\
&\cdots x_i^\pm(z_{\sigma(r)})  x_j^\pm(w)
       x_i^\pm(z_{\sigma(r+1)})\cdots x_i^\pm(z_{\sigma({m}_{ij})})\ , x_k^\mp(w_0)\Bigg]=0,\ \te{if}\ \check{a}_{ij}<0.
\end{split}\end{equation}
\end{prop}
\begin{proof}
We first prove \eqref{tri-q9}.
Note that it suffices to treat the case  $i=j$.
In this case, it follows from  (Q7) that  the LHS of \eqref{tri-q9}  equals to
\begin{align*}
\frac{\pm1}{q_i-q_i^{-1}}\sum_{k\in \Z_N}\delta_{i,\mu^k(i)}\sum_{\sigma\in S_{3}}
    p_i^\pm(z_{\sigma(1)},z_{\sigma(2)},z_{\sigma(3)})C_{i,k}^\pm(z_1,z_2,z_3),
    \end{align*}
    where $C_{i,k}^\pm(z_1,z_2,z_3)$ stands for the following formal series:
    \begin{align*}
             & x_i^\pm(z_{\sigma(1)})x_i^\pm(z_{\sigma(2)})
        \Big(\phi_i^+(z_{\sigma(3)} q^{\mp\half c})\delta\left(
            \frac{\xi^kw q^{\pm c}}{z_{\sigma(3)}}
        \right)
        -
        \phi_i^-(z_{\sigma(3)} q^{\pm\half c})\delta\left(
            \frac{\xi^kw q^{\mp c} }{z_{\sigma(3)}}
        \right)
    \Big)\\
   & +x_i^\pm(z_{\sigma(1)})
        \Big(\phi_i^+(z_{\sigma(2)}q^{\mp\half c})\delta\left(
            \frac{\xi^kw q^{\pm c}}{z_{\sigma(2)}}
        \right)
        -
        \phi_i^-(z_{\sigma(2)}q^{\pm\half c})\delta\left(
            \frac{\xi^kw q^{\mp c} }{z_{\sigma(2)}}
        \right)
    \Big)x_i^\pm(z_{\sigma(3)})\\
   & +\Big(\phi_i^+(z_{\sigma(1)}q^{\mp\half c})\delta\left(
            \frac{\xi^kw q^{\pm c}}{z_{\sigma(1)}}
        \right)
        -
        \phi_i^-(z_{\sigma(1)}q^{\pm\half c})\delta\left(
            \frac{\xi^kw q^{\mp c} }{z_{\sigma(1)}}
        \right)
    \Big) x_i^\pm(z_{\sigma(2)})x_i^\pm(z_{\sigma(3)}).
    \end{align*}
    It is straightforward to see that
    \begin{align*}
    &\sum_{\sigma\in S_{3}}p_i^\pm(z_{\sigma(1)},z_{\sigma(2)},z_{\sigma(3)})C_{i,k}^\pm(z_1,z_2,z_3)
    \\=&
    \sum_{l=1}^3T_{i,l}^{\pm,+}(z_1,z_2,z_3)\delta\left(
            \frac{\xi^kw q^{\pm c}}{z_l}
        \right)-T_{i,l}^{\pm,-}(z_1,z_2,z_3)\delta\left(
            \frac{\xi^kw q^{\mp c} }{z_l}
        \right),
        \end{align*}
      recalling  that $T_{i,l}^{\pm,\eta}(z_1,z_2,z_3)$ stands for the LHS of \eqref{eq:tri-decomp-tech-3}.
        Thus the relation  \eqref{tri-q9} follows from Lemma \ref{lem:tri-t-1}.
Next, by a same argument as that in the proof of \cite[Lemma 10]{He-representation-coprod-proof},
\eqref{tri-q10} follows from Lemma \ref{lem:tri-t-2}.
\end{proof}

\textbf{Proof of Theorem \ref{thm:tri-decomp}:}
We would use a general proof of triangular decompositions (cf. \cite[Lemma 3.5]{He-representation-coprod-proof}).
Let $A$ be a completed and separated   $\C[[\hbar]]$-algebra and
$(A^-,H,A^+)$ a triangular decomposition of $A$. Let $B^+$ and $B^-$ be respectively a closed two-sided ideal of
$A^+$ and $A^-$, and let $B$ be the closed ideal of $A$ generated by $B^++B^-$.
Set $C=A/B$ and denote by $C^\pm$ the image of
$B^\pm$ in $C$.
Assume that $A B^+\subset B^+ A$ and $B^- A\subset A B^-$. Then  $(C^+,H,C^-)$ is a triangular decomposition
of $C$ and $C^\pm$ are isomorphic to $A^\pm/B^\pm$.

Note that
$(\uf^-,\hf,\uf^+)$
is a triangular decomposition of $\uf$.
Moreover, $\uf^+$ (resp. $\uf^-$)  is isomorphic to the $\C[[\hbar]]$-algebras
topologically freely generated by  $x_{i,n}^+$ (resp.  $x_{i,n}^-$),
and $\hf$ is isomorphic to the $\C[[\hbar]]$-algebra topologically
 generated by  $h, h_{i,m}, c,d$
with relations (Q0)-(Q2).
In view of the above criteria, it follows from  Proposition \ref{prop:tri-re1} that
$\ul$ admits a triangular decomposition
$(\ul^-,\hl,\ul^+)$ induced from the triangular decomposition $(\uf^-,\hf,\uf^+)$ of $\uf$.
Furthermore, due to Proposition \ref{prop:tri-re2},  the triangular decomposition
$(\ul^-,\hl,\ul^+)$ of
$\ul$ induces a triangular decomposition of $\qtar$ as stated in
Theorem \ref{thm:tri-decomp}.
This completes the proof of Theorem \ref{thm:tri-decomp}.

\section{Affine quantum Serre relations}\label{sec:qptar}
In this section, we introduce a notion of ``normal ordered products'' of currents on restricted modules for $\ul^\pm$ and
then reformulate the affine quantum Serre relations
 (Q9) and (Q10) in terms of normal ordered products.
Throughout this section, we assume that the automorphism $\mu$ satisfies the condition (LC3).
We remark that in this case the polynomials appearing in the quantum affine Serre relations (Q8) and (Q9) can be simplified as follows: ($(i,j)\in\mathbb I$ and $r=0,1\dots,m_{ij}$)
\begin{align}
  &F_{ij}^\pm(z,w)=z^{d_{ij}}-q_i^{\pm d_{ij} a_{ij}}w^{d_{ij}},\\
  &G_{ij}^\pm(z,w)=q_i^{\pm d_{ij} a_{ij}}z^{d_{ij}}-w^{d_{ij}},\\
  &p_{ij,r}^\pm(z_1,\dots,z_{m_{ij}},w)=\prod_{1\le a<b\le m_{ij}}p_{ij}^\pm(z_a,z_b).
\end{align}

\subsection{Normal ordered products}

We start with some notations.
Recall that a $\C[[\hbar]]$-module $W$ is called \emph{topologically free} if $W=W^0[[\hbar]]$ for
some vector space $W^0$ over $\C$.
We denote by $\mathcal M_f$ the category of topologically free $\C[[\hbar]]$-modules.
For $W\in \mathcal M_f$ and $m,n\in \Z_+$, there is a $\C[[\hbar]]$-module map
\begin{align*}
\tilde{\pi}_n^{(m)}: \mathrm{End}_{\C[[\hbar]]}(W)[[z_1^{\pm 1},\dots,z_m^{\pm 1}]]
\rightarrow \mathrm{End}_{\C[[\hbar]]}(W/\hbar^nW)[[z_1^{\pm 1},\dots,z_m^{\pm 1}]]
\end{align*}
induced by the canonical $\C[[\hbar]]$-map $\mathrm{End}_{\C[[\hbar]]}(W)\rightarrow \mathrm{End}_{\C[[\hbar]]}(W/\hbar^nW)$.
As usual, for a $\C[[\hbar]]$-module $W_0$,
we denote by $W_0((z_1,\dots,z_m))$ the space of lower truncated (infinite) integral power series in
the  variables $z_1,\dots,z_m$ with coefficients in $W_0$. Set
\[\E^{(m)}(W_0)=\mathrm{Hom}_{\C[[\hbar]]}(W_0,W_0((z_1,\dots,z_m))).\]
The following notion is an $\hbar$-analogue of $\E^{(m)}(W_0)$ introduced in \cite{Li-h-adic}.

\begin{de}\label{de:emhw}
Let $W$ be a topologically free $\C[[\hbar]]$-module and  $m$  a positive integer.
Define $\E^{(m)}_\hbar(W)$ to be the $\C[[\hbar]]$-submodule of
$\mathrm{End}_{\C[[\hbar]]}(W)[[z_1^{\pm 1},\dots,z_m^{\pm 1}]]$, consisting of each formal series
$\psi(z_1,\dots,z_m)$
such that for every $n\in \Z_+$,
\begin{align*}\label{defehr}
\tilde{\pi}_n^{(m)}(\psi(z_1,\dots,z_m))\in \E^{(m)}(W/\hbar^n W),
\end{align*}
or equivalently, for every $v\in W$ and $n\in \Z_+$,
\begin{align*}
\psi(z_1,\dots,z_m)v\in W((z_1,\dots,z_m))+\hbar^n W[[z_1^{\pm 1},\dots,z_m^{\pm 1}]].
\end{align*}
\end{de}

Let $W=W^0[[\hbar]]\in \mathcal M_f$.
For convenience, we will also write $\E_\hbar(W)=\E_\hbar^{(1)}(W)$.
One notices that for any $m\in \Z_+$,  $\E_\hbar^{(m)}(W)
=\E^{(m)}(W^0)[[\hbar]]$ is also a topologically free $\C[[\hbar]]$-module.
Recall  from \cite[Remark 4.7]{Li-h-adic} that for every $n\in \Z_+$,
 there is a  surjective
  $\C[[\hbar]]$-map $\pi_n^{(m)}:\E_\hbar^{(m)}(W)\rightarrow
\E^{(m)}(W/\hbar^n W)$ induced by $\tilde{\pi}_n^{(m)}$ with kernel $\hbar^n\E_\hbar^{(m)}(W)$.
Furthermore, there is an inverse system
\begin{eqnarray}\label{indefem}
0\longleftarrow \E^{(m)}(W/\hbar W)\stackrel{\theta_1^{(m)}}{\longleftarrow} \E^{(m)}(W/\hbar^2 W)\stackrel{\theta_2^{(m)}}{\longleftarrow} \E^{(m)}(W/\hbar^3 W)
\stackrel{\theta_3^{(m)}}{\longleftarrow}\cdots
\end{eqnarray}
with $\E_\hbar^{(m)}(W)$ equipped with $\C[[\hbar]]$-maps $\pi_n^{(m)}$ as an inverse limit, where for $n\in \Z_+$,
\[\theta_n^{(m)}:\E^{(m)}(W/\hbar^{n+1} W)\rightarrow \E^{(m)}(W/\hbar^n W)\]
is the canonical $\C[[\hbar]]$-map.

\begin{rem}
{\em Recall that for any $\C[[\hbar]]$-module $W_0$ and invertible elements $c_1,c_2,\dots,c_m\in \C[[\hbar]]$, there is a (well-defined) specialization map (cf.\,\cite{LL})
\begin{align*}
\prod_{1\le r\le m}\lim_{z_r\mapsto c_r z}:\quad\E^{(m)}(W_0)&\rightarrow \E(W_0),\\\quad \sum_{n_1,\dots,n_m\in \Z} v_{n_1,\dots,n_m}z_1^{n_1}\cdots z_m^{n_m}
&\mapsto\sum_{n_1,\dots,n_m\in \Z} v_{n_1,\dots,n_m}(c_1z)^{n_1}\cdots (c_mz)^{n_m}.\end{align*}
Let $A(z_1,\dots,z_m)\in \E_\hbar^{(m)}(W)$ with $W$ a topologically free $\C[[\hbar]]$-module.
For $n\in\Z_+$, set
\begin{align*}
A_n(c_1z,\dots,c_mz):=\prod_{1\le r\le m}\lim_{z_r\mapsto c_r z}\pi_n^{(m)}(A(z_1,\dots,z_m))\in \E(W/\hbar^nW).
\end{align*}
It is clear that for each $n$,
\[\prod_{1\le r\le m}\lim_{z_r\mapsto c_r z}\circ\, \theta_n^{(m)}=\theta_n^{(1)}\circ \prod_{1\le r\le m}\lim_{z_r\mapsto c_r z}
\quad \text{and}\quad \pi_n^{(m)}=\theta_n^{(m)}\circ \pi_{n+1}^{(m)}.\]
This implies that $\theta_n^{(1)}(A_{n+1}(c_1z,\dots,c_mz))=A_{n}(c_1z,\dots,c_mz)$.
Since $\E_\hbar(W)$ is the inverse limit of the system \eqref{indefem} (with $m=1$), there is an inverse limit
\[A(c_1z,\dots,c_mz):= \varprojlim_{n>0} A_n(c_1z,\dots,c_mz)\in \E_\hbar(W),\]
called the specialization of $A(z_1,\dots,z_m)$ at $z_r=c_rz$ for $1\le r\le m$.
}
\end{rem}

\begin{de}\label{de:restrictedulmod}
Let $\U$ be one of the algebras $\ul^\pm, \us^\pm, \qtar^\pm$, or $\qtar$.
We say that a (left) $\U$-module $W$ is \emph{restricted}
if $W$ is topologically free as a $\C[[\hbar]]$-module,
and for each $i\in I,w\in W$, $x_{i,n}^\pm w\rightarrow 0$ as $n\rightarrow +\infty$,
that is, for every $m\in\Z_+$,
there exists  $m'\in\Z$, such that
\begin{align*}
  x_{i,n}^\pm w\in\hbar^mW\quad\te{for }n\ge m'.
\end{align*}
We denote by $\R^l_\pm$, $\R^s_\pm$, $\R_\pm$, and $\R$ the categories of restricted modules for
$\ul^\pm$, $\us^\pm$, $\qtar^\pm$, and $\qtar$, respectively.
\end{de}

Note that a  $\ul^\pm$-module $W$ is restricted if and only if $W\in \mathcal M_f$ and
$x_i^\pm(z)\in \E_\hbar(W)$ for $i\in I$.
The following argument is standard in the vertex algebra theory.

\begin{lem}\label{lem:normal-ordering-pre-desc}
For $W\in \R_\pm^{l}$ and  $i,j\in I$, we have
\begin{align*}
F_{ij}^\pm(z_1,z_2)x_i^\pm(z_1)x_j^\pm(z_2)\in\E_\hbar^{(2)}(W).
\end{align*}
\end{lem}
\begin{proof}
Let $v\in W$ and $n\in \Z_+$. Then by definition we have
$F_{ij}^\pm(z_1,z_2)x_i^\pm(z_1)x_j^\pm(z_2)v\in W((z_1))((z_2))+\hbar^n W[[z_1^{\pm 1},z_2^{\pm 1}]]$.
On the other hand, in view of (Q8), we have
\[F_{ij}^\pm(z_1,z_2)x_i^\pm(z_1)x_j^\pm(z_2)v=G_{ij}^\pm(z_1,z_2)x_j^\pm(z_2)x_i^\pm(z_1)v\in W((z_2))((z_1))+\hbar^n W[[z_1^{\pm 1},z_2^{\pm 1}]].\]
This forces that $F_{ij}^\pm(z_1,z_2)x_i^\pm(z_1)x_j^\pm(z_2)\in \E_\hbar^{(2)}(W)$, as required.
\end{proof}

Let $\C_*[[z_1,\dots,z_m,\hbar]]$ denote the algebra extension of $\C[[z_1,\dots,z_m,\hbar]]$
by inverting $z_a$, $z_a-cz_b$ with $1\le a\ne b\le m$ and $c$ invertible in $\C[[\hbar]]$.
Denote by
\[\iota_{z_1,\dots,z_m}:\C_*[[z_1,\dots,z_m,\hbar]]\rightarrow \C[[\hbar]]((z_1))\cdots((z_m)),\]
the canonical algebra embedding  that preserves each elements of $\C[[z_1,\dots,z_m,\hbar]]$.

\begin{rem}\label{rem:product}
{\em
Here we make a convention on the products of $\C[[\hbar]]$-valued formal series for later use.
Let $W$ be a topologically free $\C[[\hbar]]$-module,
 $f(z_1,\dots,z_m)\in \C[[\hbar]][[z_1^{\pm 1},\dots, z_m^{\pm 1}]]$ and $x(z_1,\dots,z_m)\in \mathrm{End}_{\C[[\hbar]]}(W)[[z_1^{\pm 1},\dots, z_m^{\pm 1}]]$.
Assume that for any $n\in \Z_+$, the product
\begin{align}\label{eq:productcondition}
f(z_1,\dots,z_m)\cdot \tilde{\pi}^{(m)}_n(x(z_1,\dots,z_m))
\end{align}
exists in $\mathrm{End}_{\C[[\hbar]]}(W/\hbar^n W)[[z_1^{\pm 1},\dots, z_m^{\pm 1}]]$.
Note that
\[\tilde\theta_n^{(m)}(f(z_1,\dots,z_m)\cdot \tilde{\pi}^{(m)}_{n+1}(x(z_1,\dots,z_m)))
=f(z_1,\dots,z_m)\cdot \tilde{\pi}^{(m)}_{n}(x(z_1,\dots,z_m)),\]
where $\tilde\theta_n^{(m)}$ denotes the canonical $\C[[\hbar]]$-map from $\mathrm{End}_{\C[[\hbar]]}(W/\hbar^{n+1} W)[[z_1^{\pm 1},\dots,z_m^{\pm 1}]]$
to $ \mathrm{End}_{\C[[\hbar]]}(W/\hbar^n W)[[z_1^{\pm 1},\dots,z_m^{\pm 1}]]$.
 Then we define the product of $f(z_1,\dots,z_m)$ and $x(z_1,\dots,z_m)$ to be
\[f(z_1,\dots,z_m)x(z_1,\dots,z_m)=\varprojlim_{n>0}f(z_1,\dots,z_m)\cdot \tilde{\pi}^{(m)}_n(x(z_1,\dots,z_m)),\]
which is viewed as an element of $\mathrm{End}_{\C[[\hbar]]}[[z_1^{\pm 1},\dots,z_m^{\pm 1}]]$ via the canonical identification
\[\mathrm{End}_{\C[[\hbar]]}(W)[[z_1^{\pm 1},\dots,z_m^{\pm 1}]]=\varprojlim_{n>0} \mathrm{End}_{\C[[\hbar]]}(W/\hbar^n W)[[z_1^{\pm 1},\dots,z_m^{\pm 1}]].\]
For example, suppose that $W\in \R_\pm^{l}$, $f(z_1,\dots,z_m)\in \C[[\hbar]]((z_1))\cdots((z_m))$ and $x(z_1,\dots,z_m)=x_{i_1}^\pm(z_1)\cdots
x_{i_m}^\pm(z_m)$ for some $i_1,\dots,i_m\in I$.
Then by definition we have
\begin{align*}
  x(z_1,\dots,z_m)v\in W((z_1))\cdots((z_m))+\hbar^nW[[z_1^{\pm 1},\dots,z_m^{\pm 1}]]
\end{align*}
for any $v\in W$ and $n\in\Z_+$.
It then follows that
\[\tilde{\pi}^{(m)}_n(x(z_1,\dots,z_m))\in \mathrm{End}_{\C[[\hbar]]}(W/\hbar^n W)((z_1))\cdots((z_m)).\]
This implies that the product \eqref{eq:productcondition} exists, by noting that $\mathrm{End}_{\C[[\hbar]]}(W/\hbar^n W)((z_1))\cdots((z_m))$ is naturally a $\C[[\hbar]]((z_1))\cdots((z_m))$-module.

}
\end{rem}

\begin{lem}\label{lem:normal-ordering-pre-desc-ii}
Assume that $i,j\in I$ satisfying  $j=\mu^k(i)$ for some $k\in \Z_N$. Then for any $W\in \R_\pm^{l}$ we have
\begin{equation}\begin{split}\label{eq::normal-ordering-pre-desc}
&\iota_{z_1,z_2}(z_1^{d_i}-\xi^{-kd_i}z_2^{d_i})^{-1} F_{ij}^\pm(z_1,z_2)x_i^\pm(z_1)x_j^\pm(z_2)\\
=\ &\iota_{z_2,z_1}(z_1^{d_i}-\xi^{-kd_i}z_2^{d_i})^{-1} F_{ij}^\pm(z_1,z_2)x_i^\pm(z_1)x_j^\pm(z_2).
\end{split}\end{equation}
In particular, we have $\iota_{z_1,z_2}(z_1^{d_i}-\xi^{-kd_i}z_2^{d_i})^{-1} F_{ij}^\pm(z_1,z_2)x_i^\pm(z_1)x_j^\pm(z_2)\in \E_\hbar^{(2)}(W)$.
\end{lem}

\begin{proof}
Set $x_{ij}^\pm(z_1,z_2)=F_{ij}^\pm(z_1,z_2)x_i^\pm(z_1)x_j^\pm(z_2)$.
Recall from  Lemma \ref{lem:normal-ordering-pre-desc} that $x_{ij}^\pm(z_1,z_2)\in  \E_\hbar^{(2)}(W)$.
This together with  Remark \ref{rem:product} implies that the products in \eqref{eq::normal-ordering-pre-desc} are well-defined.

Assume now that $i=j$. Recall from \eqref{Fiiep} and \eqref{Giiep} that
\begin{align*}
  F_{ii}^\pm(z_1,z_2)=\(z_1^{d_i}-q_i^{2d_i}z_2^{d_i}\)\(z_1^{d_i}+q_i^{-d_i}z_2^{d_i}\)^{s_i-1}
  =-G_{ii}^\pm(z_2,z_1).
\end{align*}
This together with (Q8) gives
\begin{align*}
  x_{ii}^\pm(z_1,z_2)=G_{ii}^\pm(z_1,z_2)x_i^\pm(z_2)x_i^\pm(z_1)=-F_{ii}^\pm(z_2,z_1)x_i^\pm(z_2)x_i^\pm(z_1)=-x_{ii}^\pm(z_2,z_1).
\end{align*}
So we have $x_{ii}^\pm(z_1,z_1)=0$ on $W$.
 Note that for $k\in \Z_{d_i}$, we have
 $F_{ii}^\pm(z_1,\xi_{d_i}^kz_2)=F_{ii}^\pm(z_1,z_2)$ and   $x_i^\pm(\xi_{d_i}^kz_1)=x_i^\pm(z_1)$ (see (Q0)).
This implies  that for every $k\in \Z_{d_i}$,
\begin{align}\label{eq:xiidi}
x_{ii}^\pm(z_1,\xi_{d_i}^kz_1)=x_{ii}^\pm(z_1,z_1)=0\quad \te{on}\ W.
\end{align}
In view of this, for any $v\in W$, we have
\begin{align*}
(\iota_{z_1,z_2}(z_1^{d_i}-z_2^{d_i})^{-1}-\iota_{z_2,z_1}(z_1^{d_i}-z_2^{d_i})^{-1})x_{ii}^\pm(z_1,z_2)v
=z_1^{-d_i}\delta(z_1^{d_i}/z_2^{d_i})x_{ii}^\pm(z_1,z_2)v=0.
\end{align*}
This proves the lemma for the case that  $j=i$.

For the general case, noting that
\begin{align*}
  F_{ij}^\pm(z,w)=&\prod_{p\in\Z_N}(z-\xi^pq_i^{\pm a_{i\mu^p(j)}}w)
  =\prod_{p\in\Z_N}(z-\xi^pq_i^{\pm a_{i\mu^{p+k}(i)}}w)
  =F_{ii}^\pm(z,\xi^{-k}w).
\end{align*}
Combing this with (Q0) we have that
\begin{align*}
  &\iota_{z_1,z_2}(z_1^{d_i}-\xi^{-kd_i}z_2^{d_i})\inv F_{ij}^\pm(z_1,z_2)x_i^\pm(z_1)x_j^\pm(z_2)\\
  =&\iota_{z_1,z_2}(z_1^{d_i}-\xi^{-kd_i}z_2^{d_i})\inv F_{ii}^\pm(z_1,\xi^{-k}z_2)x_i^\pm(z_1)x_i^\pm(\xi^{-k}z_2)\\
  =&\iota_{z_2,z_1}(z_1^{d_i}-\xi^{-kd_i}z_2^{d_i})\inv F_{ii}^\pm(z_1,\xi^{-k}z_2)x_i^\pm(z_1)x_i^\pm(\xi^{-k}z_2)\\
  =&\iota_{z_2,z_1}(z_1^{d_i}-\xi^{-kd_i}z_2^{d_i})\inv F_{ij}^\pm(z_1,z_2)x_i^\pm(z_1)x_j^\pm(z_2).
\end{align*}
This completes the proof of the lemma.
\end{proof}

For $i,j\in I$, set
\begin{align}
f_{ij}^\pm(z,w)=\prod_{k\in\Z_N,a_{i\mu^k(j)}>0}\(z-\xi^k w\)^{-1}\cdot F_{ij}^\pm(z,w)\in \C_*[[z,w,\hbar]].
\end{align}
Now we introduce a notion of ``normal ordered products'' of currents on $\ul^\pm$.

\begin{de}\label{de:normal-ordering}
For $W\in \R_\pm^l$ and $i_1,\dots,i_m\in I$, we define a normal ordered product
\begin{align*}
\:x_{i_1}^\pm(z_1)x_{i_2}^\pm(z_2)\cdots x_{i_m}^\pm(z_m)\;\in \mathrm{End}_{\C[[\hbar]]}(W)[[z_1^{\pm 1},z_2^{\pm 1},\dots,z_m^{\pm 1}]]
\end{align*}
 of the currents $x_{i_1}^\pm(z_1),\dots,x_{i_m}^\pm(z_m)$
 to be
\begin{align*}
\prod_{1\le r<s\le m}\iota_{z_r,z_s}\(f_{i_ri_s}^\pm(z_r,z_s)\)
 x_{i_1}^\pm(z_1)x_{i_2}^\pm(z_2)\cdots x_{i_m}^\pm(z_m).
\end{align*}
\end{de}

 Remark \ref{rem:product} implies that the above product is well-defined.
For $i,j\in I$, set
\begin{align}\label{eq:def-C-ij}
  C_{ij}=\prod_{k\in \Z_N,a_{i\mu^k(j)}< 0}(-\xi^k).
\end{align}
We have the following properties of the normal order products.

\begin{prop}\label{prop:normal-ordering-rational}
Let $W\in \R_\pm^l$ and $i_1,i_2,\dots,i_m\in I$.  Then we have
\begin{align}
\:x_{i_1}^\pm(z_1)x_{i_2}^\pm(z_2)\cdots x_{i_m}^\pm(z_m)\;\in\E_\hbar^{(m)}(W),
\end{align}
and for $k_1,\dots,k_m\in \Z_N$,
\begin{align}\label{normalordermu}
\:x_{\mu^{k_1}(i_1)}^\pm(z_1)\cdots x_{\mu^{k_m}(i_m)}^\pm(z_m)\;
=\:x_{i_1}^\pm(\xi^{-k_1}z_1)\cdots x_{i_m}^\pm(\xi^{-k_m}z_m)\;.
\end{align}
Furthermore, for $\sigma\in S_m$, we have \begin{align}\label{normalordersigma}
  \:x_{i_1}^\pm(z_1)\cdots x_{i_m}^\pm(z_m)\;=
  \(\prod_{\substack{ 1\le s<t\le m\\ \sigma(s)>\sigma(t)}}
  C_{i_si_t}\)
  \:x_{i_{\sigma(1)}}^\pm(z_{\sigma(1)})\cdots
  x_{i_{\sigma(m)}}^\pm(z_{\sigma(m)})\;.
\end{align}
\end{prop}
\begin{proof}
We first treat the case that $m=2$.
From Lemma \ref{lem:normal-ordering-pre-desc} and Lemma \ref{lem:normal-ordering-pre-desc-ii}, it follows
that $\:x_{i_1}^\pm(z_1)x_{i_2}^\pm(z_2)\;\in \E_\hbar^{(2)}(W)$.
Note that the identity \eqref{normalordermu} follows directly from definition and the relations (Q0).
For the third assertion, for $i,j\in I$ we have
\begin{align}
 G_{ij}^\pm(z_1,z_2)=\prod_{k\in \Gamma_{ij}}(-\xi^k)\cdot F_{ji}^\pm(z_2,z_1)
 =\prod_{k\in \Z_N, a_{i\mu^{k}(j)}>0}(-\xi^k)\cdot C_{ij}\cdot F_{ji}^\pm(z_2,z_1).\label{eq:f_ij-property1}
\end{align}
This together with Lemma \ref{lem:normal-ordering-pre-desc-ii} and the relations (Q8) gives that
\begin{align*}
&\iota_{z_1,z_2}\(f_{ij}^\pm(z_1,z_2)\) x_i^{\pm}(z_1)x_j^{\pm}(z_2)\\
=\ &\iota_{z_1,z_2}\(\prod_{k\in\Z_N,a_{i\mu^k(j)}>0}\(z_1-\xi^kz_2\)^{-1}\)F_{ij}^\pm(z_1,z_2)x_i^{\pm}(z_1)x_j^{\pm}(z_2)\\
=\ &\iota_{z_2,z_1}\(\prod_{k\in\Z_N,a_{i\mu^k(j)}>0}\(z_1-\xi^kz_2\)^{-1}\)F_{ij}^\pm(z_1,z_2)x_i^{\pm}(z_1)x_j^{\pm}(z_2)\\
=\ &\iota_{z_2,z_1}\(\prod_{k\in\Z_N,a_{i\mu^k(j)}>0}\(z_1-\xi^kz_2\)^{-1}\)G_{ij}^\pm(z_1,z_2)x_j^{\pm}(z_2)x_i^{\pm}(z_1)\\
=\ &C_{ij}\, \iota_{z_2,z_1}\(\prod_{k\in\Z_N,a_{j\mu^k(i)}>0}\(z_2-\xi^kz_1\)^{-1}\) F_{ji}^\pm(z_2,z_1)x_j^\pm(z_2)x_i^\pm(z_1)\\
=\ &C_{ij}\,\iota_{z_2,z_1}\(f_{ji}^\pm(z_2,z_1)\) x_j^{\pm}(z_2)x_i^{\pm}(z_1).
\end{align*}
Thus we have $\:x_{i}^\pm(z_1)x_{j}^\pm(z_2)\;=C_{ij}\:x_{j}^\pm(z_2)x_{i}^\pm(z_1)\;$, as required.

For the general case, we prove it by induction on $m$.
Indeed, for any $v\in W$, $n\in \Z_+$, by induction assumption we have
\[\:x_{i_1}^\pm(z_1)x_{i_2}^\pm(z_2)\cdots x_{i_m}^\pm(z_m)\;v\in W((z_m))((z_1,\dots,z_{m-1}))+\hbar^nW[[z_1^{\pm 1},\dots,z_m^{\pm 1}]].\]
On the other hand, by using the fact ($k=1,\dots,m-1$)
\begin{align*}
\iota_{z_k,z_m}\(f_{i_ki_m}^\pm(z_k,z_m)\) x_{i_k}^{\pm}(z_k)x_{i_m}^{\pm}(z_m)
=C_{i_ki_m}\,\iota_{z_m,z_k}\(f_{i_mi_k}^\pm(z_m,z_k)\) x_{i_m}^{\pm}(z_m)x_{i_k}^{\pm}(z_k),
\end{align*} we can move the term $x_{i_m}^{\pm}(z_m)$ in
$\:x_{i_1}^\pm(z_1)\cdots x_{i_m}^\pm(z_m)\;$ to the left so that
\[\:x_{i_1}^\pm(z_1)x_{i_2}^\pm(z_2)\cdots x_{i_m}^\pm(z_m)\;v\in W((z_m))((z_1,\dots,z_{m-1}))+\hbar^nW[[z_1^{\pm 1},\dots,z_m^{\pm 1}]].\]
This gives that $\:x_{i_1}^\pm(z_1)\cdots x_{i_m}^\pm(z_m)\;\in \E_\hbar^{(m)}(W)$.
Similarly,
the remaining two assertions in the proposition follow by 
an induction argument.
\end{proof}

\begin{rem}\label{rem:norord}{\em From the proof of Proposition \ref{prop:normal-ordering-rational}, we see that in the definition of the
normal order product
$\:x_{i_1}^\pm(z_1)\cdots x_{i_m}^\pm(z_m)\;$, the iota-maps  $\iota_{i_r,i_s}$ can be replaced with
$\iota_{i_s,i_r}$ for  $1\le r<s\le m$. That is,
\[\:x_{i_1}^\pm(z_1)\cdots x_{i_m}^\pm(z_m)\;=\prod_{1\le r<s\le m}f_{i_ri_s}^\pm(z_r,z_s)
 x_{i_1}^\pm(z_1)\cdots x_{i_m}^\pm(z_m),\]
 which is independent from the expansions of $f_{i_ri_s}^\pm(z_r,z_s)$.
}\end{rem}

\subsection{On the relations (Q9)}\label{subsec:qptar}
 This subsection is devoted  to prove the following result:

\begin{prop}\label{prop:iii-normal-ordering}
Let $W\in \R_\pm^l$ and $i\in I$ with $s_i=2$. Then as operators on $W$,
\begin{align}\label{eq:req9}\sum_{\sigma\in S_{3}}
    p_i^\pm(z_{\sigma(1)},z_{\sigma(2)},z_{\sigma(3)})\,
                x_i^\pm(z_{\sigma(1)})x_i^\pm(z_{\sigma(2)})x_i^\pm(z_{\sigma(3)})=0
\end{align} if and only if
\begin{align}
&\:x_i^\pm(z)x_i^\pm(q_i^{2}z)
    x_i^\pm(\xi_{2d_{i}}q_iz)\;=0.\label{eq:construct-of-no-tttt3}
\end{align}
\end{prop}

To prove Proposition \ref{prop:iii-normal-ordering} and for later use, we need the following notion.

\begin{de}\label{de:quan-comm}
For $W\in \R_\pm^l$ and $i_1,\dots,i_m\in I$, we define the $g$-commutator
\[[x_{i_1}^\pm(z_1),\dots,x_{i_m}^\pm(z_m)]_g\in \mathrm{End}_{\C[[\hbar]]}(W)[[z_1^{\pm 1},\dots,z_m^{\pm 1}]]\] inductively such that
$[x_{i_m}^\pm(z_m)]_g=x_{i_m}^\pm(z_m)$ and for $r=m-1,\dots,1$,
\begin{eqnarray}\label{eq:def-g-commutator}
\begin{split}
&  [x_{i_r}^\pm(z_r),\dots,x_{i_m}^\pm(z_m)]_g=x_{i_r}^\pm(z_r)[x_{i_{r+1}}^\pm(z_{r+1}),\dots,x_{i_m}^\pm(z_m)]_g\\
&  \quad-\(\prod_{r+1\le a\le m}g_{i_ai_r}(z_r/z_a)^{\mp 1}\)
    [x_{i_{r+1}}^\pm(z_{r+1}),\dots,x_{i_m}^\pm(z_m)]_gx_{i_r}^\pm(z_r).
\end{split}
\end{eqnarray}
\end{de}

\begin{rem}{\em Here we show that the $g$-commutator is well-defined by using induction  on $m$.
From the definition of $g_{ij}(z)$ (see \eqref{e:g}), it follows that
\begin{align*}
  g_{ij}(w/z)^{\pm 1}\in\C[[\hbar,w/z]]\subset \C[[\hbar]][z,z\inv][[w]].
\end{align*}
This  implies that
\begin{align*}
  \prod_{2\le a\le m}g_{i_ai_r}(z_1/z_a)^{\mp 1}\in \C[[\hbar]][z_{2}^{\pm 1},\dots,z_m^{\pm 1}][[z_1]].
\end{align*}
By induction hypothesis we have a well-defined formal series $[x_{i_2}^\pm(z_2),\dots,x_{i_m}^\pm(z_m)]_g$. And,
for any $v\in W$ and $n\in \Z$, we have
\[[x_{i_2}^\pm(z_2),\dots,x_{i_m}^\pm(z_m)]_g x_{i_1}^\pm(z_1)v\in W[[z_{2}^{\pm 1},\dots,z_m^{\pm 1}]]((z_1))+\hbar^nW[[z_{1}^{\pm 1},\dots,z_m^{\pm 1}]].\]
This means that
\[\tilde\pi_n^{(m)}([x_{i_2}^\pm(z_2),\dots,x_{i_m}^\pm(z_m)]_g x_{i_1}^\pm(z_1))\in \mathrm{End}_{\C[[\hbar]]}(W/\hbar^n W)[[z_{2}^{\pm 1},\dots,z_m^{\pm 1}]]((z_1)).\]
Thus for any $n\in \Z_+$ the product
\[\prod_{2\le a\le m}g_{i_ai_r}(z_1/z_a)^{\mp 1}\cdot \tilde\pi_n^{(m)}([x_{i_2}^\pm(z_2),\dots,x_{i_m}^\pm(z_m)]_g x_{i_1}^\pm(z_1))\]
exists. Then  Remark  \ref{rem:product} implies that the $g$-commutator is well-defined. }
\end{rem}

Recall from \eqref{e:g} that for $i,j\in I$,
\begin{align}\label{gijex}
g_{ji}(z/w)^{\mp 1}=\iota_{w,z}\(\frac{G^\pm_{ij}(z,w)}{F^\pm_{ij}(z,w)}\)
=C_{ij}\,\iota_{w,z}\(\frac{f_{ji}^\pm(w,z)}{f_{ij}^\pm(z,w)}\).
\end{align}
This  implies that
\begin{align}\label{gcomm1}
[x_i^\pm(z),x_j^\pm(w)]_g=
\:x_i^\pm(z)x_j^\pm(w)\;\(\iota_{z,w}(f_{ij}^\pm(z,w)^{-1})-\iota_{w,z}(f_{ij}^\pm(z,w)^{-1})\),
\end{align}
which is an $\E_\hbar^{(2)}(W)$-linear combination of $\delta$-functions.
In general, the $g$-commutator
$[x_{i_1}^\pm(z_1),\dots,x_{i_m}^\pm(z_m)]_g$
is also an $\E_\hbar^{(m)}(W)$-linear combination of  products of $\delta$-functions
(along with their partial differentials).

The following
two elementary lemmas will be used  later on.

\begin{lem}\cite{CTW} Let $c_1,\dots,c_n,d_1,\dots,d_m$ be distinct invertible elements in $\C[[\hbar]]$. Then
\begin{align}\label{deltadec}
&(\iota_{z,w}-\iota_{w,z})\(\prod_{i=1}^n(z-c_iw)^{-1} \prod_{j=1}^m(z-d_iw)^{-2}\)\nonumber\\
=&\sum_{i=1}^n \lim_{z\to c_iw}\(\prod_{a\ne i}(z-c_aw)^{-1} \prod_{j=1}^m(z-d_iw)^{-2}\) z\inv\delta\(\frac{c_iw}{z}\)\\
+&\sum_{j=1}^m \lim_{z\to d_jw}\(\prod_{i=1}^n(z-c_iw)^{-1} \prod_{b\ne j}(z-d_bw)^{-2}\)
\frac{1}{d_j}\frac{\partial}{\partial w}z\inv\delta\(\frac{d_jw}{z}\)
\nonumber\\
-&\sum_{j=1}^m \lim_{z\to d_jw}
\frac{\partial}{\partial w}
\(\prod_{i=1}^n(z-c_iw)^{-1} \prod_{b\ne j}(z-d_bw)^{-2}\)
\frac{1}{d_j}z\inv\delta\(\frac{d_jw}{z}\).
\nonumber
\end{align}
\end{lem}

\begin{lem}\cite{LL}\label{lem:delta-function-independent}
Let $m,s\in\Z_+$ and let $((c_{1i})_{i=1}^s,(n_{1i})_{i=1}^s),\dots,
((c_{mi})_{i=1}^s,(n_{mi})_{i=1}^s)$ be
elements in $(\C[[\hbar]]\setminus\{0\})^s\times\N^s$ satisfying
\[((c_{ji})_{i=1}^s,(n_{ji})_{i=1}^s)\ne( (c_{ki})_{i=1}^s,(n_{ki})_{i=1}^s)\]
for all $1\le j\ne k\le m$.
Then for any $f_1(z),f_2(z),\dots,f_m(z)\in\E_\hbar(W)$, we have
\begin{align*}
  \sum_{i=1}^m f_i(z)\prod_{j=1}^s\(\frac{\partial}{\partial z}\)^{n_{ij}}
  z_j\inv \delta\(\frac{c_{ij}z}{z_j}\)=0
\end{align*}
if and only if $f_i(z)=0$ for all $i=1,2,\dots,m$, where $z,z_1,\dots,z_s$ are  mutually commuting independent formal variables.
\end{lem}

Now we calculate the $g$-commutators that are related to  (Q9).

\begin{lem}\label{lem:giii} For $W\in \R_\pm^l$ and $i\in I$ with $s_i=2$, we have
\begin{align*}
&[x_i^\pm(z_1),x_i^\pm(z_2),x_i^\pm(z_3)]_g=\:x_i^\pm(z_1)x_i^\pm(z_2) x_i^\pm(z_3)\;\\
&\cdot\Bigg(\frac{1+q_i^{\mp d_i}}{(1+q_i^{\pm3d_i})(1+q_i^{\mp 3d_i})}z_1^{-2d_i}\delta\(\frac{-q_i^{\pm d_i}z_3^{d_i}}{z_1^{d_i}}\)z_2^{-d_i}\delta\(\frac{q_i^{\pm 2d_i}z_3^{d_i}}{z_2^{d_i}}\)\\
&+\frac{(1-q_i^{\mp 2d_i})(1-q_i^{\mp 4d_i})}{(1+q_i^{\mp 3d_i})(1+q_i^{\mp3d_i})(1+q_i^{\mp5d_i})}
 z_1^{-2d_i}\delta\(\frac{q^{\pm 4d_i}z_3^{d_i}}{z_1^{d_i}}\)z_2^{-d_i}\delta\(\frac{q_i^{\pm 2d_i}z_3^{d_i}}{z_2^{d_i}}\)\\
&+\frac{-q_i^{\mp 2d_i}(1+q_i^{\pm d_i})}{(1+q_i^{\mp 3d_i})(1+q^{\pm5d_i})}z_1^{-2d_i}\delta\(\frac{-q_i^{\mp d_i}z_3^{d_i}}{z_1^{d_i}}\)z_2^{-d_i}\delta\(\frac{q_i^{\pm 2d_i}z_3^{d_i}}{z_2^{d_i}}\)\\
&+\frac{1+q_i^{\mp d_i}}{(1+q_i^{\mp3d_i})(1+q_i^{\pm 3d_i})}z_1^{-2d_i}\delta\(\frac{-q_i^{\pm d_i}z_3^{d_i}}{z_1^{d_i}}\)z_2^{-d_i}\delta\(\frac{-q_i^{\mp d_i}z_3^{d_i}}{z_2^{d_i}}\)\\
&+\frac{(1+q_i^{\pm d_i})(1-q_i^{\pm 2d_i})}{(1+q_i^{\pm 3d_i})(1+q_i^{\pm 3d_i})(1-q_i^{\pm 4d_i})}
 z_1^{-2d_i}\delta\(\frac{q_i^{\mp2 d_i}z_3^{d_i}}{z_1^{d_i}}\)z_2^{-d_i}\delta\(\frac{-q_i^{\mp d_i}z_3^{d_i}}{z_2^{d_i}}\)\\
&+\frac{1+q_i^{\pm d_i}}{(1+q_i^{\pm 3d_i})(1+q_i^{\mp d_i})(1+q_i^{\mp 2d_i})}
 z_1^{-2d_i}\delta\(\frac{q_i^{\pm2 d_i}z_3^{d_i}}{z_1^{d_i}}\)z_2^{-d_i}\delta\(\frac{-q_i^{\mp d_i}z_3^{d_i}}{z_2^{d_i}}\)\Bigg).
\end{align*}
\end{lem}
\begin{proof}Firstly, it follows from \eqref{gcomm1} and \eqref{deltadec} that
\begin{align*}
 &[x_i^\pm(z_2),x_i^\pm(z_3)]_g\\
 =&\:x_i^\pm(z_2) x_i^\pm(z_3)\;(\iota_{z_2,z_3}-\iota_{z_3,z_2})
 \(\frac{z_2^{d_i}-z_3^{d_i}}{(z_2^{d_i}+q_i^{\mp d_i}z_3^{d_i})(z_2^{d_i}-q_i^{\pm 2d_i}z_3^{d_i})}\)\\
 =&\:x_i^\pm(z_2) x_i^\pm(z_3)\;\(\frac{1-q_i^{\mp 2d_i}}{1+q_i^{\mp 3d_i}}z_2^{-d_i}\delta\(\frac{q_i^{\pm 2d_i}z_3^{d_i}}{z_2^{d_i}}\)
 +\frac{1+q_i^{\pm d_i}}{1+q_i^{\pm 3d_i}}z_2^{-d_i}\delta\(\frac{-q_i^{\mp d_i}z_3^{d_i}}{z_2^{d_i}}\)\).
 \end{align*}
 Then by definition we have
 \begin{align*}
 &[x_i^\pm(z_1),x_i^\pm(z_2),x_i^\pm(z_3)]_g
 =(\iota_{z_3,z_2,z_1}-\iota_{z_1,z_2,z_3})
 \(f_{ii}^\pm(z_1,z_2)^{-1}f_{ii}^\pm(z_1,z_3)^{-1}\)\\
 &\:x_i^\pm(z_1)x_i^\pm(z_2) x_i^\pm(z_3)\;\(\frac{1-q_i^{\mp 2d_i}}{1+q_i^{\mp 3d_i}}z_2^{-d_i}\delta\(\frac{q_i^{\pm 2d_i}z_3^{d_i}}{z_2^{d_i}}\)
 +\frac{1+q_i^{\pm d_i}}{1+q_i^{\pm 3d_i}}z_2^{-d_i}\delta\(\frac{-q_i^{\mp d_i}z_3^{d_i}}{z_2^{d_i}}\)\).
 \end{align*}

Now the assertion is implied by the following two facts, which can be proved directly by using \eqref{deltadec}:
 \begin{align*}
 &(\iota_{z_1,z_3}-\iota_{z_3,z_1})\(\lim_{z_2^{d_i}\rightarrow q_i^{\pm 2d_i}z_3^{d_i}}
 \frac{1}{f_{ii}^{\pm}(z_1,z_2)f_{ii}^{\pm}(z_1,z_3)}\)\\
 =&(\iota_{z_1,z_3}-\iota_{z_3,z_1})
 \(\frac{z_1^{d_i}-z_3^{d_i}}
 {(z_1^{d_i}+q_i^{\pm d_i}z_3^{d_i})(z_1^{d_i}-q^{\pm 4d_i}z_3^{d_i})(z_1^{d_i}+q_i^{\mp d_i}z_3^{d_i})}\)\\
 =&\frac{1+q_i^{\mp d_i}}{(1+q_i^{\pm3d_i})(1-q_i^{\mp2d_i})}z_1^{-2d_i}\delta\(\frac{-q_i^{\pm d_i}z_3^{d_i}}{z_1^{d_i}}\)
 +\frac{1-q_i^{\mp 4d_i}}{(1+q_i^{\mp3d_i})(1+q_i^{\mp5d_i})}\\
 &z_1^{-2d_i}\delta\(\frac{q^{\pm 4d_i}z_3^{d_i}}{z_1^{d_i}}\)
 +\frac{1+q_i^{\pm d_i}}{(1-q^{\pm2d_i})(1+q^{\pm5d_i})}z_1^{-2d_i}\delta\(\frac{-q_i^{\mp d_i}z_3^{d_i}}{z_1^{d_i}}\).
 \end{align*}
 and
 \begin{align*}
 &(\iota_{z_1,z_3}-\iota_{z_3,z_1})(\lim_{z_2^{d_i}\rightarrow -q_i^{\mp d_i}z_3^{d_i}}
 f_{ii}^{\pm}(z_1,z_2)^{-1}f_{ii}^{\pm}(z_1,z_3)^{-1})\\
 =&(\iota_{z_1,z_3}-\iota_{z_3,z_1})
 \(\frac{z_1^{d_i}-z_3^{d_i}}
 {(z_1^{d_i}+q_i^{\pm d_i}z_3^{d_i})(z_1^{d_i}-q_i^{\mp 2d_i}z_3^{d_i})(z_1^{d_i}-q_i^{\pm 2d_i}z_3^{d_i})}\)\\
 =&\frac{1+q_i^{\mp d_i}}{(1+q_i^{\mp3d_i})(1+q_i^{\pm d_i})}z_1^{-2d_i}\delta\(\frac{-q_i^{\pm d_i}z_3^{d_i}}{z_1^{d_i}}\)
  +\frac{1-q_i^{\pm 2d_i}}{(1+q_i^{\pm 3d_i})(1-q_i^{\pm 4d_i})}\\
 &z_1^{-2d_i}\delta\(\frac{q_i^{\mp2 d_i}z_3^{d_i}}{z_1^{d_i}}\)
 +\frac{1-q_i^{\mp 2d_i}}{(1+q_i^{\mp d_i})(1-q_i^{\mp 4d_i})}
 z_1^{-2d_i}\delta\(\frac{q_i^{\pm2 d_i}z_3^{d_i}}{z_1^{d_i}}\).
 \end{align*}
\end{proof}

\textbf{Proof of Proposition \ref{prop:iii-normal-ordering}:}
Denote by $R_{iii}^\pm$ the LHS of \eqref{eq:req9}.
Note that for  $\sigma\in S_3$, we can write $x_i^{\pm}(z_{\sigma(1)})x_i^{\pm} (z_{\sigma(2)})x_i^{\pm}(z_{\sigma(3)})\in \mathrm{End}(W)[[z_1,z_2,z_3]]$
as a summation of the $g$-commutators and the currents $x_i^\pm(z_3)x_i^\pm(z_2)x_i^\pm(z_1)$.
 For example,
 \begin{align*}
 &\quad x_i^\pm(z_1)x_i^\pm(z_2)x_i^\pm(z_3)\\
 &=[x_i^\pm(z_1),x_i^\pm(z_2),x_i^\pm(z_3)]_g+g_{ii}(z_1/z_3)^{\mp 1}g_{ii}(z_1/z_2)^{\mp 1}
 [x_i^\pm(z_2),x_i^\pm(z_3)]_g x_i^\pm(z_1)\\
 &+g_{ii}(z_2/z_3)^{\mp 1}[x_i^\pm(z_1),x_i^\pm(z_3)]_gx_i^\pm(z_2)
 +g_{ii}(z_2/z_3)^{\mp 1} g_{ii}(z_1/z_3)^{\mp 1}x_i^\pm(z_3)\\
 &\cdot[x_i^\pm(z_1),x_i^\pm(z_2)]_g
 +g_{ii}(z_2/z_3)^{\mp 1} g_{ii}(z_1/z_3)^{\mp 1} g_{ii}(z_1/z_2)^{\mp 1} x_i^\pm(z_3)x_i^\pm(z_2)x_i^\pm(z_1).
 \end{align*}
 Then it is straightforward to see that $R_{iii}^\pm$ can be rewritten as (recalling \eqref{gijex})
\begin{align*}
&p_i^\pm(z_1,z_2,z_3)[x_i^\pm(z_1),x_i^\pm(z_2),x_i^\pm(z_3)]_g
  +p_i^\pm(z_2,z_1,z_3)[x_i^\pm(z_2),x_i^\pm(z_1),x_i^\pm(z_3)]_g\\
  +&\iota_{z_3,z_2,z_1}\(\frac{A_i^\pm(z_1,z_2,z_3)}{F_{ii}^{\pm}(z_1,z_3)F_{ii}^\pm(z_2,z_3)}\)x_i^\pm(z_3)[x_i^\pm(z_1),x_i^\pm(z_2)]_g \\
  +&\iota_{z_3,z_1,z_2}\(\frac{A_i^\pm(z_1,z_3,z_2)}{F_{ii}^{\pm}(z_2,z_3)G_{ii}^\pm(z_2,z_1)}\)[x_i^\pm(z_1),x_i^\pm(z_3)]_gx_i^\pm(z_2)\\
  +&\iota_{z_3,z_2,z_1}\(\frac{A_i^\pm(z_2,z_3,z_1)}{F_{ii}^{\pm}(z_1,z_3)F_{ii}^\pm(z_1,z_2)}\)[x_i^\pm(z_2),x_i^\pm(z_3)]_gx_i^\pm(z_1)\\
  +&\iota_{z_3,z_2,z_1}\(
  \frac{G_{ii}^{\pm}(z_1,z_2)A_i^\pm(z_1,z_2,z_3)+F_{ii}^{\pm}(z_1,z_2)A_i^\pm(z_2,z_1,z_3)}{
  F_{ii}^{\pm}(z_1,z_2)F_{ii}^{\pm}(z_1,z_3)F_{ii}^\pm(z_2,z_3)}\)
 x_i^\pm(z_3)x_i^\pm(z_2)x_i^\pm(z_1),
\end{align*}
where
\begin{align*}
&A^\pm_i(z_1,z_2,z_3)=G_{ii}^\pm(z_1,z_3)G_{ii}^\pm(z_2,z_3) p_i^\pm(z_1,z_2,z_3)\\
&+G_{ii}^\pm(z_1,z_3)F_{ii}^\pm(z_2,z_3) p_i^\pm(z_1,z_3,z_2)
+F_{ii}^\pm(z_1,z_3)F_{ii}^\pm(z_2,z_3) p_i^\pm(z_3,z_1,z_2).
\end{align*}
are as in \eqref{defai123} (noting that $F_{ii}^{\pm}(z,w)=-G_{ii}^\pm(w,z)$).

Recall from \eqref{tri-lem-2} that $A^\pm_i(z_1,z_2,z_3)=F_{ii}^\pm(z_1,z_2)B_i^\pm(z_1,z_2,z_3)$, where $B_i^\pm(z_1,z_2,z_3)$
is a polynomial satisfying the symmetry:
$B_i^\pm(z_1,z_2,z_3)=B_i^\pm(z_2,z_1,z_3)$.
This implies that $G_{ii}^{\pm}(z_1,z_2)A_i^\pm(z_1,z_2,z_3)+F_{ii}^{\pm}(z_1,z_2)A_i^\pm(z_2,z_1,z_3)=0$.
Furthermore, since $F_{ii}^\pm(z_1,z_2)$ $[x_i^\pm(z_1),x_i^\pm(z_2)]_g=0$, we obtain
\begin{equation}\begin{split}\label{eq:R_iii-rewrite1}
 R_{iii}^\pm=\sum_{\sigma\in S_2}p_i^\pm(z_{\sigma(1)},z_{\sigma(2)},z_3)[x_i^\pm(z_{\sigma(1)}),x_i^\pm(z_{\sigma(2)}),x_i^\pm(z_3)]_g.
 \end{split}
 \end{equation}
It is straightforward to check that
\begin{align}\label{eq:q9eq1}
\lim_{z_1^{d_i}\rightarrow q^{\pm 4d_i}z_3^{d_i}}
\lim_{z_2^{d_i}\rightarrow q_i^{\pm 2d_i}z_3^{d_i}} p_i^\pm(z_1,z_2,z_3)=0,
\end{align}
\begin{equation}\begin{split}\label{eq:q9eq2}
&\lim_{z_1^{d_i}\rightarrow q_i^{\mp d_i}z_3^{d_i}}
\lim_{z_2^{d_i}\rightarrow q_i^{\pm 2d_i}z_3^{d_i}}
\frac{q_i^{\mp 2d_i}p_i^\pm(z_1,z_2,z_3)(1+q_i^{\pm d_i})}{(1+q_i^{\mp 3d_i})(1+q^{\pm5d_i})}z_1^{-d_i}\\
=&q_i^{\mp \frac{7}{2}d_i}\frac{1+q_i^{\pm d_i}}{1+q_i^{\mp 3d_i}}\\
=&\lim_{z_1^{d_i}\rightarrow q_i^{\pm2 d_i}z_3^{d_i}}
\lim_{z_2^{d_i}\rightarrow -q_i^{\mp d_i}z_3^{d_i}}
\frac{p_i^\pm(z_1,z_2,z_3)(1+q_i^{\pm d_i})}{(1+q_i^{\pm 3d_i})(1+q_i^{\mp d_i})(1+q_i^{\mp 2d_i})}
 z_1^{-d_i}.
\end{split}\end{equation}

Recalling Proposition \ref{prop:normal-ordering-rational}, we have
\begin{align}\label{eq:q9eq3}
\:x_i^{\pm}(z_1)x_i^{\pm}(z_2)x_i^{\pm}(z_3)\;=\:x_i^{\pm}(z_2)x_i^{\pm}(z_1)x_i^{\pm}(z_3)\;.
\end{align}

Now, it follows from Lemma \ref{lem:giii}, \eqref{eq:R_iii-rewrite1}, \eqref{eq:q9eq1}, \eqref{eq:q9eq2}
and \eqref{eq:q9eq3} that
 \begin{equation}\begin{split}\label{eq:R_iii-rewrite2}
 &R_{iii}^\pm=\sum_{\sigma\in S_2}p_i^\pm(z_{\sigma(1)},z_{\sigma(2)},z_3)\:x_i^\pm(z_{\sigma(1)})x_i^\pm(z_{\sigma(2)})x_i^\pm(z_3)\;\\
 &\cdot\Bigg(\frac{1+q_i^{\mp d_i}}{(1+q_i^{\pm3d_i})(1+q_i^{\mp 3d_i})}z_{\sigma(1)}^{-2d_i}\delta\(\frac{-q_i^{\pm d_i}z_3^{d_i}}{z_{\sigma(1)}^{d_i}}\)z_{\sigma(2)}^{-d_i}\delta\(\frac{q_i^{\pm 2d_i}z_3^{d_i}}{z_{\sigma(2)}^{d_i}}\)\\
&+\frac{1+q_i^{\mp d_i}}{(1+q_i^{\mp3d_i})(1+q_i^{\pm 3d_i})}z_{\sigma(1)}^{-2d_i}\delta\(\frac{-q_i^{\pm d_i}z_3^{d_i}}{z_{\sigma(1)}^{d_i}}\)z_{\sigma(2)}^{-d_i}\delta\(\frac{-q_i^{\mp d_i}z_3^{d_i}}{z_{\sigma(2)}^{d_i}}\)\\
&+\frac{(1+q_i^{\pm d_i})(1-q_i^{\pm 2d_i})}{(1+q_i^{\pm 3d_i})(1+q_i^{\pm 3d_i})(1-q_i^{\pm 4d_i})}
 z_{\sigma(1)}^{-2d_i}\delta\(\frac{q_i^{\mp2 d_i}z_3^{d_i}}{z_{\sigma(1)}^{d_i}}\)z_{\sigma(2)}^{-d_i}\delta\(\frac{-q_i^{\mp d_i}z_3^{d_i}}{z_{\sigma(2)}^{d_i}}\)\Bigg).
\end{split}\end{equation}
Thus, by using Lemma \ref{lem:delta-function-independent} and
\eqref{eq:R_iii-rewrite2}, we find that $R_{iii}^\pm=0$ on $W$
if and only if
\begin{align*}
&\:x_i^\pm(\xi_{2d_i}q_i^{\pm 1}z_3)x_i^\pm(q_i^{\pm 2}z_3)x_i^\pm(z_3)\;
=\:x_i^\pm(\xi_{2d_i}q_i^{\pm 1}z_3)x_i^\pm(\xi_{2d_i}q_i^{\mp 1}z_3)x_i^\pm(z_3)\;\\
&=\:x_i^\pm(q_i^{\mp2}z_3)x_i^\pm(\xi_{2d_i}q_i^{\mp 1}z_3)x_i^\pm(z_3)\;=0.
\end{align*}
Finally, the proposition follows from the following simple fact:
\begin{align*}
&\:x_i^\pm(z)x_i^\pm(q_i^{2}z)
    x_i^\pm(\xi_{2d_{i}}q_iz)\;
=\lim_{z_3\rightarrow q_i^{1\mp 1}z}\:x_i^\pm(\xi_{2d_i}q_i^{\pm 1}z_3)x_i^\pm(q_i^{\pm 2}z_3)x_i^\pm(z_3)\;\\
=&\lim_{z_3\rightarrow \xi_{2d_i}q_iz}\:x_i^\pm(\xi_{2d_i}q_i^{\pm 1}z_3)x_i^\pm(\xi_{2d_i}q_i^{\mp 1}z_3)x_i^\pm(z_3)\;\\
=&\lim_{z_3\rightarrow q_i^{1\pm 1}z}\:x_i^\pm(q_i^{\mp2}z_3)x_i^\pm(\xi_{2d_i}q_i^{\mp 1}z_3)x_i^\pm(z_3)\;.
\end{align*}

\subsection{On affine quantum Serre relations}\label{subsec:aff-qSerre-rel}
In this subsection we study the  affine quantum Serre relations on restricted $\us^\pm$-modules in a general setting.
As applications, we give a simple characterization of the relations (Q10) and also 
a partial answer of the conjecture  given in Remark \ref{rem:q11}.

Let $(i,j)\in \mathbb I$, $m\in \Z_+$, $f^\pm=f^\pm(z,w)\in \C[[\hbar]][z^{d_i},w^{d_i}]$ be homogenous and $B=(B_0,B_1,\dots,B_m)\in (\C[[\hbar]])^{m+1}$.
Associated to these data, we introduce the currents
 \begin{align*}
  D^\pm_{ij}(m,f^\pm,B)=\sum_{\sigma\in S_m}
  \Bigg\{&\prod_{1\le a<b\le m}f^\pm(z_{\sigma(a)},z_{\sigma(b)})\big(
  \sum_{r=0}^m (-1)^rB_{r} x_i^\pm(z_{\sigma(1)})\\
  &\cdots x_i^\pm(z_{\sigma(r)})x_j^\pm(w)
  x_i^\pm(z_{\sigma(r+1)})\cdots x_i^\pm(z_{\sigma(m)})\big)\Bigg\},
\end{align*}
and the polynomials
\begin{align}\notag
  P_{ij}^\pm(m,f^\pm,B)
  =&\sum_{\sigma\in S_m}(-1)^{|\sigma|} \sum_{r=0}^m (-1)^rB_{r}\prod_{1\le a<b\le m}f^\pm(z_{\sigma(a)},z_{\sigma(b)})G_{ii}^\pm(z_{\sigma(a)},z_{\sigma(b)})
  \\ \label{e:newiden}
 & \times\prod_{a=0}^r G_{ij}^\pm(z_{\sigma(a)},w)
  \prod_{b=r+1}^m F_{ij}^\pm(z_{\sigma(b)},w).
\end{align}
Note that by specializing
\begin{align}\label{mfb1}
m=m_{ij},\ f^\pm(z,w)=p_{ij}^\pm(z,w)\ \te{and}\ B_{r}=\binom{m_{ij}}{r}_{q_i^{d_{ij}}},\ r=0,\dots,m_{ij},
\end{align}
the relations $D^\pm_{ij}(m,f^\pm,B)=0$ are nothing but the relations (Q10).
On the other hand, by specializing
\begin{align}\label{mfb2}
m=\check{m}_{ij},\ f^\pm(z,w)=1\ \te{and}\ B_{r}=\binom{\check{m}_{ij}}{r}_{\check{q}_i},\ r=0,\dots,\check{m}_{ij},
\end{align}
the relations $D^\pm_{ij}(m,f^\pm,B)=0$ are precisely the relations (Q11).

The following result will be proved in Section \ref{sec:pf-prop-aff-q-serre}.

\begin{thm}\label{prop:Dr-to-normal-ordering}
Let  $W$ be a restricted $\us^\pm$-module and let $(i,j)\in\mathbb I$, $m\in \Z_+$, $f^\pm\in \C[[\hbar]][z^{d_i},w^{d_i}]$ be homogenous, $B\in (\C[[\hbar]])^{m+1}$ such that
the polynomial $P_{ij}^\pm(m,f^\pm,B)=0$. If
\begin{align}\label{reforq10}
  \:x_i^\pm(q_i^{a_{ij}}\xi_{d_{ij}}^l w)
    x_i^\pm(q_i^{a_{ij}+2}\xi_{d_{ij}}^l w)
    \cdots x_i^\pm(q_i^{-a_{ij}}\xi_{d_{ij}}^l w)x_j^\pm(w)\;=0\quad{\te{on}}\ W
\end{align}
for all $0\le l <d_{ij}/d_i$,
 then $D_{ij}^\pm(m,f^\pm,B)=0$ on $W$.

Conversely, if $m=m_{ij}$, $f^\pm(w,q_i^{\pm 2n}w)\ne 0$ for $n=1,\dots,-a_{ij}$,
  and $D_{ij}^\pm(m,f^\pm,B)=0$ on $W$, then  \eqref{reforq10} holds for all  $0\le l <d_{ij}/d_i$.
\end{thm}

In view of Theorem \ref{prop:Dr-to-normal-ordering}, we can prove the following main result of this section.

\begin{thm}\label{thm:Dr-equiv-normal-ordering} Let $W$ be a restricted $\ul^\pm$-module. Then the relations (Q9) and (Q10) hold on $W$
if and only if for all $i,j\in I$ with $a_{ij}<0$,
\begin{align}\label{eq:Dr-equiv-normal-ordering-rel}
  \:x_i^\pm(q_i^{a_{ij}}w)x_i^\pm(q_i^{a_{ij}+2}w)\cdots  x_i^\pm(q_i^{-a_{ij}-2}w)
  x_i^\pm(q_i^{-a_{ij}}w)x_j^\pm(w)\;=0\quad \te{on $W$}.
\end{align}
\end{thm}
\begin{proof}
Note that for $i,j\in I$ with $a_{ij}<0$ and $i\in \mathcal O(j)$,
it follows from \eqref{normalordermu} that
 \eqref{eq:construct-of-no-tttt3} is 
 equivalent
to the relations $\:x_i^\pm(q_i^{-1}w)x_i^\pm(q_iw)x_j^\pm(w)\;$ $=0$.
In view of Proposition \ref{prop:iii-normal-ordering}, it remains to prove that for $(i,j)\in \mathbb I$,  the relations (Q10) are equivalent to
\eqref{eq:Dr-equiv-normal-ordering-rel} on $W$.
Notice that now we  may view $W$ as a restricted $\us^\pm$-module.

Let $(i,j)\in \mathbb I$. Note that  for any $k\in \Gamma_{ij}$ (see \eqref{gammaij}), $(\mu^{-k}(i),j)$ also lies in $\mathbb I$.
By replacing $(i,j)$ with $(\mu^k(i),j)$ in \eqref{eq:Dr-equiv-normal-ordering-rel} we obtain
\begin{align*}
\:x_{\mu^{-k}(i)}^\pm(q_{i}^{a_{ij}}w)x_{\mu^{-k}(i)}^\pm(q_{i}^{a_{ij}+2}w)\cdots  x_{\mu^{-k}(i)}^\pm(q_{i}^{-a_{ij}-2}w)
  x_{\mu^{-k}(i)}^\pm(q_i^{-a_{ij}}w)x_j^\pm(w)\;=0\quad \te{on $W$},
  \end{align*}
 where we use the facts $ a_{\mu^{-k}(i)j}=a_{ij}$ and $ q_{\mu^{-k}(i)}=q_i$ (see (LC2) and \eqref{eq:defqi}).
This together with the fact $x_{\mu^{-k}(i)}(w)=x_i(\xi^k w)$ (see (Q0)) gives that
  \begin{align*}
  \:x_i^\pm(q_i^{a_{ij}}\xi^kw)x_i^\pm(q_i^{a_{ij}+2}\xi^kw)\cdots  x_i^\pm(q_i^{-a_{ij}-2}\xi^kw)
  x_i^\pm(q_i^{-a_{ij}}\xi^kw)x_j^\pm(w)\;=0\quad \te{on $W$}.
\end{align*}
Furthermore, it follows from (LC3) that
$\Gamma_{ij}=\{\frac{N}{d_{ij}}l\mid 0\le l<d_{ij}\}$
and hence
\[\{\xi^k\mid k\in \Gamma_{ij}\}=\{\xi^{\frac{N}{d_{ij}}l}=\xi_{d_{ij}}^l\mid 0\le l<d_{ij}\}.\]
Therefore,   the relations \eqref{eq:Dr-equiv-normal-ordering-rel} hold for all $(\mu^{-k}(i),j)$ with $k\in \Gamma_{ij}$ if and only if
 the relations \eqref{reforq10} hold for all $0\le l<d_{ij}$.

Now, take the triple $(m,f^\pm,B)$ as in \eqref{mfb1} so that
 the relations $D^\pm_{ij}(m,f^\pm,B)=0$ are exactly the relations (Q10). In this case,
it was proved in \cite{J-KM} (see also \cite{He-representation-coprod-proof}) that $P^\pm_{ij}(m,f^\pm,B)=0$.
Furthermore, one can easily check that $p_{ij}^\pm(w,q_i^{\pm 2n}w)\ne 0$ for $n=1,\dots,-a_{ij}$.
Then the assertion follows from  Theorem \ref{prop:Dr-to-normal-ordering}.
\end{proof}

Combining  Theorem \ref{prop:Dr-to-normal-ordering} with Theorem \ref{thm:Dr-equiv-normal-ordering}, we immediately get that
\begin{coro}\label{coro:affine-sr}
Let  $W\in \R_\pm$, $(i,j)\in\mathbb I$, $m\in \Z_+$, $f^\pm\in \C[[\hbar]][z^{d_i},w^{d_i}]$ be homogenous and $B\in (\C[[\hbar]])^{m+1}$.
If $P_{ij}^\pm(m,f^\pm,B)=0$, then $D_{ij}^\pm(m,f^\pm,B)=0$ on $W$.
\end{coro}

In particular, we have

\begin{coro}\label{prop:normal-ordering-to-DJ}
Assume that $\fg$ is of finite or affine type. Then the relations (Q11) hold on any restricted $\qtar^\pm$-module.
\end{coro}
\begin{proof} Take $(m,f^\pm,B)$ as in \eqref{mfb2} so that the relations (Q11) are just the relations $D_{ij}^\pm(m,f^\pm,B)=0$.
In view of Corollary \ref{coro:affine-sr}, it suffices to check that
$P_{ij}^\pm(\check{m}_{ij},1,B)=0$, which can be checked directly case by case  (by using Maple for example).
In fact, the polynomial $P_{ij}^\pm(\check{m}_{ij},1,B)$
 depends only on the positive integers $\check{m}_{ij}$, $d_{ij}/d_i$ and $s_i$.
Furthermore,  when $\fg$ is of finite type or affine type, one has $\check{m}_{ij}\le 5$, $d_{ij}/d_i\le 4$ and $s_i\le 2$.
\end{proof}

\begin{rem}{\em We conjecture that $P_{ij}^\pm(m,f^\pm,B)=0$ (see \eqref{e:newiden}) for any triple $(m,f^\pm,B)$ as in \eqref{mfb2}, which depends on the positive integers $\check{m}_{ij}$, $d_{ij}/d_i$ and $s_i$.
When $\mu=\mathrm{id}$ (and so $d_{ij}/d_i=1=s_i$), this combinatorial identity was discovered in \cite{J-KM}.
}\end{rem}

\section{Hopf algebra structure}\label{sec:Hopf}

In this section, we define a Hopf algebra structure on the restricted
completion   $\qptar$ of $\qtar$.
As in the last section, we assume that the linking condition (LC3) holds.

\subsection{Restricted $\qtar$-modules}

In this subsection, we show that the category $\R$ of restricted $\qtar$-modules is a monoidal category.

Let $U$ and $V$ be two  topologically free $\C[[\hbar]]$-modules.
Recall from \eqref{indefem} that  $\E_\hbar(U)\wh\ot \E_\hbar(V)$ is the limit of the inverse system
\begin{align*}
0\leftarrow \E(U/\hbar U)\ot \E(V/\hbar V)
\leftarrow \E(U/\hbar^2 U)\ot \E(V/\hbar^2 V)\leftarrow \cdots,
\end{align*}
noting that $\E(U/\hbar^n U)\ot \E(V/\hbar^n V)\cong \E(V)/\hbar^n\E(V)
\ot \E(U)/\hbar^n\E(U)$.
On the other hand, $\E_\hbar(U\wh\ot V)$ is the limit of the inverse system
\begin{align*}
0\leftarrow \E(U/\hbar U\ot V/\hbar V)
\leftarrow \E(U/\hbar^2 U\ot V/\hbar^2 V)\leftarrow \cdots,
\end{align*}
noting that $\E(U/\hbar^n U\ot V/\hbar^n V)\cong \E(U\ot V/\hbar^n (U\ot V))
\cong \E(U\wh\ot V/\hbar^n (U\wh\ot V))$.
In view of these two inverse limits,  we have
a $\C[[\hbar]]$-map
\begin{align}
\theta_{U,V}=\varprojlim\limits_{n\in \Z_+} \theta_{U/\hbar^n U, V/\hbar^n V}:\E_\hbar(U)\wh\ot \E_\hbar(V)\rightarrow \E_\hbar(U\wh\ot V),\end{align}
where for two vector $\C$-spaces $U^0$ and $V^0$,  $\theta_{U^0,V^0}$ denotes the $\C$-map from
$\E(U^0)\ot \E(V^0)$ to $\E(U^0\ot V^0)$ defined by
\begin{align*}
\sum_{m\in \Z} a_m z^{-m}\ot \sum_{n\in \Z}b_nz^{-n}\mapsto ((u\ot v)\mapsto
\sum_{m\in \Z} (\sum_{k\in \Z}a_k(u)\ot b_{m-k}(v)) z^{-m}).
\end{align*}

 Without confusion, 
 for $a(z)\in \E_\hbar(U)$ and $b(z)\in  \E_\hbar(V)$, we shall denote  $\theta_{U,V}(a(z)\wh\ot b(z))$ by $a(z)\ot b(z)$ for simplicity
 in this  section.

Note that for $W\in \R$ and $i\in I$, one has $x_i^\pm(z), \phi_i^\pm(z)\in \E_\hbar(W)$.
As an application of  Theorem \ref{thm:Dr-equiv-normal-ordering}, we have

\begin{prop}\label{prop:delta} Let $U$ and $V$ be two restricted $\qtar$-modules.
Then there is a restricted $\qtar$-module structure on the $\hbar$-adically
completed tensor product space $U\wh\ot V$  with the action 
$\Delta$ defined by 
 ($i\in I,h\in\check\h$):
\begin{eqnarray*}
&&\te{(Co1)}\quad\Delta \left(x_i^+(z)\right)
    =x_i^+(z)\otimes 1+\phi_i^-(zq^{\frac{c_1}{2}})\otimes x_i^+(z q^{c_1}),\\
&&\te{(Co2)}\quad\Delta\left(x_i^-(z)\right)
    =1\otimes x_i^-(z)+x_i^-(zq^{c_2})\otimes \phi_i^+(z q^{\frac{c_2}{2}}),\\
&&\te{(Co3)}\quad\Delta\left(\phi_i^\pm(z)\right)
    =\phi_i^\pm(zq^{\pm\frac{c_2}{2}})\otimes \phi_i^\pm(z q^{\mp\frac{c_1}{2}}),\\
&&\te{(Co4)}\quad\Delta(h)=h\ot 1+1\ot h,\quad
\Delta(c)=c_1+c_2,\quad \Delta(d)=d\ot 1+1\ot d,
\end{eqnarray*}
where $c_1=c\ot 1$ and $c_2=1\otimes c$.
\end{prop}
\begin{proof}
We need to show that the action $\Delta$ is compatible with the relations (Q0)-(Q10).
As in the untwisted case (\cite[Theorem 2.1]{DI-generalization-qaff},\cite[Proposition 29]{He-representation-coprod-proof}), one
can easily check that $\Delta$ is compatible with the relations (Q0)-(Q8), and we omit the details.
Then $U\wh\ot V$ is naturally a restricted $\ul^\pm$-module.
Due to Theorem \ref{thm:Dr-equiv-normal-ordering}, it suffices to
 check
the compatibility between $\Delta$ and the relations \eqref{eq:Dr-equiv-normal-ordering-rel}.

Let $i,j\in I$ with $a_{ij}<0$. Set $m=1-a_{ij}$ and $\overline m=\{1,2,\dots,m\}$. Then we have
\begin{align*}
&\qquad \Delta(\:x_i^+(z_1)
\cdots
x_i^+(z_{m})x_j^+(w)\;)\\
&=\Delta\(\iota_{z_1,\dots,z_m,w}\prod_{1\le r<s\le m}f_{ii}^+(z_s,z_r)
\prod_{1\le r\le m}f_{ij}^+(z_r,w)
x_i^+(z_1)
\cdots
x_i^+(z_{m})x_j^+(w)\)
\\
&=\sum_{J\subset\overline{m}}
\iota_{z_1,\dots,z_m,w}\prod_{1\le r<s\le m}f_{ii}^+(z_s,z_r)
\prod_{1\le r\le m}f_{ij}^+(z_r,w)
\xi(z_1)\cdots\xi(z_m)x_j^+(w)\\
&\qquad\otimes x_i^+(z_{j_1}q^{c_1})\cdots
    x_i^+(z_{j_r}q^{c_1})\\
&+\sum_{J\subset\overline{m}} \iota_{z_1,\dots,z_m,w}\prod_{1\le r<s\le m}f_{ii}^+(z_s,z_r)
  \prod_{1\le r\le m}f_{ij}^+(z_r,w)
  \xi(z_1)\cdots\xi(z_m)\phi_j^-(wq^{\frac{c_1}{2}})\\
&\qquad\otimes x_i^+(z_{j_1}q^{c_1})\cdots
    x_i^+(z_{j_r}q^{c_1})x_j^+(w q^{c_1}),
\end{align*}
where $\xi(z_a)=\phi_i^-(z_{a}q^{\frac{c_1}{2}})$ if $a\in J$, $\xi(z_a)=x_i^+(z_a)$ if $a\notin J$ and
$J=\{j_1,\dots,j_r\}$ with $j_1<\cdots<j_r$.
Note that for $k,l\in I$, we have
 $\phi_k^-(zq^{\frac{c}{2}})x_l^+(w)\ot 1\in\E_\hbar^{(2)}(U\wh\ot V)$.
 This together with the relations (Q6) gives
\begin{align*}
\iota_{w,z}f_{kl}^+(z,w)\phi_k^-(zq^{\frac{c_1}{2}})x_l^+(w)\ot 1
=C_{kl}\iota_{w,z}f_{lk}^+(w,z)x_l^+(w)\phi_k^-(zq^{\frac{c_1}{2}})\ot 1.
\end{align*}
In view of this and Proposition \ref{prop:normal-ordering-rational}, we can move those $\phi_j^-(z_jq^{\frac{c_1}{2}})$ to the left so that
\begin{align*}
&\Delta\(\:x_i^+(z_1)
x_i^+(z_2)\cdots
x_i^+(z_{m})x_j^+(w)\;\)\\
=&\sum_{J\subset\overline{m}}
\iota_{z_1,\dots,z_m,w}\(\prod_{a\le r<b}f_{ii}^+(z_{j_a},z_{j_b})\prod_{1\le a\le r}f_{ij}^+(z_{j_a},w)\)q^{r(1-r)d_ic_1/2}\\
&\times
\phi_i^-(z_{j_1}q^{\frac{c_1}{2}})\cdots \phi_i^-(z_{j_r}q^{\frac{c_1}{2}})
\:x_i^+(z_{j_{r+1}})\cdots x_i^+(z_{j_{m}})x_j^+(w)\;\\
&\quad
\otimes \:x_i^+(z_{j_1} q^{c_1})\cdots
    x_i^+(z_{j_r} q^{c_1})\;\\
+&\sum_{J\subset\overline{m}}(-1)^{m-r}
\iota_{z_1,\dots,z_m,w}\(
\prod_{a\le r<b}f_{ii}^+(z_{j_a},z_{j_b})\prod_{b=r+1}^{m}f_{ji}^+(w,z_{j_b})\)
q^{r\(\frac{1-r}{2}d_i-d_{ij}\)c_1}
\\
&\times
\phi_i^-(z_{j_1}q^{\frac{c_1}{2}})\cdots \phi_i^-(z_{j_r}q^{\frac{c_1}{2}})
\phi_j^-(wq^{\frac{c_1}{2}})
\:x_i^+(z_{j_{r+1}})\cdots x_i^+(z_{j_{m}})\;\\
&\quad
\otimes \:x_i^+(z_{j_1} q^{c_1})\cdots
x_i^+(z_{j_r} q^{c_1})x_j^+(w q^{c_1})\;,
\end{align*}
noting that  $C_{ii}=1$ and $C_{ij}=C_{ji}=-1$,
where $r=|J|$ and $j_1,\dots,j_{m}$ are distinct numbers in $\overline{m}$ such that
\begin{align*}
J=\{j_1,\dots,j_r\},\quad j_1<\cdots,j_r\quad \te{and}\quad
j_{r+1}<\cdots<j_{m}.
\end{align*}

Now by taking $z_a=q_i^{a_{ij}+(a-1)2}w$ for $a=1,\dots,m$, we obtain
\begin{align*}
  &\Delta\(\:x_i^+(q_i^{a_{ij}} w)x_i^+(q_i^{a_{ij}+2}w)
  \cdots x_i^+(q_i^{-a_{ij}} w)x_j^+(w)\;\)\\
=&\:x_i^+(q_i^{a_{ij}} w)x_i^+(q_i^{a_{ij}+2} w)
  \cdots x_i^+(q_i^{-a_{ij}} w)x_j^+(w)\;\otimes 1\\
+&\phi_i^-(q_i^{a_{ij}} wq^{\frac{c_1}{2}})
\phi_i^-(q_i^{a_{ij}+2} wq^{\frac{c_1}{2}})
\cdots
\phi_i^-(q_i^{-a_{ij}} wq^{\frac{c_1}{2}})
\phi_j^-(wq^{\frac{c_1}{2}})
q^{m\(\frac{1-m}{2}d_i-d_{ij}\)c_1}
\\
&
\otimes \:x_i^+(q_i^{a_{ij}} w q^{c_1})
    x_i^+(q_i^{a_{ij}+2} w q^{c_1})
  \cdots x_i^+(q_i^{-a_{ij}} w q^{c_1})
  x_j^+(w q^{c_1})\;.
\end{align*}
This implies that $\Delta$ is compatible with the relation \eqref{eq:Dr-equiv-normal-ordering-rel}, as desired.
\end{proof}

It is obvious that the canonical $\C[[\hbar]]$-isomorphism
\[a:(U\wh\ot V)\wh\ot W\cong U\wh\ot(V\wh\ot W)\] is a $\qtar$-module isomorphism.
Furthermore, we endow $\C[[\hbar]]$ with a $\qtar$-module structure
given by ($i\in I, h\in \check{\h}$)
\begin{eqnarray*}
\te{(CoU) }\qquad  \epsilon\(x_i^\pm(z)\)=0=\epsilon(h)=\epsilon(c)=\epsilon(d),\quad
  \epsilon\(\phi_i^\pm(z)\)=1.
\end{eqnarray*}
Then $\C[[\hbar]]\in \R$ and for every $U\in \R$, the natural $\C[[\hbar]]$-isomorphisms
\begin{align*}
l: \C[[\hbar]]\wh\ot U\cong U\quad\te{and}\quad r: U\wh\ot\C[[\hbar]]\cong U
\end{align*}
are $\qtar$-module isomorphisms.
In view of Proposition \ref{prop:delta}, we have the following straightforward result.

\begin{thm}\label{thm:monoidal-cat}
The category $\R$ of restricted $\qtar$-modules, together with
the  completed tensor product $\wh{\ot}$, the trivial module $\C[[\hbar]]$,
the associativity constraint $a$, the left constraint $l$ and
the right constraint $r$  form a monoidal category.
\end{thm}

\subsection{Hopf structure on $\qptar$}
In this subsection we define a topological Hopf algebra structure on a
completion of $\qtar$.

 For each $W\in\R$, we endow with $\End_{\C[[\hbar]]}(W)$ the weak topology given in
 \cite[\S 7.1 Remark]{EK-I}, noting that $\End_{\C[[\hbar]]}(W)$  is the space of continuous endomorphisms on $W$ (with
 the $\hbar$-adic topology).
Then  $\End_{\C[[\hbar]]}(W)$ is a topological algebra over $\C[[\hbar]]$
such that
\begin{align*}
  \set{(K:\hbar^nW)}{K\subset W,\, |K|<\infty,\,n\in\Z_+}
\end{align*}
form a local basis at  $0$, where
\begin{align*}
  (K:\hbar^nW)=\set{\varphi\in\End_{\C[[\hbar]]}(W)}{\varphi(K)\subset \hbar^nW}.
\end{align*}
 Since $W$ is separated and complete,
 $\End_{\C[[\hbar]]}(W)$ is separated and complete under the weak topology (\cite[\S 7]{EK-I}).

Let $\mathcal F:\R\to\mathcal M_f$ be the forgetful functor, and let
 $\mathcal A=\End_{\C[[\hbar]]}(\mathcal F)$ be  the algebra of endomorphisms of the functor $\mathcal F$.
For any $W\in \R$, there is a canonical $\C[[\hbar]]$-map $\pi_W$ from $\mathcal A$ to
$\End_{\C[[\hbar]]}(W)$ such that $\pi_W(a)=a_W$ for every natural transformation $a=(a_W:W\rightarrow W)_{W\in \R}\in \mathcal{A}$.
We equip $\mathcal A$ with the projective topology relative to the family $\{\pi_W\}_{W\in \R}$, namely, the coarsest
 topology for which each $\pi_W$ is continuous.
 Then $\mathcal{A}$ is a topological algebra over $\C[[\hbar]]$.

 \begin{lem}\label{lem:topoemd}
 The topological algebra $\mathcal A$ is separated, complete and topologically free.
 \end{lem}
 \begin{proof} We first prove that $\mathcal{A}$ is separated. Let $a$ be an element in the closure of $0$, which is the intersection
 of all open left ideals of $\mathcal{A}$.
 Then  for any $W\in\mathcal R$, $a_W=\pi_W(a)$ lies in the intersection of all open left ideals of $\End_{\C[[\hbar]]}(W)$.
Since $\End_{\C[[\hbar]]}(W)$ is separated, we get $a_W=0$.
This forces that $a=0$, as required.

To prove the completeness, let $(a^{(\alpha)})_{\alpha\in J}$ be a Cauchy net in $\mathcal{A}$.
For $W\in\mathcal R$, $(\pi_W(a^{(\alpha)}))_{\alpha\in J}$ is a Cauchy net in $\End_{\C[[\hbar]]}(W)$,
 and hence converges to an element, say $a_W$, in $\End_{\C[[\hbar]]}(W)$.
Let $f:V\rightarrow W$ be a morphism of objects $V,W\in \R$.
Then for any $v\in V$ and $n\in \Z_+$, there exists $\alpha\in J$ such that
\[\pi_W(a-a^{(\alpha)})\in (f(v),\hbar^nW)\quad\te{and}\quad \pi_V(a-a^{(\alpha)})\in (v:\hbar^n V).\]
This implies that
\begin{align*}
&f(a_V(v))-a_W(f(v))=f(a_V(v)-a_W(f(v))+f(a_V^{(\alpha)}(v)-a_W^{(\alpha)}(f(v))\\
=&f((a_V-a^{(\alpha)}_V)(v))-(a_W-a^{(\alpha)}_W)(f(v))\in \hbar^n W.\end{align*}
As $W$ is separated, we have that $f(a_V)(v)=a_W(f(v))$.
This gives that $(a_W)_{W\in \R}\in \mathcal{A}$.
Moreover, it is obvious that $(a^{(\alpha)})_{\alpha\in J}$
converges to  $(a_W)_{W\in \R}$, which proves that $\mathcal{A}$ is complete.

Note that $\pi_W$ is $\C[[\hbar]]$-linear for $W\in\mathcal R$.
Then $\pi_W(\hbar^n\mathcal A)\subset (K:\hbar^nW)$ for $n\in\Z_+$, and $K\subset W$ with $|K|<\infty$.
Thus, $\pi_W$ is continuous with $\mathcal A$ equipped with the $\hbar$-adic topology.
This implies that the projective topology on $\mathcal A$ is coarser than the $\hbar$-adic topology.
Then we obtain that $\mathcal{A}$  is  $\hbar$-adically separated and $\hbar$-adically complete.
Note that for each $\C$-vector space $W^0$, there is a canonical identification
 $\End_{\C[[\hbar]]}(W^0[[\hbar]])=(\mathrm{End}(W^0))[[\hbar]]$.
Then each $\End_{\C[[\hbar]]}(W)$ is  torsion-free, and so is $\mathcal{A}$.
Hence, $\mathcal{A}$ is topologically free (\cite{Kassel-topologically-free}).

 \end{proof}



Recall from Theorem \ref{thm:monoidal-cat} that $\R$ is a monoidal category.
For each $U,V\in\R$, let $J_{U,V}:\mathcal F(U)\wh\ot\mathcal F(V)\to\mathcal F(U\wh\ot V)$ be the canonical $\C[[\hbar]]$-isomorphism.
Then $\set{J_{U,V}}{U,V\in\R}$ defines a  tensor structure on the functor $\mathcal F$ \cite[Section 2.4]{EK-I}.
Let $\mathcal F^2:\R\times\R\to\mathcal M_f$ be the bifunctor defined by
$\mathcal F^2(U,V)=\mathcal F(U)\wh\ot\mathcal F(V)$, and set
 $\mathcal A^2=\End_{\C[[\hbar]]}(\mathcal F^2)$.
Similar to $\mathcal A$, we endow  $\mathcal A^2$ a weakest topological structure so that it is a complete separated topological $\C[[\hbar]]$-algebra.
Following \cite[Section 9.1]{EK-I},  $\mathcal A$ admits a natural ``coproduct'' $\Delta:\mathcal A\to \mathcal A^2$ defined by $\Delta(a)_{(U,V)}(u\ot v)=J_{U,V}\inv a_{U\wh\ot V} J_{U,V}(u\ot v)$,
$a\in \mathcal A$, $u\in\mathcal F(U)$, $v\in\mathcal F(V)$,
where $a_U$ denotes the action of $a$ on $\mathcal F(U)$.
We also define the counit on $\mathcal A$ by $\epsilon(a)=a_{\C[[\hbar]]}(1)\in\C[[\hbar]]$, where $\C[[\hbar]]$ is the trivial $\qtar$-module in $\R$.

Notice that each $a\in\qtar$ naturally defines an endomorphism of $\mathcal F$, so we have
a canonical $\C[[\hbar]]$-algebra homomorphism $\psi$ from $\qtar$ to $\mathcal A$.
\begin{de}\label{de:rest-comp}
We define the {\em restricted completion}  $\qptar$ of $\qtar$ to be
the closure of $\psi(\qtar)$ in $\mathcal A$.
\end{de}

  It is straightforward to see that $\qptar$ is complete, separated and topologically free.
  Furthermore, the following standard result is clear.

\begin{prop}\label{prop:rest-comp}
The continuous  $\qptar$-modules in $\mathcal M_f$ are exactly restricted $\qtar$-modules.
\end{prop}

Denote by $\wt{\mathcal O}$ the set of open left ideals of $\qptar$.
Set $\wt{\mathcal O}_1=\wt{\mathcal O}$ and
\begin{align*}
&\wt{\mathcal O}_{m+1}=\set{L_1\wh\ot \qptar^{\wh\ot^m}+\qptar\wh\ot L_m}
{L_1\in\wt{\mathcal O},\,L_m\in\wt{\mathcal O}_m},\quad m\in\Z_+.
\end{align*}
It is easy to verify that $\qptar^{\wh\ot^m}$ is a topological $\C[[\hbar]]$-algebra
with $\wt{\mathcal O}_m$  a local basis at $0$.
For $m\in\Z_+$, set
\begin{align}
  \qptar^{\wt\ot^m}=\varprojlim_{L\in\wt{\mathcal O}_m}\qptar^{\wh\ot^m}/L.
\end{align}
One notices that $\qptar^{\wt\ot^2}$ is a closed subalgebra of $\mathcal A^2$.

For each element $a$ in \eqref{eq:tqagenerators}, we still denote its image in $\qptar$ by $a$.
From Theorem \ref{thm:monoidal-cat}, we have that $\Delta$ satisfies the conditions (Co1-Co4) and $\epsilon$ satisfies the condition (CoU).
Then we have that $\Delta(\qptar)\subset \qptar^{\wt\ot^2}$.
Following \cite[Proposition 9.1]{EK-I}, one immediately gets the following result.

\begin{prop}
$(\qptar,\Delta,\epsilon)$ is a topological bialgebra over $\C[[\hbar]]$.
\end{prop}

By a similar argument 
of \cite[Theorem 2.1]{DI-generalization-qaff}, one can verify that
there is a continuous anti-homomorphism $S:\qptar\to\qptar$ determined by
$(i\in I,h\in\check\h)$
\begin{align*}
  S(c)&=-c,\quad S(d)=-d,\quad
  S(h)=-h,\quad
  S\(x_i^+(z)\)=-\phi_i^-(zq^{-\frac{c}{2}})\inverse x_i^+(zq^{-c}),\\
  &S\(x_i^-(z)\)=-x_i^-(zq^{-c})\phi_i^+(zq^{-\frac{c}{2}})^{-1},\quad  S\(\phi_i^\pm(z)\)=\phi_i^\pm(z)\inverse.
\end{align*}

In 
summary, we have obtained the following  main result of this section:

\begin{thm}\label{thm:hopf}
$(\qptar,\Delta,\epsilon,S)$ is a topological Hopf algebra over $\C[[\hbar]]$.
\end{thm}

\section{Specialization and quantization of EALAs}\label{sec:specialization}
In this section we study the classical limit of $\qtar$ and establish its connection with 
the quantization 
of extended affine Lie algebras.
\subsection{Specialization}
By specializing $\hbar=0$ in the definition of $\qtar$, we have:
\begin{de}
Define  $\hat\g_\mu$ to be the Lie algebra over $\C$
generated by the set
 \[\set{L_h,\,H_{i,m},\, X_{i,m}^\pm,\, C,\, D}{h\in\check{\h},\, i\in I,\ m\in \Z}\]
 and subject to the following relations $(h,h'\in \check\h,i,j\in I,m,n\in\Z)$
\begin{eqnarray*}
&&\text{(L0) }H_{\mu (i),n}=\xi^n H_{i,n},\ \
   L_{r_i\sum_{k\in \Z_N}\al_{\mu^k(i)}^\vee}=H_{i,0}, \\
&&\text{(L1) }[D,L_h]=0,\quad [D,H_{i,m}]=mH_{i,m}^\pm,\quad [L_h,H_{i,m}]=0=[L_h,L_{h'}], \\
&&\text{(L2) }[H_{i,m},H_{j,n}]
=\sum_{k\in \Z_N} r_i a_{i\mu^k(j)}\delta_{m+n,0}m\xi^{km}C,\quad C\ \te{is central},\\
&&\text{(L3) } [D,X^\pm_{i,m}]=mX^\pm_{i,m},\quad \ [L_h,X_{j,n}^\pm]=\pm \al_{j}(h) X_{j,n}^\pm, \\
&&\text{(L4) }[H_{i,m},X^\pm_{j,n}]
=\pm\sum_{k\in \Z_N}r_i a_{i\mu^k(j)}X^\pm_{j,m+n}\xi^{km},\\
&&\text{(L5) }[X^+_{i,m},X^-_{j,n}]
=\sum_{k\in \Z_N}\delta_{i,\mu^k(j)}\(\frac{H_{j,m+n}}{r_j}+
\frac{m}{r_j}\delta_{m+n,0} C\)\xi^{km},\\
&&\text{(L6) }X^\pm_{\mu (i),n}=\xi^n X^\pm_{i,n},\ \ F_{ij}(z,w)[X_i^\pm(z),X_j^\pm(w)]=0,\\
&&\text{(L7) }  \sum_{\sigma\in S_2} p_i(z_{\sigma(1)},z_{\sigma(2)},z_3)
 [X_i^\pm(z_{\sigma(1)}),[X_i^\pm(z_{\sigma(2)}),X_i^\pm(z_3)]]=0,\ \text{if }s_i=2,
\\
 &&\text{(L8) }p_{ij,0}(z_1,\dots,z_{m_{ij}},w)
[X_i^\pm(z_{1}),\cdots,[X_i^\pm(z_{m_{ij}}), X_j^\pm(w)]] =0,\ \text{if }\ \check{a}_{ij}< 0,
\end{eqnarray*}
where  $X_{i}^\pm(z)=\sum\limits_{m\in\Z}X_{i,m}^\pm z^{-m}$  and
\begin{align*}
&F_{ij}(z,w)=F^\pm_{ij}(z,w)|_{\hbar\mapsto 0}=G^\pm_{i,j}(z,w)|_{\hbar\mapsto 0}=\prod_{k\in \Gamma_{ij}}(z-\xi^kw),\\
&p_{ij}(z,w)=p_{ij}^\pm(z,w)|_{\hbar\mapsto 0}=(z^{d_i}+w^{d_i})^{s_i-1}\frac{z^{d_{ij}}-w^{d_{ij}}}{z^{d_i}-w^{d_i}},\\
&p_{ij,0}(z_1,\dots,z_{m_{ij}},w)=p^\pm_{ij,0}(z_1,\dots,z_{m_{ij}},w)|_{\hbar\mapsto 0}\\
&\qquad=\prod_{1\le a<b\le m_{ij}}p_{ij}(z_a,z_b)
\prod_{1\le a\le m_{ij}}\prod_{k\in \bar\Gamma_{ij}}(z_a-\xi^kw),\\
&p_i(z_1,z_2,z_3)=p_i^\pm(z_1,z_2,z_3)|_{\hbar\mapsto 0}=z_1^{d_i}-2z_2^{d_i}+z_3^{d_i}.
\end{align*}
\end{de}

Then we have the following result.

\begin{prop}\label{prop:classical-limit}
The classical limit $\qtar/\hbar\qtar$ of $\qtar$ is isomorphic to
the universal enveloping algebra $\U(\hat\g_\mu)$ of $\hat\g_\mu$.
\end{prop}
\begin{proof}
Recall the generating set \eqref{eq:tqagenerators} of $\qtar$.
Let $\U^f$ denote  the free $\C$-algebra generated by \eqref{eq:tqagenerators}, and set  $\U^f_\hbar=\U^f[[\hbar]]$.
Then by definition $\qtar$ is the quotient algebra of $\U^f_\hbar$  modulo the closed ideal $\bar{I}$, where
$\bar{I}$ stands for  the closure (for the $\hbar$-adic topology) of
 the two-sided ideal $I$  of $\U^f_\hbar$ generated by the coefficients in (Q0)-(Q10).

Let $\psi$ denote  the  surjective $\C$-algebra homomorphism from the free $\C$-algebra  $\U^f$ to $\U(\hat\fg_\mu)$
determined by
\begin{align*}
  h\mapsto L_h,\ h_{i,m}\mapsto H_{i,m},\ x_{i,m}^{\pm}\mapsto X_{i,m}^\pm,\ c\mapsto C,\ d\mapsto D
\end{align*}
where $h\in\check\h,\,i\in I$ and $m\in\Z$.
We also view $\psi$ as a $\C$-algebra homomorphism from $\U^f_\hbar$ to $\U(\hat\fg_\mu)$ by letting $\psi(\hbar)=0$.
Note that, via the homomorphism $\psi$, the relations (Q0)-(Q10) in $\U^f_\hbar$ are exactly the relations (L1)-(L8) of $\hat\fg_\mu$ (see (Q0$'$)-(Q7$'$) in Section \ref{subsec:q-KM-alg}).
This implies that $\psi(I)=0$ and hence $\psi(\bar{I})=0$ (as $\psi(\hbar)=0$ and $I/\hbar I=\bar{I}/\hbar \bar{I}$).
Therefore, we obtain a surjective $\C$-algebra homomorphism  from $\U_\hbar(\hat\fg_\mu)$ to $\U(\hat\fg_\mu)$.
Furthermore, this homomorphism  induces a surjective $\C$-algebra homomorphism from $\U_\hbar(\hat\fg_\mu)/\hbar\U_\hbar(\hat\fg_\mu)$ to $\U(\hat\fg_\mu)$,
which we also denote as $\psi$.

On the other hand, we have a surjective $\C$-algebra homomorphism $\varphi$ from $\U(\hat\fg_\mu)$ to $\U_\hbar(\hat\fg_\mu)/\hbar\U_\hbar(\hat\fg_\mu)$  determined by
\begin{align*}
 &L_h\mapsto h+\hbar\U_\hbar(\hat\fg_\mu),\quad H_{i,m}\mapsto h_{i,m}+\hbar\U_\hbar(\hat\fg_\mu),\quad X_{i,n}^\pm\mapsto x_{i,n}^\pm+\hbar\U_\hbar(\hat\fg_\mu),\\
 &C\mapsto c+\hbar\U_\hbar(\hat\fg_\mu),\quad D\mapsto d+\hbar\U_\hbar(\hat\fg_\mu),\quad\te{for } i\in I,\,m\in\Z^\times,\,n\in\Z.
\end{align*}
It is obvious that $\varphi$ is the inverse of $\psi$, so the proposition is proved.
\end{proof}

Let $\hat{\g}_\mu^+$ (resp. $\hat\g_\mu^-$; resp. $\hat\h_\mu$) be the subalgebra of
$\hat\g_\mu$  generated by the elements $X_{i,m}^+$ (resp. $X_{i,m}^-$; resp. $H_{i,m}, L_h, C, D$).
By specializing $\hbar=0$  in the proof of Theorem \ref{thm:tri-decomp}, one also 
obtains the following result.

\begin{prop}\label{prop:lietri}$\hat\fg_\mu=\hat\g_\mu^+\oplus \hat\h_\mu\oplus \hat\g_\mu^-$ and
$\hat{\g}_\mu^+$ (resp. $\hat\g_\mu^-$; resp. $\hat\h_\mu$) is the subalgebra of
$\hat\g_\mu$ abstractly generated by $X_{i,m}^+$ (resp. $X_{i,m}^-$; resp. $H_{i,m}, L_h, C, D$)
subject to the relations \te{(L5)-(L7)} with ``$+$'' (resp. \te{(L5)-(L7)} with ``$-$''; resp. \te{(L0)-(L2)}).
\end{prop}

Let $\hat\g_\mu'$ be the subalgebra of $\hat\g_\mu$ generated by the elements $X_{i,m}^\pm, H_{i,m}, C$.
Then it follows from Proposition \ref{prop:lietri} that $\hat\g_\mu'$ is abstractly generated
by these elements with relations (L0),(L2) and (L4)-(L8).
Furthermore, we have
\begin{align}\label{hatgmucd}
\hat\g_\mu=\hat\g_\mu'\oplus \sum_{h\in \h''\cap \check{\h}}\C L_h \oplus \C D.
\end{align}

Let
$
\mathrm{Aff}(\g)=\(\C[t,t^{-1}]\ot \g\)\oplus \C \bm{c} \oplus \C \bm{d}
$
be the affinization of $\g$ \cite{Kac-book} with the Lie bracket:
\begin{equation}\begin{split}\label{eq:affg}
&[t^m\ot x+a_1\bm{c}+b_1\bm{d}, t^n\ot y+a_2\bm{c}+b_2\bm{d}]\\
=\,&t^{m+n}\ot [x,y]+\<x\mid y\>\delta_{m+n,0} \bm{c}+b_1n(t^n\ot y)-b_2m(t^m\ot x),
\end{split}\end{equation}
for $x\in \g$, $m\in \Z$, and $a_1,b_1,a_2,b_2\in \C$.
We extend the bilinear form on $\g$ to 
$\mathrm{Aff}(\g)$ by decreeing that
\begin{align*}
\<t^m\ot x+a_1\bm{c}+b_1\bm{d}\mid t^n\ot y+a_2\bm{c}+b_2\bm{d}\>=\delta_{m+n,0}\<x\mid y\>+a_1b_2+b_1a_2.
\end{align*}

Recall from \cite{Kac-book} that any finite order automorphism  $\sigma$  of $\g$ can be extended to  an automorphism,
say $\hat{\sigma}$, of $\mathrm{Aff}(\g)$ with
\begin{align*}
\hat{\sigma}(t^m\ot x)=\xi_M^{-m}(t^m\ot \mu(x)),\quad\hat{\sigma}(\bm{c})=\bm{c}\quad \te{and}\quad \hat{\sigma}(\bm{d})=\bm{d},
\end{align*}
for $x\in \g$ and $m\in \Z$, where $M$ is the order of $\sigma$.
Denote by $\mathrm{Aff}(\g,\sigma)$ the subalgebra of $\mathrm{Aff}(\g)$ fixed by $\hat\sigma$.
Let $\mathfrak{b}$ be a $\sigma$-invariant subspace  of $\g$.
For  $x\in \mathfrak{b}$ and $m\in \Z$, set
\begin{align*}
x_{(m)}=\sum_{p\in \Z_M}\xi_M^{-mp}\mu^p(x)\quad\te{and}\quad
\mathfrak{b}_{(m)}=\te{Span}_\C\{x_{(m)}\mid x\in \mathfrak{b}\}.
\end{align*}
Then we have
\begin{align*}
\mathrm{Aff}(\g,\sigma)=\(\sum_{m\in \Z} t^m\ot \g_{(m)}\) \oplus \C \bm{c} \oplus \C \bm{d}.
\end{align*}

Recall that $\g'=[\g,\g]$ and $\h''$ is the complementary space for $\h'(=\g'\cap \h)$ in $\h$.
Set
\begin{align*}
\hat{\CL}(\g)=\(\C[t,t^{-1}]\ot \g'\)\oplus \h''\oplus \C \bm{c}\oplus \C \bm{d}\quad \te{and}\quad
\CL(\g)=\(\C[t,t^{-1}]\ot \g'\)\oplus \C \bm{c}.
\end{align*}
Note that both $\hat{\CL}(\g)$ and $\CL(\g)$ are $\hat\mu$-invariant. Denote their corresponding  $\hat\mu$-fixed point subalgebras by
\begin{align*}
\hat{\CL}(\g,\mu)&=\(\sum_{m\in \Z} t^m\ot \g'_{(m)}\)\oplus \h''_{(0)}\oplus \C \bm{c} \oplus \C \bm{d},\\
\CL(\g,\mu)&=\(\sum_{m\in \Z} t^m\ot \g'_{(m)}\)\oplus \C \bm{c}.
\end{align*}

We have:
\begin{prop} \label{prop:generalcl} The assignment ($h\in \h, i\in I$ and $m\in \Z$)
\begin{align}\label{defpsi}
L_h\mapsto h_{(0)},\ H_{i,m}\mapsto  t^m\ot r_i \al_{i(m)}^\vee,\
X_{i,m}^\pm\mapsto t^m\ot e^\pm_{i(m)},\ C\mapsto \frac{\bm{c}}{N},\ D\mapsto \bm{d},
\end{align}
defines a surjective Lie homomorphism from $\hat{\g}_\mu$ to $\hat{\CL}(\g,\mu)$.
\end{prop}
\begin{proof} Note that $\hat{\CL}(\g,\mu)$ is generated by the
 elements $h_{(0)}, t^m\ot r_i\al_{i(m)}^\vee, t^m\ot e^{\pm}_{i(m)}, \bm{c}, \bm{d}$,
where $h\in \h, i\in I, m\in \Z$.
Thus it suffices to verify that, under the correspondence \eqref{defpsi},
 these elements satisfy the relations (L0)-(L8).
The verification of  (L0)-(L5) is straightforward and omitted.
For the relations (L6)-(L8), define
the generating functions $x(z)=\sum_{n\in \Z}
(t^n\ot x_{(n)})z^{-n}$ for $x\in \g$.
And, for  $i,j\in I$, $m\in \Z_+$ and $\bm{k}=(k_1,\dots,k_m)\in (\Z_N)^m$, set
\begin{align*}
e_{ij}^\pm(\bm{k})&=[e^\pm_{\mu^{k_1}(i)},[e^\pm_{\mu^{k_2}(i)},\cdots,[e^\pm_{\mu^{k_{m}}(i)},e^\pm_j]]],\\
\al_{ij}(\bm{k})&=\al_{\mu^{k_1}(i)}+\al_{\mu^{k_2}(i)}+\cdots+\al_{\mu^{k_m}(i)}+\al_j.
\end{align*}
Note that $e_{ij}^\pm(\bm{k})\ne 0$ only if
$\al_{ij}(\bm{k})\in \Delta$, the root system  of $\g$.
Furthermore, it is straightforward to see that
\begin{equation}\begin{split}\label{serre-ex}
  &[e^\pm_i(z_1),[e^\pm_i(z_2),\cdots,[e^\pm_i(z_{m}),e^\pm_j(w)]]]\\
  =&\sum\limits_{\bm{k}=(k_1,\dots,k_m)\in(\Z_N)^m}
 e_{ij}^\pm(\bm{k})(w) \delta\(\frac{\xi^{-k_1}w}{z_1}\)
 \delta\(\frac{\xi^{-k_2}w}{z_2}\)\cdots \delta\(\frac{\xi^{-k_{m}}w}{z_{m}}\).
\end{split}\end{equation}

Note that the relations (L6) follow from \eqref{serre-ex}.
We now prove the relations (L7). So, assume that $i=j\in I$, $s_i=2$ and $m=2$.
In this case we claim that
\begin{align}\label{claim1}
 \al_{ii}(\bm{k})\notin \Delta,\quad \forall\ \bm{k}=(k_1,k_2)\in (\Z_N)^2.
 \end{align}
This claim together with \eqref{serre-ex} gives
 that \begin{align*}
 [e_i^\pm(z_1),[e_i^\pm(z_2),e_i^\pm(w)]]=0\end{align*} and in particular the relations (L7)
hold for $e_i^\pm(z)$.
We divide the proof of the claim \eqref{claim1} into three cases: (1) if both $k_1$ and $k_2$ can be divided by $N_i/2$,
then it follows from Lemma \ref{lem:linking} (ii) that
 $\{\al_{\mu^{k_1}(i)},\al_{\mu^{k_2}(i)},\al_i\}$ form a base for the root
system of type $A_1$ or $A_2$. This implies that $\al_{ii}(\bm{k})\notin \Delta$  as
any root in such a root system   has
height $<3$; (2) if  exactly one of $k_1,k_2$, say $k_1$, can be divided by $N_i/2$, then $a_{\mu^{k_2}(i)}$ is
orthogonal to  $\al_{\mu^{k_1}(i)}, \al_i$ and so $\al_{ii}(\bm{k})\notin \Delta$;
(3) if neither $k_1$ nor $k_2$ can be divided by $N_i/2$, then $\al_i$ is
orthogonal to  $\al_{\mu^{k_1}(i)}, \al_{\mu^{k_2}(i)}$ and so $\al_{ii}(\bm{k})\notin \Delta$, as desired.

Now we  prove the relations (L8).
Let $i,j\in I$ with $\check{a}_{ij}<0$, $m=m_{ij}=1-\bar{a}_{ij}$ and
$\bm{k}=(k_1,\dots,k_m)\in (\Z_N)^m$.
Recall the set $\Gamma_{ij}^{\mathrm{d}}$ defined in \eqref{eq:gammaij-}. Set
\begin{align}\label{eq:defpij'}
\Omega_{ij}=\Gamma_{ij}^{\mathrm{d}}\setminus\{lN_i\mid 0\le l\le d_i-1\}\ \te{and}\
p_{ij}'(z,w)=(z^{d_i}+w^{d_i})^{s_i-1}\prod_{k\in \Omega_{ij}}(z-\xi^kw).
\end{align}
Note that the polynomial $p_{ij}'(z,w)$ divides $p_{ij}(z,w)$ as
\[p_{ij}(z,w)=(z^{d_i}+w^{d_i})^{s_i-1}\prod_{k\in \<\Gamma_{ij}^{\mathrm{d}}\>\setminus\{lN_i\mid 0\le l\le d_i-1\}}(z-\xi^kw).\]
We claim that
\begin{enumerate}\item[(a)]
if
 $\al_{ij}(\bm{k})\in \Delta$, then $\xi^{k_s}z_s-\xi^{k_t}z_t$  divides $p'_{ij}(z_s,z_t)$ for some $1\le s<t\le m$.
\end{enumerate}
In view of
\eqref{serre-ex}, this claim implies that
\[\prod_{1\le a<b\le m} p'_{ij}(z_a,z_b)\cdot [e^\pm_i(z_1),[e^\pm_i(z_2),\cdots,[e^\pm_i(z_{m}),e^\pm_j(w)]]]=0,\]
and hence
\[\prod_{1\le a<b\le m} p_{ij}(z_a,z_b)\cdot [e^\pm_i(z_1),[e^\pm_i(z_2),\cdots,[e^\pm_i(z_{m}),e^\pm_j(w)]]]=0.\]
Recall that $\prod_{1\le a<b\le m} p_{ij}(z_a,z_b)$ divides $p_{ij,0}(z_1,\dots,z_m,w)$.
This particulary shows that (L8) hold in $\hat{\CL}(\g,\mu)$.

Now we are ready to prove the claim (a).
So let us assume that $\al_{ij}(\bm{k})\in \Delta$.
We first prove the following assertion:
\begin{enumerate}\item[(b)]
$\xi^{k_s}z_s-\xi^{k_t}z_t$ does not  divide $z_s^{d_i}-z_t^{d_i}$ for some $1\le s<t\le m$.
\end{enumerate}
Otherwise, it follows that  $N_i$ divides $k_s-k_t$ for all $s,t$.
Then $\mu^{k_1}(i)=\cdots=\mu^{k_m}(i)=i'$ for some
$i'\in \mathcal O(i)$.
In view of (LC2), we see  either $a_{i'j}=0$ or $a_{i'j}=\bar{a}_{ij}$ and in each case
$\al_{ij}(\bm{k})=m\al_{i'}+\al_j\notin \Delta$, as required.

Suppose now that $k_r\in \Gamma_{ij}$ for all $r=1,\dots,m$. It then follows from the assertion (b) that
there are $1\le s<t\le m$ such that $k_t-k_s\notin \{lN_i\mid 0\le l\le d_i-1\}$.
This means that  $k_t-k_s\in \Omega_{ij}$ and so we have
 the following assertion:
\begin{enumerate}\item[(c)] if $k_r\in \Gamma_{ij}$ for all $r=1,\dots,m$, then
$\xi^{k_s}z_s-\xi^{k_t}z_t$ divides $\prod_{k\in \Omega_{ij}}(z_s-\xi^kz_t)$ for some $1\le s<t\le m$.
\end{enumerate}
On the other hand, assume that $k_{s_0}\notin \Gamma_{ij}$ for some $s_0=1,\dots,m$.
Recall that
 $k_{s_0}\notin \Gamma_{ij}$ implies that $a_{\mu^{k_{s_0}}(i)j}=0$.
For the case $s_i=1$, it follows from Lemma \ref{lem:linking} (i) that
 $a_{\mu^{k_{s_0}}(i)i'}=0$ for any $i'\ne \mu^{k_{s_0}}(i)\in \mathcal O(i)$.
Thus, $\al_{\mu^{k_{s_0}}(i)}$ is orthogonal to all the  elements in the set
\begin{align}\label{orthoset}
\{\al_{\mu^{k_s}(i)}, \al_j\mid 1\le s\le m\}\setminus \{\al_{\mu^{k_{s_0}}(i)}\},\end{align}
which  gives  $\al_{ij}(\bm{k})\notin \Delta$, a contradiction.
And for the case  $s_i=2$,
 if $k_{s_0}-k_{t}\not\equiv \frac{N_i}{2} (\ \te{mod}\ N_i)$  for all $t=1,\dots,m$,
then by Lemma \ref{lem:linking} (ii), $\al_{\mu^{k_{s_0}}(i)}$ is  orthogonal to all the  elements in
\eqref{orthoset} and hence  $\al_{ij}(\bm{k})\notin \Delta$, a contradiction too.
In summary, we obtain that
 \begin{enumerate} \item[(d)] if $k_{s_0}\notin \Gamma_{ij}$ for some $s_0=1,\dots,m$, then
 $s_i=2$ and
there exists a $t_0=1,\dots,m$ such that $k_{s_0}-k_{t_0}\equiv \frac{N_i}{2} (\ \te{mod}\ N_i)$.
\end{enumerate}

Recall that $\prod_{k\in \Omega_{ij}}(z_s-\xi^kz_t)$ divides $p'_{ij}(z_s,z_t)$.
Thus, from the assertions (c) and (d), it follows that the claim (a) holds
except the case that $s_i=2$  and there exist $s_0,t_0$ such that
$k_{s_0}\notin \Gamma_{ij}$ and $k_{s_0}-k_{t_0}\equiv \frac{N_i}{2} (\ \te{mod}\ N_i)$.
However, in this case, we have that $\xi^{k_{s_0}}z_{s_0}-\xi^{k_{t_0}}z_{t_0}$ divides $z_{s_0}^{d_i}+z_{t_0}^{d_i}$ and
hence  divides $p'_{ij}(z_{s_0},z_{t_0})$ as well.
This completes the proof of claim (a), as required.
\end{proof}

We denote the resulting surjective homomorphism given in Proposition \ref{prop:generalcl} by $\psi_{\g,\mu}$.
The following result was proved in \cite{Da2} (see also \cite{CJKT-drin-pre}).

\begin{thm} Assume that the GCM $A$ is of finite type. Then the homomorphism $\psi_{\g,\mu}: \hat\g_\mu\rightarrow \hat{\CL}(\g,\mu)$ is
an isomorphism.
\end{thm}

In general, $\psi_{\g,\mu}$ is not injective and when $\g$ is of affine type,  we will determine the kernel of
$\psi_{\g,\mu}$ in Section 7.3.

\subsection{Basics on EALAs}\label{sec:q-EALA}
We start with the definition of an extended affine Lie algebra (EALA for short).
Let $\E$ be a Lie algebra equipped with a nontrivial finite-dimensional self-centralizing
ad-diagonalizable subalgebra $\CH$ and a nondegenerate invariant symmetric bilinear form
$(\ \mid\ )$.
Let $\E=\oplus_{\al\in \CH^*} \E_\al$ be the root space decomposition of $\E$ with respect to
$\CH$, and let $\Phi=\{\al\in \CH^*\mid \E_\al \ne 0\}$ be the corresponding root system.
The form $(\ \mid \ )$ restricted to $\CH=\E_0$ is nondegenerate, and hence induces a
nondegenerate symmetric bilinear form on $\CH^*$.
Set
\[\Phi^\times=\{\al\in \Phi\mid (\al\mid\al)\ne 0\}\quad\te{and}\quad
\Phi^0=\{\al\in \Phi\mid(\al\mid\al)=0\}.\]
Let $\E_c$ be the subalgebra of $\E$ generated by the root spaces $\E_\al, \al\in \Phi^\times$,  called the core of
$\E$. Following \cite{AABGP}, we give:

\begin{de}\label{deeala} The triple $(\E,\CH,(\ \mid \ ))$  is called an extended affine Lie algebra if
\begin{enumerate}
\item $\te{ad}(x)$ is locally nilpotent for $x\in \E_\al, \al\in \Phi^\times$.
\item $\Phi^\times$ cannot be decomposed as a union of two orthogonal nonempty subsets.
\item The centralizer of $\E_c$ in $\E$ is contained in $\E_c$.
\item $\Phi$ is a discrete subset of $\CH^*$.
\end{enumerate}
\end{de}

The condition (4) in Definition \ref{deeala} implies that the subgroup $\<\Phi^0\>$ of $\CH^*$ generated by $\Phi^0$ is
a free abelian group of finite rank.
This rank is called the nullity of $\E$.
Nullity $0$ EALAs are exactly finite dimensional simple Lie algebras, while nullity $1$ EALAs are exactly
affine Kac-Moody algebras \cite{ABGP}.
 Set $\CH_\sigma=\h_{(0)}\oplus \C\bm c\oplus \C\bm d$,
an abelian subalgebra of $\mathrm{Aff}(\g,\sigma)$. Then we have

\begin{lem}\label{lem:ealaaff}
Assume that $\g$ is of finite type or  affine type. Then for any diagram automorphism $\sigma$ of $\g$,
 the triple $(\mathrm{Aff}(\g,\sigma),\CH_\sigma,\<\ \mid\ \>)$ is  an EALA if and only if $\sigma$ satisfies the linking
 conditions (LC1) and (LC2).
\end{lem}
\begin{proof} It was proved in \cite{ABP} that  $(\mathrm{Aff}(\g,\sigma),\CH_\sigma,\<\ \mid\ \>)$ is an EALA if and only if either $\g$ is of finite type, or $\g$
is of affine type and $\sigma$ is nontransitive.
Then the assertion follows by Lemma \ref{lem:LC1aff}.
\end{proof}

\begin{rem}{\em In the Definition \ref{deeala},  the choice of the invariant bilinear form is not important: let
$(\E,\CH,(\ \mid \ ))$ and $(\E,\CH,(\ \mid \ )')$ be two EALA structures on $\E$.
Then $\Phi=\Phi'$, $\Phi^{\times}=\Phi'^{\times}$ and $\Phi^0=\Phi'^{0}$,
where we distinguish the notations for $(\E,\CH,(\ \mid \ )')$ by $'$.
One may see  \cite[Corollary 3.3]{CNPY-conjugacy-EALA} for details.
Thus, as in \cite{Neher-EALA-survey,CNPY-conjugacy-EALA}, from now on we will denote EALAs as the couples $(\E,\CH)$.
}
\end{rem}

\begin{rem}\label{csact}{\em In a series of papers \cite{CGP-conjugacy,CNP-conjugacy-non-fgc-lie-tori,CNPY-conjugacy-EALA,CNP-conjugacy-non-fgc-centerless-core}, the conjugacy theorem of Cartan subalgebras for EALAs was proved.
Explicitly, let
$(\E,\CH)$ and $(\E,\CH')$ be two  EALA structures on a Lie algebra $\E$. Then
 there is an automorphism of $\E$ mapping $\CH$ onto $\CH'$. In particular, the root system is an invariant of $\E$.
}
\end{rem}

 The structure of EALAs is intimately connected to Lie torus introduced by Yoshii \cite{Y}.
 Let $S$ be a finite irreducible  root system, which is not necessarily reduced and contains $0$,
 and let $Q(S)$ be the corresponding root lattice.
 Let $\Lambda$ be a free abelian group of finite type, and
let $L$ be a Lie algebra graded by $(Q(S),\Lambda)$, that is, $L=\oplus_{\al\in Q(S),\lambda\in\Lambda}L_\al^\lambda$ such that
 $[L_\al^\lambda, L_\beta^\gamma]\subset L_{\al+\beta}^{\lambda+\gamma}$.
 For convenience, we set
 \begin{align*}
 L_\al=\oplus_{\lambda\in \Lambda}L_\al^\lambda,\ \al\in Q(S)\quad\te{and}\quad
 L^\lambda=\oplus_{\al\in Q(S)} L_\al^\lambda,\ \lambda\in \Lambda.
 \end{align*}
A Lie torus  of type $(S,\Lambda)$ is by definition a $(Q(S),\Lambda)$-graded Lie algebra satisfying certain conditions \cite{Y}, and
its nullity is defined as the rank of $\Lambda$.

\begin{rem}\label{uniLietorus}{\em Let $L$ be a Lie torus. Then $L$ is perfect and so it has
a universal central extension
$f:\mathfrak u(L)\rightarrow L$.
It is known that $\mathfrak u(L)$ is a Lie torus of the same type with $f(\mathfrak u(L)_\al^\lambda)=L_\al^\lambda$ for $\al\in Q(S)$, $\lambda\in \Lambda$ (cf.\cite{Neher-EALA-survey}).
}\end{rem}

\begin{rem}\label{centerless}{\em
A Lie torus $L$ is called centerless if its center $Z(L)=0$.
For example, the centerless core $\E_{cc}=\E_c/Z(\E_c)$ of an EALA $\E$ is a centerless Lie torus.
Let $L$ be a centerless Lie torus of type $(S,\Lambda)$. Then
 $\mathfrak u(L)$ is a centrally closed Lie torus.
Conversely, let $L$ be a  Lie torus of type $(S,\Lambda)$, which is centrally closed. Then  $L/Z(L)$ a centerless Lie torus
of type $(S,\Lambda)$ such that $L\cong \mathfrak u(L/Z(L))$.
}
\end{rem}

It is known that the core of an EALA is naturally a Lie torus \cite{AG-root-sys-core-EALA, Neher-EALA-survey}. More precisely,
let $(\E,\CH)$ be an EALA of nullity $n$ with the root system $\Phi$. Set $X=\te{span}_{\mathbb R}(\Phi)$ and $X^0=\te{span}_{\mathbb R}(\Phi^0)$.
Then the image $S$ of $\Phi$ in the quotient map $\pi:X\rightarrow Y=X/X^0$ is a finite irreducible root system, and
there is a linear map $f:Y\rightarrow X$ such that $f\circ \pi=\mathrm{id}_Y$ and $f(S_{ind})\subset \Phi$, where $S_{ind}=\{0\}\cup \{\beta\in S\mid \frac{1}{2}\al\notin S\}$.
Furthermore, the core $\E_c$ of $\E$ is a nullity $n$ Lie torus of type $(S,\Lambda)$ with
$(\E_c)_\al^\lambda=\E_c\cap \E_{f(\al)+\lambda}$, where  $\Lambda=\<\Phi^0\>$.
In particular, the $\Lambda$-grading on $\E_c$ is determined by
\begin{align*}
(\E_c)^\lambda=\oplus_{\beta\in \Phi, p(\beta)=\lambda}\E_c\cap \E_\beta,
\end{align*}
where $p$ is the projection map from $X=f(Y)\oplus \Lambda$ to $\Lambda$.

For example, let $\E=\mathrm{Aff}(\g,\mu)$ with $\g$ is of affine type (see Lemma \ref{lem:ealaaff}).
We view $(\check\h)^*$ as a subspace of $(\CH_\mu)^*$ in a natural way,  and
 define $\delta\in (\CH_\mu)^*$ such that $\delta(\check{\h})=0=\delta(\bm c)$ and
$\delta(\bm d)=1$.
Recall that $\check{\g}$ is the Kac-Moody algebra associated with $\check{A}$ and
$\check{\al}_i, i\in \check{I}$ are its simple roots (see Section \ref{subsec:link-cond}).
Then we have ($i\in \check{I},\ m\in \Z$):
\begin{align*}
t^m\ot e_{i(m)}^{\pm}\in \mathrm{Aff}(\g,\mu)_{\check{\al}_i+m\delta},\quad
t^m\ot \al_{i(m)}^\vee\in \mathrm{Aff}(\g,\mu)_{m\delta},\ \h''_{(0)}\oplus \C \bm d\subset \mathrm{Aff}(\g,\mu)_{0}.
\end{align*}
Since these elements generate the algebra $\mathrm{Aff}(\g,\mu)$, it follows that the root system $\Phi\subset \check{Q}\oplus \Z\delta$,
 where $\check{Q}=\oplus_{i\in \check{I}}\C\check{\al}_i$ is the root lattice of $\check{\g}$.
 On the other hand, 
we have $\Phi^0\subset \Z\delta_2\oplus \Z\delta$, where  $\delta_2$ denotes the null root in $\check\g$.
In fact, it follows from \cite[Corollary 2.31]{AABGP} that $\Phi^0=\Z\delta_2\oplus \Z \delta$.
Thus, the core $\CL(\g,\mu)$ of $\mathrm{Aff}(\g,\mu)$ is a Lie torus of type $(S,\Lambda)$, where $\Lambda=\Phi^0$ and the root system $S$ is determined in \cite[Theorem 12.2.1]{ABP}.
Note that the elements $\pi(\check{\al}_i), i\in \check{I}\setminus\{\omega\}$ form a basis of $Y$ and so we can define  the section
$f$ by $f(\pi(\check{\al}_i))=\check{\al}_i$
for $i\in \check{I}\setminus \{\omega\}$, where $\omega\in \check{I}$ is the additional node of $\check\g$ \cite{Kac-book}.
Then we have $p(\check{\al}_i)=\delta_{i,\omega}\delta_2$ for $i\in \check{I}$, which implies
$p(\al)=\al(\check{\rd_2})\delta_2+\al(\bm d)\delta$ for $\al\in \Phi$,
where $\check{\rd}_2$ is a scaling element in $\check\h$ such that $\check{\al}_i(\check{\rd}_2)=\delta_{i,\omega}$ for $i\in \check{I}$.
In particular, the $\Lambda$-grading on $\CL(\g,\mu)$ is given by
\begin{align}\label{lambdagrading}
\CL(\g,\mu)^{m\delta_2+n\delta}=\{x\in \CL(\g,\mu)\mid [\check{\rd}_2,x]=mx,\ [\bm d, x]=nx\}.
\end{align}

We define the \emph{extended core} $\bar{\E}$ of an EALA $(\E,\CH)$ to be the subalgebra $\E_c+\CH$ of $\E$.
It follows from Remark \ref{csact} that, up to the isomorphism, the extended core of an EALA is independent with
the choice of its Cartan subalgebra.
Following \cite{BGK}, we say that an EALA is maximal if its core is centrally closed.

\begin{de} For a nonnegative integer $n$, denote by $\wh{E}_n$ the class of Lie algebras which are isomorphic to
the extended core of a maximal EALAs with nullity $n$.
\end{de}

Note that   $\wh{E}_0$ is the class of finite dimensional simple Lie algebras and
$\wh{E}_1$ is the class of affine Kac-Moody algebras.
Let $L$ be a Lie torus of type $(S,\Lambda)$. Then any $\theta\in \mathrm{Hom}_\Z(\Lambda,\C)$ induces a so-called
degree derivation $\partial_\theta$ of $L$: $\partial_\theta(x)=\theta(\lambda)x$ for all $x\in L^\lambda$.
Form the semi-product Lie algebra $L_e=L\oplus \mathcal D_L$, where $\mathcal D_L=\{\partial_\theta\mid\theta\in \mathrm{Hom}_\Z(\Lambda,\C)\}$.
Note that
for any $\lambda\in \Lambda$, $L^\lambda=\{x\in L\mid [\partial_\theta,x]=\theta(\lambda) x\ \te{for all}\ \theta\in \mathrm{Hom}_\Z(\Lambda,\C)\}$
\cite[Section 8]{N2}. We have:

\begin{lem}\label{chawhen} Let $\CL$ be a nullity $n$ Lie torus that is centrally closed. Then $\CL_e\in \wh{E}_n$.
Conversely, any algebra in $\wh{E}_n$ is of this form.
\end{lem}
\begin{proof} The proof follows Neher's construction of maximal EALAs, and a sketch is given here; see 
\cite{N2} or \cite{Neher-EALA-survey} for details.
Let $L$ be a centerless Lie torus of type $(S,\Lambda)$ and $\Gamma$ the central grading group of $L$, which is a free
subgroup of $\Gamma$.
Let $D=\te{SCDer}(L)$ be the space of skew centroidal derivations on $L$, which is $\Gamma$-graded with the degree zero space $D^0=\mathcal D_L$.
Let $D^{gr*}$ the $\Gamma$-graded dual of $D$ and let $\tau:D\times D\rightarrow D^{gr*}$ be an affine two-cocycle.
Then there is a maximal EALA structure on $\E(L,\tau)=L\oplus D^{gr*}\oplus D$ with the Cartan subalgebra
$\CH(L)=L_0^0\oplus D^{0*}\oplus D^0$ and the core $\E(L,\tau)_c=L\oplus D^{gr*}$. Conversely, any maximal EALA arises in this way.

Note that one of the conditions for $\tau$ requires that
 $\tau(\mathcal D_L, D)=0$.
 This gives that the Lie bracket on $\overline{\E(L,\tau)}=L\oplus D^{gr*}\oplus \mathcal D_L$ is independent from the choice of $\tau$.
 In particular, we have  $\overline{\E(L,\tau)}=(L\oplus D^{gr*})_e$ as a Lie algebra.
 Now, let $\CL$ be a centrally closed Lie torus of nullity $n$. Then by Remark \ref{centerless}, there is a centerless Lie torus $L$ such that
 $\CL\cong \mathfrak u(L)\cong L\oplus D^{gr*}$ and hence
 \[\CL_e\cong (L\oplus D^{gr*})_e=\overline{\E(L,\tau_0)}\in \wh{E}_n,\]where $\tau_0$ denotes
 the trivial two-cocycle on $D$.
 On the other hand, let $\bar{\E}\in \wh{E}_n$ and by Neher's construction we may assume $\E$ has the form $\E(L,\tau)$.
 Then we have $\bar{\E}=\overline{\E(L,\tau)}=(L\oplus D^{gr*})_e$, as required.
\end{proof}

\subsection{Quantization of nullity $2$ EALAs}
In this subsection, we first classify the algebras in $\wh{E}_2$ and then consider their  quantization.
Throughout this subsection we always assume that $\g$ is of affine type $X_\wp^{(r)}$, where $\wp\in\Z_+$ and $r=1,2,3$.
For convenience, we set $I=\{0,1,\dots,\ell\}$.

Let $\dot\g$ be a finite dimensional simple Lie algebra of type $X_\wp$ and
$\dot\nu$ a diagram automorphism of $\dot\g$ with order $r$.
For  $\dot{x}\in \dot\g$ and $m\in \Z$, we set
\begin{align*}
\dot{x}_{[m]}=\sum_{p\in \Z_r}\xi_r^{-mp}\dot{\nu}^p(\dot{x})\quad\te{and}\quad
\dot\g_{[m]}=\te{Span}_\C\{\dot{x}_{[m]}\mid \dot{x}\in \dot\g\}.
\end{align*}
It was shown in \cite[Chap.\,8]{Kac-book} that the affine Kac-Moody algebra $\g$ can be realized as the Lie algebra
\begin{align}\label{identifyg}\g=(\sum_{m\in \Z}t_2^m\ot \dot\g_{[m]})\oplus \C\rk_2\oplus \C\rd_2\end{align}
 with the Lie bracket given by
\begin{align*}
&[t_2^{m_1}\ot \dot{x}+a_1\rk_2+b_1\rd_2, t_2^{m_2}\ot \dot{y}+a_2\rk_2+b_2\rd_2]\\
=\,&t_2^{m_1+m_2}\ot [\dot{x},\dot{y}]+(\dot{x}|\dot{y})\delta_{m_1+m_2,0}m_1\rk_2
+b_1m_2t_2^{m_2}\ot \dot{y}-b_2m_1t_2^{m_1}\ot \dot{x},
\end{align*}
where $m_1,m_2\in \Z$, $\dot{x}\in \dot\g_{[m_1]},\dot{y}\in \dot\g_{[m_2]}$, $a_1,a_2,b_1,b_2\in \C$
and $(\cdot|\cdot)$ is a nondegenerate  invariant symmetric bilinear form on $\dot\g$.
Let $\dot{\h}$ be a Cartan subalgebra of $\dot\g$ and take
\[\h=(\dot{\h}\cap \dot{\g}_{[0]})\oplus \C\rk_2 \oplus \C\rd_2\]
as the Cartan subalgebra of $\g$.
We also realize the simple roots $\al_i$'s and simple coroots $\al_i^\vee$'s as in \cite{Kac-book}.
Note that $\rd_2\in \h$ is a scaling element such that $\al_i(\rd_2)=\delta_{i,0}$ for $i\in I$.

\begin{rem}\label{rem:invform}{\em Note that $(\cdot|\cdot)$ induces an invariant bilinear form on $\g'$ such that
\begin{align*}(t_2^{m_1}\ot \dot{x}+a_1\rk_2|t_2^{m_2}\ot \dot{y}+a_2\rk_2)
=\delta_{m_1+m_2,0}(\dot{x}|\dot{y}).
\end{align*}
Recall from \cite[Exercise 2.5]{Kac-book} that any two such forms on $\g'$ are proportional.
Thus, we may (and do) fix the choice of $(\cdot|\cdot)$ on $\dot{\g}$ such that $(\cdot|\cdot)=\<\cdot|\cdot\>|_{\g'\times \g'}$.
}
\end{rem}

Let $\mathcal K$ be a vector space over $\C$ spanned by the symbols
\begin{align*}
t_1^{m_1} t_2^{m_2} \rk_i,\quad i=1,2,\ m_1\in \Z,\ m_2\in r\Z
\end{align*}
and subject to the relations
$
m_1t_1^{m_1} t_2^{m_2} \rk_1+m_2t_1^{m_1} t_2^{m_2} \rk_2=0.
$
Let
\[\ft(\g)=\sum_{m_1,m_2\in \Z} t_1^{m_1}t_2^{m_2}\ot \dot\g_{[m_2]}  \oplus \mathcal K\]
be the toroidal Lie algebra associated with $\g$, where $\mathcal K$ is the center space and
\begin{align}\label{toroidalre}
[t_1^{m_1}t_2^{m_2}\ot \dot{x}, t_1^{n_1}t_2^{n_2}\ot \dot{y}]
=t_1^{m_1+n_1}t_2^{m_2+n_2}\ot [\dot{x},\dot{y}] + (\dot{x}|\dot{y})\sum_{i=1,2}m_i t_1^{m_1+n_1}t_2^{m_2+n_2}\rk_i
\end{align}
for $\dot{x}\in \dot{\g}_{[m_2]}$, $\dot{y}\in \dot{\g}_{[n_2]}$ and $m_1,n_1,m_2,n_2\in \Z$.
With the identification \eqref{identifyg}, for $x=t_2^{m_2}\ot \dot{x}+a\rk_2\in \g'$ ($m_2\in \Z, \dot{x}\in \dot{\g}, a\in \C$) and $m_1\in \Z$, let us set
\begin{align}\label{t1m1x}
t_1^{m_1}\ot x:=t_1^{m_1}t_2^{m_2}\ot \dot{x}+a t_1^{m_1}\rk_2\in \ft(\g).
\end{align}
Note that these elements together with the central elements
$t_1^{m_1}t_2^{m_2}\rk_1$ $(m_1\in \Z, m_2\in r\Z\setminus\{0\})$, $\rk_1$ span the algebra $\ft(\g)$.
Denote by $\Delta$, $\Delta^{re}$ and $\Delta^{im}$
the sets of roots, real roots and imaginary roots in $\g$, respectively.
In view of Remark \ref{rem:invform}, the commutator \eqref{toroidalre} can be rewritten as follows:

\begin{lem}\label{lem:commutator} Let $\al, \beta\in \Delta$, $x\in \fg_\al, y\in \fg_\beta$ and $m_1,n_1\in \Z$.
If $\al+\beta\in \Delta^{re}\cup \{0\}$, then we have
\begin{align}\label{commutator1}
[t_1^{m_1}\ot x, t_1^{n_1}\ot y]=t_1^{m_1+n_1}\ot [x,y]+m_1\delta_{m_1,n_1}\<x \mid y\>\rk_1.\end{align}
If $x=t_2^{m_2}\ot \dot{x}$, $y=t_2^{n_2}\ot \dot{y}$ and $\al+\beta\in \Delta^{im}\setminus\{0\}$, then
\begin{align}\label{commutator2}
[t_1^{m_1}\ot x, t_1^{n_1}\ot y]=t_1^{m_1+n_1}\ot [x,y]
+(\dot{x}\mid \dot{y})\frac{m_1n_2-m_2n_1}{m_2+n_2}t_1^{m_1+n_1}t_2^{m_2+n_2}\rk_1.
\end{align}
\end{lem}

Note that the elements
\begin{align}\label{geneoftg}
t_1^{m}\ot e_i^\pm,\quad t_1^m\ot \al_i^\vee,\quad \rk_1,\ m\in \Z,\ i\in I
\end{align}
generate the algebra $\ft(\g)$.
From \cite{CJKT-uce},
there is an automorphism $\hat\mu$ of $\ft(\g)$ with
\begin{equation}\label{defhatmu}
t_1^m\ot e_i^\pm\mapsto \xi^{-m} t_1^m\ot e_{\mu(i)}^\pm,\quad t_1^m\ot \al_i^\vee\mapsto \xi^{-m}t_1^m\ot \al_{\mu(i)}^\vee,\quad
\rk_1\mapsto \rk_1
\end{equation}
for $i\in I,\ m\in \Z$.
Denote by $\ft(\g,\mu)$ the subalgebra of $\ft(\g)$ fixed by $\hat\mu$.
Using \eqref{geneoftg} and \eqref{defhatmu}, we see that $\ft(\g,\mu)$ is generated by  the elements
$t_1^m\ot e_{i(m)}^\pm, t_1^m\ot \al_{i(m)}^\vee, \rk_1$ for $i\in I, m\in \Z$.
We have

\begin{lem}\label{phigmu}  The assignment $(i\in I, m\in \Z)$
\begin{align}\label{eq:phigmu}
t_1^m\ot e_{i(m)}^\pm\mapsto t^m\ot e_{i(m)}^\pm,\ t_1^m\ot \al_{i(m)}^\vee\mapsto t^m\ot \al^\vee_{i(m)},\ \rk_1\mapsto \bm c.
\end{align}
determines a  surjective
homomorphism
$\phi_{\g,\mu}:\ft(\g,\mu)\rightarrow \CL(\g,\mu)$, which
is a universal central extension.
\end{lem}
\begin{proof} From \eqref{eq:affg} and Lemma \ref{lem:commutator}, it follows that the map
\begin{align}\label{eq:centralextun}
\phi:\ft(\g)\rightarrow \CL(\g),\quad t_1^m\ot x\rightarrow t^m\ot x,\ \rk_1\mapsto \bm c,\ t_1^{m_1}t_2^{m_2}\rk_1\mapsto 0
\end{align}
is a central extension, where $m\in \Z, x\in \g, m_1\in \Z$ and $m_2\in r\Z\setminus\{0\}$. Recall that $\CL(\g,\mu)$ is the $\hat\mu$-fixed
point subalgebra of $\CL(\g)$.
As  $\phi\circ\hat\mu=\hat\mu\circ \phi$ on the generators  in \eqref{geneoftg},
$\phi$ induces a central extension $\phi_{\g,\mu}$ from $\ft(\g,\mu)$ to $\CL(\g,\mu)$, which is determined by \eqref{eq:phigmu}.
The fact that $\phi_{\g,\mu}$ is universal follows from
\cite[Theorem 3.3]{CJKT-uce}.
\end{proof}

By adding two canonical derivations $\rd_1,\rd_2$ to $\ft(\g)$, we obtain the  Lie algebra
\begin{align*}
\hat{\ft}(\g)=\ft(\g)\oplus \C\rd_1\oplus \C\rd_2,
\end{align*}
where ($i,j=1,2$, $\dot{x}\in \dot{\g}_{[m_2]}$ and $m_1,m_2\in \Z$)
\begin{align}\label{d1d2}
[\rd_i, t_1^{m_1}t_2^{m_2}\ot \dot{x}]=m_it_1^{m_1}t_2^{m_2}\ot \dot{x}\quad\te{and}\quad
[\rd_i,t_1^{m_1}t_2^{m_2}\rk_j]=m_i t_1^{m_1}t_2^{m_2}\rk_j.
\end{align}
With the notation given in \eqref{t1m1x}, we in particular have
\begin{align}\label{bracketdi}
[\rd_2, t_1^{m_1}\ot x]=t_1^{m_1}\ot [\rd_2,x_{(m)}],\quad
[\rd_1, t_1^{m_1}\ot x_{(m)}]=m_1 t_1^{m_1}\ot x.
\end{align}

One notices that  $\mu(\rd_2)\in\rd_2+\h'$ (see the proof of \cite[Proposition 7.2.1]{ABP}), $\h=\h'\oplus \C \rd_2=\h'\oplus \h''$ and $\h''$ is $\mu$-invariant.
Set
\begin{align}
\check\rd_2=1/N\sum_{p\in \Z_N}\mu^p(\rd_2)\in \check{\h}.
\end{align}
Then we have $\check{\h}=\h'_{(0)}\oplus \C\check\rd_2$ and $\check{\rd}_2$ is a scaling element in $\check\h$ such that
$\check{\al}_i(\check{\rd}_2)=\delta_{i,0}$ for $i\in \check{I}$ (see Section 2.2).
 By  \eqref{commutator1} and
\eqref{bracketdi}, one immediately has

\begin{lem}\label{lem:actiondi} For $i\in I$ and $m\in \Z$, one has
\begin{align*}
[\check{\rd}_2, t_1^m\ot \al_{i(m)}^\vee]&=0,\ [\check\rd_2, t_1^m\ot e_{i(m)}^\pm]=
\pm \al_i(\check\rd_2) t_1^m\ot e_{i(m)}^\pm,\ [\check\rd_2,\rk_1]=0,\\
[\rd_1, t_1^m\ot \al_{i(m)}^\vee]&=m\,t_1^m\ot \al_{i(m)}^\vee,\ [\rd_1, t_1^m\ot e_{i(m)}^\pm]=
m\,t_1^m\ot e_{i(m)}^\pm,\ [\rd_1,\rk_1]=0.
\end{align*}
\end{lem}

In view of Lemma \ref{lem:actiondi}, we have the following subalgebra of $\hat{\ft}(\g)$:
\begin{align}\label{eq:hatftgmu}
\hat{\ft}(\g,\mu)=\ft(\g,\mu)\oplus \C\rd_1\oplus \C\check\rd_2.
\end{align}

\begin{rem}\label{rem:liftder}{\em
 Recall that $\hat\CL(\g,\mu)=\CL(\g,\mu)\oplus \C\check{\rd}_2\oplus \C\bm d$.
View $\check{\rd}_2,\bm d$ (resp.\,$\check{\rd}_2, \rd_1$) as derivations on $\CL(\g,\mu)$
(resp.\,$\ft(\g,\mu)$ by the adjoint actions.
Via the universal central extension $\phi_{\g,\mu}$, the derivations  $\check{\rd}_2,\bm d$ on $\CL(\g,\mu)$ can be extended
uniquely to $\ft(\g,\mu)$, which are exactly $\check{\rd}_2$, $\rd_1$ (see Lemma \ref{lem:actiondi}).}
\end{rem}

Recall that $\CL(\g,\mu)$ is a Lie torus with the $\Lambda$-grading given by \eqref{lambdagrading}.
By Remark \ref{uniLietorus}, $\ft(\g,\mu)$ is also a Lie torus with the $\Lambda$-grading
determined by
$\phi_{\g,\mu}(\ft(\g,\mu)^\lambda)=\CL(\g,\mu)^{\lambda}$ for $\lambda\in \Lambda$.
Using  Remark \ref{rem:liftder} and \eqref{lambdagrading}, we have
\begin{align}
\ft(\g,\mu)^{m\delta_2+n\delta}=\{x\in \ft(\g,\mu)\mid [\check{\rd}_2,x]=mx,\ [\rd_1, x]=nx\}.
\end{align}
Namely, via the adjoint action, $\C\check{\rd}_2\oplus \C\rd_1$ is the space of degree derivations on $\ft(\g,\mu)$.
From Lemma \ref{chawhen}, it follows that $\hat\ft(\g,\mu)=\ft(\g,\mu)_e\in \wh{E}_2$.

On the other hand, let $p$ be a generic complex number. Set $\C_p$ to be the quantum $2$-torus $\C_p$ associated to $p$, which is
 the associative algebra with underlying space $\C[t_1^{\pm 1},t_2^{\pm 1}]$ subject to
the basic commutation relation $t_2t_1=pt_1t_2$.
Let $\ell$ be a positive integer as before. Denote by $\fgl_{\ell+1}(\C_q)$ the matrix Lie algebra over $\C_p$ and
 set $\fsl_{\ell+1}(\C_p)=[\fgl_{\ell+1}(\C_p),\fgl_{\ell+1}(\C_p)]$, the derived subalgebra.
 Furthermore, let $\tilde{\fsl}_{\ell+1}(\C_p)$ denote the universal central extension of $\fsl_{\ell+1}(\C_p)$.
By adding two derivation $\rd_1,\rd_2$ to  $\tilde{\fsl}_{\ell+1}(\C_p)$, one obtains a nullity $2$ EALA
 $\hat{\fsl}_{\ell+1}(\C_p)$ with Lie brackets as in \eqref{d1d2}  \cite{BGK}.
 Note that the extended core of $\hat{\fsl}_{\ell+1}(\C_p)$ is itself and hence $\hat{\fsl}_{\ell+1}(\C_p)\in \wh{E}_2$.
Following \cite{ABP}, we have the following classification result.

\begin{prop}\label{prop:e2} Every algebra in $\wh{E}_2$ is either isomorphic to $\hat{\ft}(\g,\mu)$ with $\g$ affine, or
isomorphic to $\hat{\fsl}_{\ell+1}(\C_p)$ with $p$ generic.
\end{prop}
\begin{proof} From Remark \ref{centerless} and Lemma \ref{chawhen}, it follows that any algebra in $\wh{E}_2$ has the form
$\mathfrak u(L)_e$ for some centerless Lie torus $L$ of nullity $2$.
It was proved in \cite{ABP} that every centerless Lie torus of nullity $2$  is either isomorphic to
$\mathrm{Aff}(\g,\mu)_{cc}=\CL(\g,\mu)/Z(\CL(\g,\mu))$ or isomorphic to $\fsl_{\ell+1}(\C_p)$.
Note that  $\hat{\fsl}_{\ell+1}(\C_p)=\hat{\fsl}_{\ell+1}(\C_p)_e\cong\mathfrak u(\fsl_{\ell+1}(\C_p))_e$.
Furthermore, we have $\mathfrak u(\mathrm{Aff}(\g,\mu)_{cc})\cong \mathfrak u(\CL(\g,\mu))\cong  \ft(\g,\mu)$ and hence
$\mathfrak u(\mathrm{Aff}(\g,\mu)_{cc})_e\cong \hat{\ft}(\g,\mu)$, as desired.
\end{proof}

Now we are ready to give the quantization of the algebras classified in Proposition \ref{prop:e2}.
Note that when  $\g$ is of type $A_1^{(1)}$, it follows from Lemma \ref{lem:LC1aff} that
$\mu=\mathrm{id}$. In this case, the quantum affinization algebra $\U_\hbar^{JN}$ defined in \cite{J-KM,Naka-quiver}  is slightly different from that
of $\qtar$: the relations (Q8) with $i\ne j$ is now replaced by the relations
\begin{align*}
F_{ij}^\pm(z,w)(z-w)x_i^\pm(z)x_j^\pm(w)=G^\pm_{ij}(z,w)(z-w)x_j^\pm(w)x_i^\pm(z).
\end{align*}
Define $\U_\hbar^{new}$ to be the quotient  $\C[[\hbar]]$-algebra of $\U_\hbar^{JN}$
by modulo the relations
\begin{align*}
 [x_i^\pm(z_1),\(F_{ij}^\pm(z_2,w)x_i^\pm(z_2)x_j^\pm(w)
    -G_{ij}^\pm(z_2,w)x_j^\pm(w)x_i^\pm(z_2)\)]=0,\quad \te{for }i\ne j.
\end{align*}
Just as $\qtar$, it was proved in  \cite{CJKT-quantum-A111} that  $\U_\hbar^{new}$ has a triangular decomposition and a topological Hopf algebra structure.
We also denote by $\mathcal U_{\hbar,p}({\dwh{\frak sl}}_{\ell+1})$  the two parameters quantum toroidal algebra defined in \cite[Definition 2.2.1]{Sa}.
As the main result of this section, we have:

\begin{thm}\label{thm:quaneala} Let $\g$ be of affine type and $p$ a generic number. Then we have
\begin{enumerate}
\item $\mathcal U_{\hbar,p}({\dwh{\frak sl}}_{\ell+1})|_{\hbar\mapsto 0}\cong \U(\hat\fsl_{\ell+1}(\C_p))$;\\
\item $\U_\hbar^{new}|_{\hbar\mapsto 0}\cong \U(\hat\ft(\g))$ if $\g$ is of type $A_1^{(1)}$;\\
\item $\qtar|_{\hbar\mapsto 0}\cong \U(\hat\ft(\g,\mu))$ if $\g$ is of non-$A_1^{(1)}$ type.
\end{enumerate}
\end{thm}
 \begin{proof}
The first assertion was proved in \cite{VV-double-loop}, and the second assertion was proved in \cite{CJKT-quantum-A111}.
Now we prove the last isomorphism by using the Drinfeld type presentations of $\ft(\fg,\mu)$ established in \cite[Theorem 1.4]{CJKT-drin-pre} ($\ft(\fg,\mu)$ is denoted as
$\hat\fg[\mu]$ therein). Recall that the subalgebra $\hat\fg_\mu'$ of $\hat\fg_\mu$  is abstractly generated by the elements $X_{i,m}^\pm, H_{i,m}, C$
with relations (L0), (L2) and (L4-L8).
We claim that the assignment ($i\in I, m\in \Z$)
\begin{align}\label{eq:isogt}
H_{i,m}\mapsto  t_1^m\ot r_i \al_{i(m)}^\vee,\quad
X_{i,m}^\pm\mapsto t_1^m\ot e^\pm_{i(m)},\quad C\mapsto \frac{\rk_1}{N}
\end{align}
determines a Lie algebra isomorphism from $\hat\fg_\mu'$ to $\ft(\fg,\mu)$.

We first introduce a family of polynomials as follows
\begin{align*}
\bm{P}=\{P_{ij,\sigma}(z_1,\dots,z_{1-a_{ij}},w)\mid i,j\in I\ \te{with}\ a_{ij}<0\ \te{and}\ \sigma\in S_{1-a_{ij}}\},
\end{align*}
where
\begin{equation}
P_{ij,\sigma}(z_1,\dots,z_{1-a_{ij}},w)=\begin{cases}
p_{i}(z_{\sigma(1)},z_{\sigma(2)},-w),\quad&\text{if}\ i\in\mathcal{O}(j);\\
p_{ij,0}(z_1,\dots,z_{1-a_{ij}},w),\quad&\text{if}\ i\notin\mathcal{O}(j)\ \te{and}\ \sigma=1;\\
0,\quad&\text{if}\ i\notin\mathcal{O}(j)\ \te{and}\ \sigma\ne 1.
\end{cases}
\end{equation}
Note that if  $i\in\mathcal{O}(j)$ and $a_{ij}<0$, then
it follows from Lemma \ref{lem:linking} (ii) and (L0) that $X_i^\pm(w)=X_j^\pm(-w)$.
Thus, the relation (L6) can be rewritten as follows
\[\sum_{\sigma\in S_2} p_i(z_{\sigma(1)},z_{\sigma(2)},-w)
 [X_i^\pm(z_{\sigma(1)}),[X_i^\pm(z_{\sigma(2)}),X_j^\pm(w)]]=0.\]
This implies that relations (L6) and (L7) are equivalent to the following relations
\begin{align}\label{eq:newrelation}
\sum_{\sigma\in S_{1-a_{ij}}}P_{ij,\sigma}(z_1,\dots,z_{1-a_{ij}},w)[X_i^{\pm}(z_{\sigma(1)}),\cdots, [X_i^{\pm}(z_{\sigma(1-a_{ij})}),X_j^\pm(w)]]=0\end{align}
for all $i,j\in I$ with $a_{ij}<0$.
We will show in  the remark below that, under the correspondence \eqref{eq:isogt}, the relations \eqref{eq:newrelation}
(or equivalently, (L7) and (L8)) hold in $\ft(\fg,\mu)$.
This implies that the family $\bm{P}$ satisfies the condition $(\bm{P}1)$ stated in \cite[\S 1.1]{CJKT-drin-pre}, noting that $\mathbb I$
denotes the set $\{(i,j)\in I\times I\mid a_{ij}<0\}$ therein.
Then the algebra $\wh\fg_\mu'$ is exactly the algebra $\mathcal{D}_{\bm{P}}(\g,\mu)$ defined in \cite[Definition 1.2]{CJKT-drin-pre},
by identifying the generators $X_{i,m}^\pm, H_{i,m}, C$ of $\wh\fg_\mu'$
with the generators  $x_{i,m}^\pm, h_{i,m}, c$  of $\mathcal{D}_{\bm{P}}(\g,\mu)$.

One can easily check that the family $\bm{P}$ satisfies the condition $(\bm{P}2)$ stated in \cite[\S 1.1]{CJKT-drin-pre}, which says that
\[\sum_{\sigma\in S_{1-a_{ij}}}P_{ij,\sigma}(w,\dots,w,w)\ne 0\quad\text{for all}\ i,j\in I\ \te{with}\ a_{ij}<0.\]
Then the claim follows immediately from  \cite[Theorem 1.4]{CJKT-drin-pre}.
We denote by $\psi$ the resulting isomorphism induced by  \eqref{eq:isogt}.
In view of \eqref{hatgmucd} and \eqref{eq:hatftgmu}, $\psi$ can be extended to an isomorphism from $\wh\fg_\mu$ to $\hat{\ft}(\fg,\mu)$ such that
\[L_{\check\rd_2}\mapsto \check\rd_2,\qquad  D\mapsto \rd_1.\]
This together with  Proposition \ref{prop:classical-limit} proves the assertion (3) and hence the theorem.
\end{proof}

\begin{rem}
{\em Here we show that the relations \eqref{eq:newrelation} hold in $\ft(\fg,\mu)$, namely,
\begin{align}\label{eq:newrelation2}
\sum_{\sigma\in S_{1-a_{ij}}}P_{ij,\sigma}(z_1,\dots,z_{1-a_{ij}},w)[e_i^{\pm}(z_{\sigma(1)}),\cdots, [e_i^{\pm}(z_{\sigma(1-a_{ij})}),e_j^\pm(w)]]=0\end{align}
for all $i,j\in I$ with $a_{ij}<0$, where $e_i^\pm(z)=\sum_{n\in \Z} (t_1^n\ot e_{i(n)}^\pm) z^{-n}$. In fact, as in the proof of Proposition \ref{prop:classical-limit},
we have the following stronger result: if $i\in \mathcal{O}(j)$, then
$[e_i^{\pm}(z_1),[e_i^{\pm}(z_2),e_j^\pm(w)]]=0$,
 and if
$i\notin \mathcal{O}(j)$, then
\begin{align}\label{eq:newrelation3}
\prod_{1\le a<b\le 1-a_{ij}} p'_{ij}(z_a,z_b)\cdot [e^\pm_i(z_1),[e^\pm_i(z_2),\cdots,[e^\pm_i(z_{1-a_{ij}}),e^\pm_j(w)]]]=0.\end{align}
Recall from Lemma \ref{lem:LC1aff} (3) and \eqref{eq:defpij'} that $p_{ij}'(z,w)=p_{ij}(z,w)$ except for the case that
\begin{eqnarray}\label{eq:excecase}
\te{$\g$ is of type $A_{2l+1}^{(1)}$ $(l\ge 2)$ and
$\mu$ has order $l+1$.}\end{eqnarray}
We remark that the above stronger relations have been proved in \cite[Proposition 3.14]{CJKT-drin-pre} except the case \eqref{eq:excecase}.
For the case \eqref{eq:excecase}, note that in this case
$\alpha_{\mu^{k_1}(i)}+\alpha_{\mu^{k_2}(i)}+\alpha_j\in \Delta$ for some $k_1,k_2\in \Z_N$ if and only if
$\{\mu^{k_1}(i),\mu^{k_2}(i),j\}=\{i,i',j\}$, where $i'$ denotes another vertex (different from $i$) linking to $j$.
This means that $\alpha_{\mu^{k_1}(i)}+\alpha_{\mu^{k_2}(i)}+\alpha_j\in \Delta$ if and only if $k_1-k_2\in \Omega_{ij}$ (see \eqref{eq:defpij'}).
In particular, these roots cannot be null roots, which must have the form $m\sum_{i\in I}\alpha_i$, $m\in \Z$.
So by Lemma \ref{lem:commutator} it follows that
 the relations \eqref{serre-ex} still hold in $\ft(\fg,\mu)$ (see also \cite[Lemma 3.15]{CJKT-drin-pre}).
This gives that the relations \eqref{eq:newrelation3} hold for this  exceptional case.

Here we also correct two mistakes in \cite{CJKT-drin-pre}.
Firstly, the sentence ``In particular, in this case $\Gamma_{ij}^-$ is a subgroup of $\Z_N$ with order $d_{ij} = N/N_{ij}$''
stated in \cite[Lemma 3.9]{CJKT-drin-pre}
 is incorrect for the case \eqref{eq:excecase}. But this statement was not used in the remaining part of \cite{CJKT-drin-pre}.
Secondly, when $j\notin \mathcal{O}(i)$, the polynomial $p_{ij}(z,w)$ computed in \cite[Lemma 3.13]{CJKT-drin-pre} should be replaced by
$p_{ij}'(z,w)$ for the case \eqref{eq:excecase}.}
\end{rem}

\begin{rem}{\em Recall that the center $\mathcal{K}$ of $\ft(\g)$ has a basis
\[\{\rk_1,\, t_1^{m_1}\rk_2,\,t_1^{m_1}t_2^{m_2}\rk_1\mid m_1\in \Z,\ m_2\in r\Z\setminus\{0\}\}.\]
And, it follows from (the proof of) \cite[Lemma 3.2]{CJKT-uce} that the space
\[\text{Span}_\C\{\rk_1,\, t_1^{m_1}\rk_2,\,t_1^{m_1}t_2^{m_2}\rk_1\mid m_1\in N\Z,\ m_2\in r\Z\setminus\{0\}\}\]
is the center of $\ft(\g,\mu)$.
When $\g$ is of non-$A_1^{(1)}$ type, by identifying $\hat\g_\mu$ with $\hat\ft(\g,\mu)$ (see Theorem \ref{thm:quaneala}),
it follows from  \eqref{eq:centralextun} that the kernel of
$\psi_{\g,\mu}:\hat\g_\mu=\hat\ft(\g,\mu)\rightarrow \hat{\CL}(\g,\mu)$ is
\[\hat{\mathcal{K}}=\text{Span}_\C\{t_1^{m_1}t_2^{m_2}\rk_1\mid m_1\in N\Z,\ m_2\in r\Z\setminus \{0\}\}.\]
When $\g$ is of $A_1^{(1)}$ type  (and so $\mu=\mathrm{id}$),  $\hat\fg_\mu$ is isomorphic to
the quotient algebra
$\hat{\ft}(\g)/\bar{\mathcal K}$ (cf.\,\cite{E-PBW-qaff}), where $\bar{\mathcal K}=\sum_{m_1\in \Z}(\C t_1^{m_1}t_2\rk_1+\C t_1^{m_1}t_2^{-1}\rk_1)$.
Thus the kernel of $\psi_{\g,\mu}:\hat\g_\mu=\hat\ft(\g)/\bar{\mathcal K}\rightarrow \hat{\CL}(\g,\mu)$ is $\hat{\mathcal{K}}/\bar{\mathcal{K}}$.
}
\end{rem}

\section{Proof of Theorem \ref{prop:Dr-to-normal-ordering}}\label{sec:pf-prop-aff-q-serre}

This section is devoted to the proof of Theorem \ref{prop:Dr-to-normal-ordering}.
Throughout this section, let $(i,j)\in\mathbb I$ and  $W\in \R^s_\pm$ be fixed.
For a finite set $J$ and a formal variable $z$, we denote by $\underline z_J$ the map from $J$
to the set $\set{z_a}{a\in J}$ of formal variables defined by $a\mapsto z_a$.
If there is no ambiguity, we simply denote $\underline z_J$ by $\underline z$.
 For simplicity,
we also write $\frac{1}{z-w}=\iota_{z,w}\frac{1}{z-w}$ as usual.

\subsection{On $g$-commutators}
In this subsection we calculate the $g$-commutators
$[x_i^\pm(z_{1}),x_i^\pm(z_{2}),$ $\dots,x_i^\pm(z_{r}),x_j^\pm(w)]_g$ with $r\in \N$.

We start with some notations. Denote by $K=\<{-1},\xi_{d_{ij}/d_i}\>$ the subgroup of $\C^\times$ generated by $-1$ and $\xi_{d_{ij}/d_i}$,
and denote by $K^\nu=(-1)^{\nu}\<\xi_{d_{ij}/d_i}\>$, where $\nu\in \Z_2=\{0,1\}$.
For $c\in K^0$ and $r\in \N$, set $A_{c,r}^\nu=\emptyset$ if $r=0$ and set
\begin{align*}
A_{c,r}^\nu=\{((-1)^\nu c,a_{ij}-\nu+2p)\in K\times \Z\mid 0\le p\le r-1\}\quad\te{if}\ r>0.
\end{align*}
Let ${\mathscr M}'$ be the set consisting of  subsets of $K\times \Z$ which are a finite union of the $A_{c,r}^\nu$'s with $c\in K^0, r\in \N$ and $\nu\in \Z_2$.
For $M\in {\mathscr M}'$, $c\in K^0$ and $\nu\in \Z_2$, define
\begin{align*}
p_c^\nu=\max\{r-1\mid r\in\N\ \te{with}\ A_{c,r}^\nu\subset M\}.
\end{align*}
Form the sets
\begin{align*}
{\mathscr M}&=\{M\in {\mathscr M}'\mid p_c^1=-1\ \te{if}\ s_i=1\ \te{and}\ p_c^1\le p_c^0\ \te{if}\ s_i=2\ \te{for all}\ c\in K^0\},\\
\mathscr M_r&=\{M\in \mathscr M\mid |M|=r\}\quad\te{for}\quad r\in \N.
\end{align*}
For $M\in\mathscr M$ and $\nu\in \Z_2$, we further introduce the following sets:
\begin{align*}
&M^\nu=\{((-1)^\nu c,a_{ij}-\nu+2p)\in M\mid c\in K^0,\ 0\le p\le p_c^\nu\},\\
  &\partial M=\set{a\in (K\times \Z)\setminus M}{M\cup \{a\}\in\mathscr M},\\
&\partial^\ast M=\set{(-c,a_{ij}+2p_c^1+1)\in\partial M}{c\in K^0, 0\le p_c^1<p_c^0}.
\end{align*}
We also define a map as follows:
\begin{align}\label{eq:def-tau-map-valuation}
\tau^\pm:K\times \Z\to \C[[\hbar]],\quad  (c,n)\mapsto c^{1/d_i}\,q_i^{\pm n},
\end{align} where $c^{1/d_i}= \xi_{d_i|K|}^s$ if $c=\xi_{|K|}^s\in K$.
For the case $s_i=2$, we say that  $(a_1,a_2,a_3)\in (K\times \Z)^3$ is an $A_2^{(2)}$-triple if
there is a $\sigma\in S_3$ such that
\[\tau^\pm(a_{\sigma(1)})=\xi_{2d_i}q_i^{\pm 1}\tau^\pm(a_{\sigma(2)})=\xi_{2d_i}q_i^{\mp1}\tau^\pm(a_{\sigma(3)}).\]
In this case it follows from Proposition \ref{prop:iii-normal-ordering}, \eqref{normalordermu} and \eqref{normalordersigma} that for any $p\in \Z_{d_i}$,
\begin{align}\label{a22triple}
\:x_i^\pm(\xi_{d_i}^p\tau^\pm(a_1)w)x_i^\pm(\tau^\pm(a_2)w)x_i^\pm(\tau^\pm(a_3)w)\;=0.
\end{align}

\begin{rem}\label{rem:divides}{\rm Let $J$ be a finite set with $|J|=m$ and let $\psi(\underline{w}_J,w)\in \E_\hbar^{(m+1)}(W)$.
Assume that there is an $a\in J$ such that $\lim_{w_a\rightarrow \xi_{d_i}^pw}\psi(\underline{w}_J,w)=0$ for all $p\in \Z_{d_i}$.
Then for any $v\in W$, we have
\begin{align*}
&(w_a^{d_i}-w^{d_i})^{-1}\psi(\underline{w}_J,w)v+(w^{d_i}-w_a^{d_i})^{-1}\psi(\underline{w}_J,w)v\\
=&w_a^{-d_i}\delta(w_a^{d_i}/w^{d_i})\psi(\underline{w}_J,w)v=0.
\end{align*}
This gives  $(w_a^{d_i}-w^{d_i})^{-1}\psi(\underline{w}_J,w)v\in W((\underline{w}_J,w))+\hbar^nW[[w^{\pm 1},w_b^{\pm 1}\,|\,b\in J]]$ for all $n\in \Z_+$.
Thus one obtains
\begin{align}\label{wawew}(w_a^{d_i}-w^{d_i})^{-1}\psi(\underline{w}_J,w)\in \E_\hbar^{(m+1)}(W).
\end{align}
Assume further that there is another $b\in J$ such that $\lim_{w_b\rightarrow \xi_{d_i}^pw}\psi(\underline{w}_J,w)=0$ for all $p\in \Z_{d_i}$.
It follows from \eqref{wawew} that $\lim_{w_b\rightarrow \xi_{d_i}^pw}(w_a^{d_i}-w^{d_i})^{-1}\psi(\underline{w}_J,w)$ exists and is zero.
Then we have
\begin{align*}(w_b^{d_i}-w^{d_i})^{-1}(w_a^{d_i}-w^{d_i})^{-1}\psi(\underline{w}_J,w)\in \E_\hbar^{(m+1)}(W).
\end{align*}
In general, if there is a subset $J'$ of $J$ such that $\lim_{w_a\rightarrow \xi_{d_i}^pw}\psi(\underline{w}_J,w)=0$ for all $p\in \Z_{d_i}$ and $a\in J'$,
then we have
\begin{align*}\prod_{a\in J'}(w_a^{d_i}-w^{d_i})^{-1}\psi(\underline{w}_J,w)\in \E_\hbar^{(m+1)}(W).
\end{align*}
}
\end{rem}

Let $M\in\mathscr M_r$ with $r\in \Z_+$ and $M'\subset M$. In view of Proposition \ref{prop:normal-ordering-rational} and the fact that $C_{ii}=1$, we can define the following
currents on $W$:
\begin{align*}
  \wt X_{ij,r}^\pm(M,M',w,\underline w)&=\lim_{\substack{ w_a\to w, a\in M'}}
  \:\prod_{a\in M}x_i^\pm(\tau^\pm(a)w_a)x_j^\pm(w)\;\\
  &=\:\prod_{a\in M\setminus M'}x_i^\pm(\tau^\pm(a)w_a)\prod_{b\in M'}x_i^\pm(\tau^\pm(b)w)x_j^\pm(w)\;,\\
  \:x_i^\pm(z)\wt X_{ij,r}^\pm(M,M',w,\underline w)\;&=
  \lim_{\substack{ w_a\to w, a\in M'}}
  \:x_i^\pm(z)\prod_{a\in M}x_i^\pm(\tau^\pm(a)w_a)x_j^\pm(w)\;.
\end{align*}
As a convention, when $r=0$, we set $\wt X_{ij,r}^\pm(M,M',w,\underline w)=x_j^\pm(w)$.

Let $\bm{t}=(t_c\,|\,c\in K^0)\in \N^{d_{ij}/d_i}$ be a tuple satisfying the condition:
\begin{align}\label{conditiont}
 0\le t_c\le p_c^1\ \te{if}\ p_c^1\ge 0\quad\te{and}\quad t_c=0\ \te{if}\ p_c^1=-1.
 \end{align}
We define a subset of $M$ associated to $\bm{t}$ as follows:
\begin{align*}
  M_{\bm{t}}=&\set{(c,a_{ij}+2n)}{c\in K^0\ \te{with}\ p_c^0\ge 0,\,t_c\le n\le p_c^0, }\\
  &\cup\set{(-c,a_{ij}-1+2m)}{c\in K^0\ \te{with}\ p_c^1\ge 0,\, 0\le m\le t_c}.
\end{align*}
Note that $M=M_{\bm{t}}$ if $s_i=1$.
Furthermore, if  $M\setminus M_{\bm{t}}\ne \emptyset$ (and so $s_i=2$), then
for any $a\in M\setminus M_{\bm{t}}$ there exist $b,b'\in M_{\bm{t}}$ such that $(a,b,b')$ is an $A_2^{(2)}$-triple.
Indeed, if $a=(c,a_{ij}+2n)$ with $c\in K^0$ and $0\le n\le t_c-1$, then we may take $b=(-c,a_{ij}-1+2n)$ and $b'=(-c,a_{ij}+1+2n)$.
And, if $a=(-c,a_{ij}-1+2m)$ with $c\in K^0$ and $t_c<m\le p_c^1$, then we may take $b=(c,a_{ij}+2m-2)$ and $b'=(c,a_{ij}+2m)$.
In view of \eqref{a22triple} and Remark \ref{rem:divides}, we can continue to define
the following  currents on $W$:
\begin{align*}
  &X_{ij,r}^\pm(M,M_{\bm{t}}, w)=
  \lim_{\substack{w_a\to w\\ a\in M\setminus M_{\bm{t}}}}
  \prod_{a\in M\setminus M_{\bm{t}}}
  \frac{1}{\tau^\pm(a)^{d_i}\(w_a^{d_i}-w^{d_i}\)}\wt X_{ij,r}^\pm(M,M_{\bm{t}}, w,\underline w),\\
  &\:x_i^\pm(z)X_{ij,r}^\pm(M,M_{\bm{t}},w)\;=\lim_{\substack{w_a\to w\\ a\in M\setminus M_{\bm{t}}}}\prod_{a\in M\setminus M_{\bm{t}}}\frac{1}{\tau^\pm(a)^{d_i}\(w_a^{d_i}-w^{d_i}\)}\:x_i^\pm(z)\wt X_{ij,r}^\pm(M,M_{\bm{t}},w,\underline w)\;.
\end{align*}

\begin{lem}\label{lem:defXijr-2}
Let $\bm{t}=(t_c|c\in K^0)$ and $\bm{t}'=(t_c'|c\in K^0)$ be two tuples satisfying the condition  \eqref{conditiont}.
Then
\begin{align*}
  X_{ij,r}^\pm(M,M_{\bm{t}},w)=\(\prod_{c\in K^0}(-1)^{|t_c-t_c'|}\)
  \(\frac{\prod_{a\in M_{\bm{t}} } \tau^\pm(a)^{d_i} }
    {\prod_{a\in M_{\bm{t}'}} \tau^\pm(a)^{d_i} }\)
  X_{ij,r}^\pm(M,M_{\bm{t}'},w).
\end{align*}
\end{lem}

\begin{proof}
We first consider the case that there is a $c\in K^0$ such that  $t_c=t_c'+1$ and $t_d=t_d'$ for any $d\ne c\in K^0$.
Set $M_1=M_{\bm{t}}\cap M_{\bm{t}'}$ and $M_2=M_{\bm{t}}\cup M_{\bm{t}'}$.
One can verify that $M_{\bm{t}}=M_1\uplus \{b_1\}$ with $b_1=(-c,a_{ij}-1+2t_c)$ and $M_{\bm{t}'}=M_1\uplus \{b_2\}$ with
$b_2=(c,a_{ij}-2+2t_c)$.
As indicated before,  for any
$a\in M\setminus M_2$, there exist $a_1,a_2\in M_1$ such that $(a,a_1,a_2)$ is an $A_2^{(2)}$-triple.
From \eqref{a22triple} and Remark \ref{rem:divides}, it follows that
\begin{align*}
\psi(\underline{w},w):=\prod_{a\in M\setminus M_2}(w_a^{d_i}-w^{d_i})^{-1}\wt X_{ij,r}^\pm&(M,M_1,w,\underline w) \in \E_\hbar^{(m+1)}(W).
\end{align*}
where $m=|M\setminus M_1|$.
One notices from \eqref{normalordermu} that
\[\lim_{w_a\mapsto \xi^p_{d_i}w_a}\psi(\underline{w},w)=\psi(\underline{w},w)=\lim_{w\rightarrow \xi^p_{d_i}w}\psi(\underline{w},w)\quad \te{for}\ p\in \Z_{d_i},\ a\in M\setminus M_2.\]

Set $b_3=(c,a_{ij}+2t_{c})$ and $b_4=(-c,a_{ij}-3+2t_{c})$.
Then $b_3,b_4\in M_1$ and $(b_1,b_2,b_3)$, $(b_1,b_2,b_4)$ are both $A_2^{(2)}$-triples.
Again by \eqref{a22triple} and Remark \ref{rem:divides}, we see that the current
\begin{align*}
(w_{b_1}^{d_i}-w_{b_3}^{d_i})^{-1}
(w_{b_1}^{d_i}-w_{b_4}^{d_i})^{-1}\lim_{w_{b_2}\rightarrow w_{b_1}}\:x_i^\pm(\tau^\pm(b_1)w_{b_1})x_i^\pm(\tau^\pm(b_2)w_{b_2})x_i^\pm(\tau^\pm(b_3)w_{b_3})x_i^\pm(\tau^\pm(b_4)w_{b_4})\;
\end{align*}
lies in $\E_\hbar^{(3)}(W)$. By taking limits $w_{b_3}\rightarrow w$ and $w_{b_4}\rightarrow w$, we have
\[(w_{b_1}^{d_i}-w^{d_i})^{-2}
\lim_{w_{b_2}\rightarrow w_{b_1}}\:x_i^\pm(\tau^\pm(b_1)w_{b_1})x_i^\pm(\tau^\pm(b_2)w_{b_2})x_i^\pm(\tau^\pm(b_3)w)x_i^\pm(\tau^\pm(b_4)w)\;
\in \E_\hbar^{(2)}(W).\]
Thus, one gets
\[H_3(\underline w,w):=(w_{b_1}^{d_i}-w^{d_i})^{-2}\lim_{w_{b_2}\rightarrow  w_{b_1}}\psi(\underline{w},w) \in \E_\hbar^{(m)}(W).\]
By definition we have
\begin{align*}
\lim_{w_{b_2}\rightarrow  w_{b_1}}(\psi(\underline{w},w)-(w_{b_1}^{d_i}-w^{d_i})(w_{b_2}^{d_i}-w^{d_i})H_3(\underline w,w))=0.
\end{align*}
Then we can continue to define the current
\begin{align*}
H_4(\underline w,w)=(w_{b_2}^{d_i}-w_{b_1}^{d_i})^{-1}(\psi(\underline{w},w)-(w_{b_1}^{d_i}-w^{d_i})(w_{b_2}^{d_i}-w^{d_i})H_3(\underline w,w))\in \E_\hbar^{(m+1)}(W).
\end{align*}

Again by the fact that $(b_1,b_2,b_3)$ is an $A_2^{(2)}$-triple one obtains
\begin{align*}
H_1(\underline w,w)&=(w_{b_2}^{d_i}-w^{d_i})^{-1}\lim_{w_{b_1}\rightarrow  w}\psi(\underline{w},w)\in \E_\hbar^{(m)}(W),\\
H_2(\underline w,w)&=(w_{b_1}^{d_i}-w^{d_i})^{-1}\lim_{w_{b_2}\rightarrow  w}\psi(\underline{w},w)\in \E_\hbar^{(m)}(W).
\end{align*}
Then  we have
$
\lim_{w_{b_1}\rightarrow w}\psi(\underline{w},w)=(w_{b_2}^{d_i}-w^{d_i})H_1(\underline w,w),
$ while we also have
\begin{align*}
\lim_{w_{b_1}\rightarrow w}\psi(\underline{w},w)=\lim_{w_{b_1}\rightarrow w}(w_{b_2}^{d_i}-w_{b_1}^{d_i})H_4(\underline w,w)
=(w_{b_2}^{d_i}-w^{d_i})\lim_{w_{b_1}\rightarrow w}H_4(\underline w,w).
\end{align*}
Thus we get $H_1(\underline w,w)=\lim_{w_{b_1}\rightarrow w}H_4(\underline w,w)$.
Similarly, one has
\begin{align*}
(w_{b_1}^{d_i}-w^{d_i})H_2(\underline w,w)=\lim_{w_{b_2}\rightarrow w}\psi(\underline{w},w)=(w^{d_i}-w_{b_1}^{d_i})\lim_{w_{b_2}\rightarrow w}H_4(\underline w,w),
\end{align*}
which implies $H_2(\underline w,w)=-\lim_{w_{b_2}\rightarrow w}H_4(\underline w,w)$.
In summary, we obtain that
\begin{align}\label{H1=-H2}
\lim_{w_{b_2}\rightarrow w}H_1(\underline w,w)=\lim_{w_{b_2}\rightarrow w}\lim_{w_{b_1}\rightarrow w}H_4(\underline w,w)
=-\lim_{w_{b_1}\rightarrow w}H_2(\underline w,w).
\end{align}

Note that  $\wt X_{ij,r}^\pm(M,M_{\bm{t}}, w,\underline w)=\lim_{w_{b_1}\rightarrow  w}\wt X_{ij,r}^\pm(M,M_1, w,\underline w)$ and so we have
\begin{align*}
\prod_{a\in M\setminus M_{\bm{t}}}
  \frac{1}{w_a^{d_i}-w^{d_i}}\wt X_{ij,r}^\pm(M,M_{\bm{t}}, w,\underline w)
  =(w_{b_2}^{d_i}-w^{d_i})^{-1}\lim_{w_{b_1}\rightarrow  w}\psi(\underline{w},w)=
  H_1(\underline w,w).
  \end{align*}
This gives
\begin{align*}X_{ij,r}^\pm(M,M_{\bm{t}},w)=
\lim_{\substack{w_a\to w\\ a\in M\setminus M_{\bm{t}}}}
  \prod_{a\in M\setminus M_{\bm{t}}}
  \frac{1}{\tau^\pm(a)^{d_i}}\cdot H_1(\underline w,w)\\
  =\lim_{\substack{w_a\to w\\ a\in M\setminus M_{2}}}
  \prod_{a\in M\setminus M_{2}}\frac{1}{\tau^\pm(a)^{d_i}}\lim_{w_{b_2}\rightarrow w}\frac{1}{\tau^\pm(b_2)^{d_i}} H_1(\underline w,w).
\end{align*}
Similarly, one can get that
\begin{align*}X_{ij,r}^\pm(M,M_{\bm{t}'},w)=
\lim_{\substack{w_a\to w\\ a\in M\setminus M_{\bm{t}'}}}
  \prod_{a\in M\setminus M_{\bm{t}'}}
  \frac{1}{\tau^\pm(a)^{d_i}}\cdot H_2(\underline w,w)\\
  =\lim_{\substack{w_a\to w\\ a\in M\setminus M_{2}}}
  \prod_{a\in M\setminus M_{2}}\frac{1}{\tau^\pm(a)^{d_i}}\lim_{w_{b_1}\rightarrow w}\frac{1}{\tau^\pm(b_1)^{d_i}} H_2(\underline w,w).
\end{align*}
In view of \eqref{H1=-H2}, we obtain
\begin{align*}
X_{ij,r}^\pm(M,M_{\bm{t}},w)=-\frac{\tau^\pm(b_1)^{d_i} }
   {\tau^\pm(b_2)^{d_i} }
X_{ij,r}^\pm(M,M_{\bm{t}'},w),
\end{align*}
as required.
Finally, the assertion follows from an
 induction argument on the nonnegative integer $\sum_{c\in K^0}|t_c-t_c'|$.
\end{proof}

Let $\check{\bm{t}}=(t_c\mid c\in K_0)$ be the  tuple such that
 $t_c=0$ for all $c\in K^0$.
Set
\[\check M=M_{\check{\bm{t}}}=M^0\cup \set{(-c,a_{ij}-1)}{c\in K^0\ \te{with}\ p_c^1\ge 0}.\]
For convenience, we also write
\begin{align*}
  X_{ij,r}^\pm(M,w)=X_{ij,r}^\pm(M,\check M,w)\quad\te{and}\quad
  \:x_i^\pm(z)X_{ij,r}^\pm(M,w)\;=\:x_i^\pm(z)X_{ij,r}^\pm(M,\check M, w)\;.
\end{align*}
\begin{lem}\label{lem:X-add-on} If $s_i=2$, then  $\:x_i^\pm(\tau^\pm(a)w)X_{ij,r}^\pm(M,w)\;=0$ for  $a\in \bar M_1\cup \bar M_2$,
where
\begin{align*}
  \bar M_1&=\set{(-c,a_{ij}-1+2n)}{c\in K^0,\,1\le n\le p_c^0}\\
  &\cup\set{(c,a_{ij}-2)}{c\in K^0\,\te{with }p_c^1\ge 0},\\
  \bar M_2&=\set{(c,a_{ij}+2n)}{c\in K^0,\,1\le n\le p_c^1-1}\\
  &\cup\set{(-c,a_{ij}+2p_c^0+1)}{c\in K^0\,\te{with }p_c^0=p_c^1}.
\end{align*}
%
%
\end{lem}

\begin{proof}
Assume first that $a\in \bar M_1$. Then it is easy to see that there exist $b_1,b_2\in \check{M}$ such that $(a,b_1,b_2)$ is an $A_2^{(2)}$-triple.
 This together with \eqref{a22triple} gives
$\:x_i^\pm(\tau^\pm(a)w)\wt X_{ij,r}^\pm(M,\check M, w,\underline w)\;=0$ and hence $\:x_i^\pm(\tau^\pm(a)w)X_{ij,r}^\pm(M,w)\;=0$.
For the case that $a\in \bar M_2$,
we need to introduce another tuple
 $\hat{\bm{t}}=(t_c\mid c\in K_0)$, which is defined by $t_c= p_c^1$ if $p_c^1\ge 0$ and $t_c=0$ if $p_c^1=-1$.
Then it is obvious that there exist $b_1,b_2\in M_{\hat{\bm{t}}}$ such that $(a,b_1,b_2)$ is an $A_2^{(2)}$-triple.
This implies that
$\:x_i^\pm(\tau^\pm(a)w)\wt X_{ij,r}^\pm(M,M_{\hat{\bm{t}}}, w,\underline w)\;=0$.
In view of Lemma \ref{lem:defXijr-2}, we also have $\:x_i^\pm(\tau^\pm(a)w)\wt X_{ij,r}^\pm(M,\check M, w,\underline w)\;=0$ and hence $\:x_i^\pm(\tau^\pm(a)w)X_{ij,r}^\pm(M,w)\;=0$,
as required.
\end{proof}

As in Definition \ref{de:quan-comm}, we introduce the following $g$-commutator:
\begin{align*}
&[x_i^\pm(z),X_{ij,r}^\pm(M,w)]_g\\
=&x_i^\pm(z)X_{ij,r}^\pm(M,w)-g_{ji}(z/w)^{\mp 1}\prod_{a\in M}g_{ii}(z/\tau^\pm(a)w)^{\mp 1}X_{ij,r}^\pm(M,w)x_i^\pm(z),
\end{align*}
and for $a\in\partial M$, we define
\begin{align*}
&\al^\pm(a,M,w)=\begin{cases}
    \dis\lim_{z\to\tau^\pm(a)w}\frac{z^{d_i}-\tau^\pm(a)^{d_i}w^{d_i}}{f_{ij}^\pm(z,w)\prod_{a\in M}f_{ii}^\pm(z,\tau^\pm(a)w)},
& \mbox{if } a\in\partial M\setminus \partial^* M \\
    \dis\lim_{z\to\tau^\pm(a)w}\frac{(z^{d_i}-\tau^\pm(a)^{d_i}w^{d_i})^2}{f_{ij}^\pm(z,w)\prod_{a\in M}f_{ii}^\pm(z,\tau^\pm(a)w)}, &\mbox{if } a\in\partial^* M.
                     \end{cases}
\end{align*}

\begin{prop}\label{prop:g-commutator-cal-next}
Let $M\in\mathscr M_r$ with $r\in \N$. Then for any $a\in \partial M$, $\al^\pm(a,M,w)$ lies in $\C((w))[[\hbar]]$ and is nonzero.
Furthermore,
 we have
\begin{align*}
&[x_i^\pm(z),X_{ij,r}^\pm(M,w)]_g=\sum_{a\in\partial M}\al^\pm(a,M,w)X_{ij,r+1}^\pm(M\cup \{a\},w)z^{-d_i}\delta\(\frac{\tau^\pm(a)^{d_i}w^{d_i}}{z^{d_i}}\).
\end{align*}
\end{prop}
\begin{proof}
Set $D(z,w)=f_{ij}^\pm(z,w)\prod_{a\in M}f_{ii}^\pm(z,\tau^\pm(a)w)$. Then it follows from \eqref{gijex}
that
\begin{align}\label{eq:xixijrg}
[x_i^\pm(z),X_{ij,r}^\pm(M,w)]_g=\:x_i^\pm(z)X_{ij,r}^\pm(M,w)\;(\iota_{z,w}-\iota_{w,z})(D(z,w)^{-1}).
\end{align}
Assume first that $s_i=1$. Note that in this case
\begin{align}\label{eq:dzwsi1}
 D(z,w)
  =\prod_{c\in K^0}(z^{d_i}-c q_i^{\pm a_{ij}d_i\pm 2p_c^0d_i\pm 2d_i}w^{d_i})
\end{align}
and so for $a=(c,a_{ij}+2p_c^0+2)\in \partial M$ we have
\begin{align}\label{eq:alamw1}
\al^\pm(a,M,w)
=\prod_{c\ne c'\in K^0}( cq_i^{\pm a_{ij}d_i\pm 2p_c^0d_i\pm 2d_i}w^{d_i} -c' q_i^{\pm a_{ij}d_i\pm 2p_{c'}^0d_i\pm 2d_i}w^{d_i})\inverse,
\end{align}
which is nonzero.
Furthermore, by \eqref{deltadec}, \eqref{eq:xixijrg}, \eqref{eq:dzwsi1} and \eqref{eq:alamw1} we have
\begin{align*}
&[x_i^\pm(z),X_{ij,r}^\pm(M,w)]_g\\
=&\sum_{c\in K^0}
\prod_{c\ne c'\in K^0}( cq_i^{\pm a_{ij}d_i\pm 2p_c^0d_i\pm 2d_i}w^{d_i} -c' q_i^{\pm a_{ij}d_i\pm 2p_{c'}^0d_i\pm 2d_i}w^{d_i})\inverse\\
&\times
    \:x_i^\pm(c^{\frac{1}{d_i}}q_i^{\pm a_{ij}\pm 2p_c^0\pm 2}w  )X_{ij,r}^\pm(M,w)\;z^{-d_i}\delta\(\frac{ cq_i^{\pm a_{ij}d_i\pm 2p_c^0d_i\pm 2d_i}w^{d_i} }{ z^{d_i} }\)\\
=&\sum_{a\in\partial M}\al^\pm(a,M,w)X_{ij,r+1}^\pm(M\cup \{a\},w)z^{-d_i}\delta\(\frac{\tau^\pm(a)^{d_i}w^{d_i}}{z^{d_i}}\).
\end{align*}

Now we turn to the case of $s_i=2$, then we have
\begin{align*}
&D(z,w)=\prod_{c\in K^0}(z^{d_i}-cq_i^{\pm a_{ij}d_{i}\pm 2p_c^0d_i\pm 2d_i}w^{d_i})
\prod_{c\in K^0}(z^{d_i}+cq_i^{\pm a_{ij}d_{i}\pm 2p_c^1d_i\pm d_i}w^{d_i})
\\
&\times
\prod_{c\in K^0}\prod_{n=1}^{p_c^0}(z^{d_i}+cq_i^{\pm a_{ij}d_{i}\pm 2nd_i\mp d_i}w^{d_i})
\prod_{c\in K^0}\prod_{n=0}^{p_c^1}(z^{d_i}-cq_i^{\pm a_{ij}d_{i}\pm 2nd_i\mp 2d_i}w^{d_i})\\
&=\prod_{a\in \partial M\setminus \partial^\ast M}(z^{d_i}-\tau^\pm(a)^{d_i}w_a^{d_i})
\prod_{a\in  \partial^\ast M}(z^{d_i}-\tau^\pm(a)^{d_i}w_a^{d_i})^2\\
&\quad\prod_{a\in \bar{M}_1\setminus \partial^\ast M}(z^{d_i}-\tau^\pm(a)^{d_i}w_a^{d_i})
\prod_{a\in \bar{M}_2}(z^{d_i}-\tau^\pm(a)^{d_i}w_a^{d_i}).
\end{align*}
This shows that $\al^\pm(a,M,w)$ is nonzero.
Furthermore, one can conclude from  \eqref{deltadec}, \eqref{eq:xixijrg} and  Lemma \ref{lem:X-add-on} that

\begin{align*}
&[x_i^\pm(z),X_{ij,r}^\pm(M,w)]_g\\
=&\sum_{a\in\partial M\setminus\partial^\ast M}
\frac{z^{d_i}-\tau^\pm(a)^{d_i}w^{d_i}}{D(z,w)}\:x_i^\pm(z)X_{ij,r}^\pm(M,w)\;
    z^{-d_i}\delta\(\frac{\tau^\pm(a)^{d_i}w^{d_i}}{z^{d_i}}\)\\
&+\sum_{a\in \partial^\ast M}
\frac{\(z^{d_i}-\tau^\pm(a)^{d_i}w^{d_i}\)^2}{D(z,w)}\:x_i^\pm(z)X_{ij,r}^\pm(M,w)\;\frac{\partial}{\partial \tau^\pm(a)^{d_i}w^{d_i}}z^{-d_i}
\delta\(\frac{\tau^\pm(a)^{d_i}w^{d_i}}{z^{d_i}}\).
\end{align*}
Note that for $a\in \partial M\setminus \partial^*M$, we have
\begin{align*}\lim_{z\to\tau^\pm(a)w}\frac{z^{d_i}-\tau^\pm(a)^{d_i}w^{d_i}}{D(z,w)}\:x_i^\pm(z)X_{ij,r}^\pm(M,w)\;=
\al^\pm(a,M,w)X_{ij,r+1}^\pm(M\cup\{a\},w).
\end{align*}
On the other hand, for $a\in \partial^* M$, set $E_a(z^{d_i},w^{d_i})=(z^{d_i}-\tau^\pm(a)^{d_i}w^{d_i})^2\cdot D(z,w)^{-1}$.
Recall from Lemma \ref{lem:X-add-on} that $\:x_i^\pm(\tau^\pm(a)w)X_{ij,r}^\pm(M,w)\;=0$ and so $z^{d_i}-\tau^\pm(a)^{d_i}w^{d_i}$ divides
$\:x_i^\pm(z)X_{ij,r}^\pm(M,w)\;=0$. Furthermore, by definition we have
\begin{align*}
X_{ij,r+1}^\pm(M\cup\{a\},w)=\lim_{z\mapsto \tau^\pm(a)w}\frac{1}{z^{d_i}-\tau^\pm(a)^{d_i}w^{d_i}}\:x_i^\pm(z)X_{ij,r}^\pm(M,w)\;.
\end{align*}Now  we have
\begin{align*}
  &E_a(z^{d_i},w^{d_i})\:x_i^\pm(z)X_{ij,r}^\pm(M,w)\;\frac{\partial}{\partial \tau^\pm(a)^{d_i}w^{d_i}}z^{-d_i}\delta\(\frac{\tau^\pm(a)^{d_i}w^{d_i}}{z^{d_i}}\)\\
  =
  &-E_a(z^{d_i},w^{d_i})\:x_i^\pm(z)X_{ij,r}^\pm(M,w)\;\frac{\partial}{\partial z^{d_i}}z^{-d_i}\delta\(\frac{\tau^\pm(a)^{d_i}w^{d_i}}{z^{d_i}}\)\\
  =&-\frac{\partial}{\partial z^{d_i}}\(E_a(z^{d_i},w^{d_i})\:x_i^\pm(z)X_{ij,r}^\pm(M,w)\;
    z^{-d_i}\delta\(\frac{\tau^\pm(a)^{d_i}w^{d_i}}{z^{d_i}}\)\)\\
  &+\lim_{z\to\tau^\pm(a)w}\(\frac{\partial}{\partial z^{d_i}}\(E_a(z^{d_i},w^{d_i})
    \:x_i^\pm(z)X_{ij,r}^\pm(M,w)\;\)\)
    z^{-d_i}\delta\(\frac{\tau^\pm(a)^{d_i}w^{d_i}}{z^{d_i}}\)\\
    =&\lim_{z\to\tau^\pm(a)w}\(E_a(z^{d_i},w^{d_i})\frac{1}{z^{d_i}-\tau^\pm(a)^{d_i}w^{d_i}}
    \:x_i^\pm(z)X_{ij,r}^\pm(M,w)\;\)
    z^{-d_i}\delta\(\frac{\tau^\pm(a)^{d_i}w^{d_i}}{z^{d_i}}\)\\
  =&\al^\pm(a,M,w)X_{ij,r+1}^\pm(M\cup\{a\},w)
    z^{-d_i}\delta\(\frac{\tau^\pm(a)^{d_i}w^{d_i}}{z^{d_i}}\).
\end{align*}
This completes the proof of the proposition.
\end{proof}

Let $r$ be a nonnegative  integer. Set
\begin{align*}
\bar{r}=\emptyset\ \, \te{if}\ \, r=0\quad\te{and}\quad \bar{r}=\{1,2,\dots,r\}\ \, \te{if}\ \, r>0.
\end{align*}
For $M\in \mathscr M_r$, a bijective map $\theta:\bar r\rightarrow M$ is called a {\em flag} on $M$ if $r=0$ or $r>0$ and
 for each $1\le k\le r$, $\theta(\bar r\setminus\bar k)\in \mathscr M_{r-k}$
and $\theta(k)\in \partial \theta(\bar r\setminus\bar k)$.
Denote by $\mathcal F_{M}$ the set of all flags on $M$.
For $\theta\in \mathcal F_M$, set $\al_{ij,r}^\pm(M,\theta,w)=1$ if $r=0$ and
\begin{align*}
\al_{ij,r}^\pm(M,\theta,w)=\prod_{k=1}^{r}\al^\pm(\theta(k),\theta(\bar r\setminus\bar k),w)\quad\te{if}\quad r>0.
\end{align*}

For $r\in \N$, $M\in\mathscr M_r$ and $\theta\in \mathcal F_{M}$,   set $\delta^\pm(M,\theta,\underline z,w)=1$ if $r=0$ and set
\begin{align*}
&\delta^\pm(M,\theta,\underline z,w)=\prod_{a=1}^{r}z_a^{-d_i}\delta
    \(\frac{\bar\theta^\pm(a)^{d_i}w^{d_i}}{z_a^{d_i}}\)\quad\te{if}\quad r>0,
\end{align*}
where $\bar\theta^\pm$ denotes the composition map $\tau^\pm\circ\theta$. Note that for $\sigma\in S_r$, we have
\begin{align}\label{lem:delta-theta-z-invariance}
  \delta^\pm(M,\theta,\sigma \underline z,w)=\delta^\pm(M,\theta\circ\sigma\inv,\underline z,w),
\end{align}
where $(\sigma \underline z)(a)=z_{\sigma(a)}$ for $a=1,\dots,r$.

By using Proposition \ref{prop:g-commutator-cal-next} and an inductive argument, one can obtain the following main result of this subsection.

\begin{prop}\label{prop:g-commutator-cal-final}
For a nonnegative integer $r$, one has
\begin{align*}
&[x_i^\pm(z_{1}),x_i^\pm(z_{2}),\dots,x_i^\pm(z_{r}),x_j^\pm(w)]_g\\
=&\sum_{M\in\mathscr M_r}\sum_{\theta\in\mathcal F_M}\alpha_{ij,r}^\pm(M,\theta,w)
  X_{ij,r}^\pm(M,w)\delta^\pm(M,\theta,\underline z,w).
\end{align*}
\end{prop}

\subsection{Relations among $\al_{ij,r}^\pm(M,\theta,w)$}
Throughout this subsection, we fix  an $M\in\mathscr M_r$ with $r\in \Z_+$.
Here the main goal
is to  establish  some relations among $\al_{ij,r}^\pm(M,\theta,w)$ with $\theta\in\mathcal F_M$.

We first  define a partial order  $``\prec "$ on $M$ as follows: $a\prec b$ if and only if there exist a sequence $a=a_1,a_2,\dots,a_k=b$ in $M$ such that
$a_p\prec'a_{p+1}$ for $p=0,\dots,k-1$, where for
 $(c_1,n_1),(c_2,n_2)\in M$, $(c_1,n_1)\prec'(c_2,n_2)$ means that  either
$c_1=c_2$ and $n_1=n_2+2$ or $c_2\in K^0$,  $c_1=-c_2$ and $n_1=n_2-1$.
It is straightforward to see that
 a bijective map $\theta:\bar{r}\rightarrow M$ is a flag
if and only if
\begin{align}\label{lem:flags}
  &\theta(a)\prec \theta(b)\quad\te{implies}\quad a<b\quad\te{for }a,b\in\bar r.
\end{align}
Furthermore, let $(c_1,n_2),(c_2,n_2)\in M$. Then $F_{ii}^\pm(\tau^\pm((c_1,n_1))w,\tau^\pm((c_2,n_2))w)=0$ if and only if
\begin{align}\label{fii=0}
\te{either}\ (c_1,n_1)\prec'(c_2,n_2)\quad\te{or}\quad c_2\in K^1,\ c_1=-c_2\ \te{and}\ n_1=n_2-1.\end{align}
Note  that in the latter case one has
\begin{align}\label{fii=0other}
(c_2,n_2)\prec'(-c_2,n_2+1),\ (c_2,n_2-2)\prec'(-c_2,n_2-1)=(c_1,n_1).
\end{align}

For $\theta\in\mathcal F_M$, $\sigma\in S_r$ and $k=1,\dots,r$, set
\begin{align*}
F_{\theta,\sigma,k}^\pm(w)&= \prod_{\substack{k\le a<b\le r,\ \sigma\inv(a)>\sigma\inv(b)}}
  F_{ii}^\pm(\bar\theta^\pm(a)w,\bar\theta^\pm(b)w),\\
 G_{\theta,\sigma,k}^\pm(w)&= \prod_{\substack{k\le a<b\le r,\ \sigma\inv(a)>\sigma\inv(b)}}
  G_{ii}^\pm(\bar\theta^\pm(a)w,\bar\theta^\pm(b)w).
\end{align*}

\begin{lem}\label{lem:theta-sigma-Fii-nonzero}
For $\theta\in\mathcal F_M$ and $\sigma\in S_r$, $\theta\circ\sigma\in\mathcal F_M$ if and only if $F_{\theta,\sigma,1}^\pm(w)\ne 0$.
\end{lem}
\begin{proof} In view of \eqref{lem:flags}, it suffices to prove that $F_{\theta,\sigma,1}^\pm(w)=0$ if and only if
 there exist $1\le a<b\le r$ such that $\sigma\inv(a)>\sigma\inv(b)$ and $\theta(a)\prec\theta(b)$.
For  $1\le a<b\le r$, it follows from \eqref{fii=0}, \eqref{fii=0other} and \eqref{lem:flags} that $F_{ii}^\pm(\bar\theta^\pm(a)w,\bar\theta^\pm(b)w)=0$ if and only if $\theta(a)\prec'\theta(b)$.
This in particular implies that, if $F_{\theta,\sigma,1}^\pm(w)=0$, then
there exist $a<b$ such that $\sigma\inv(a)>\sigma\inv(b)$ and $\theta(a)\prec' \theta(b)$.
Conversely, let  $a<b$ such that $\sigma\inv(a)>\sigma\inv(b)$ and $\theta(a)\prec\theta(b)$.
Then there exists a sequence
$a=a_1<a_2<\cdots<a_k=b$ such that $\theta(a_p)\prec'\theta(a_{p+1})$ for $p=1,\dots,k-1$.
Take a $p$ such that $\sigma\inv(a_p)>\sigma\inv(a_{p+1})$. Then we have
$F_{ii}^\pm(\bar\theta^\pm(a_p)w,\bar\theta^\pm(a_{p+1})w)=0$ and hence $F_{\theta,\sigma,1}^\pm(w)=0$.
\end{proof}

For each $f(\underline z)\in\C(z_1,z_2,\dots,z_r)[[\hbar]]$
and each non-zero elements $c_1,c_2,\dots,c_r\in\C((w))[[\hbar]]$,
we denote by
\begin{align*}
&\varprodright_{l=1}^r\lim_{z_l\to c_l w}f(\underline z)=
\lim_{z_1\to c_1w}
    \(\lim_{z_2\to c_2w}
        \(\cdots \(\lim_{z_r\to c_rw}f(\underline z)\)\cdots\)
    \).
\end{align*}

\begin{lem}\label{lem:theta-sigma-Fii-Gii}
Let $\theta\in \mathcal F_M$, $\sigma\in S_r$ and $k=1,\dots,r$. Then the limit
\begin{align}\label{limfg}
  L_k^\pm(\underline z,w)=\varprodright_{l=k}^r\lim_{z_l\to\bar\theta^\pm(l)w}\prod_{\substack{1\le a<b\le r\\ \sigma\inv(a)>\sigma\inv(b)}}\frac{F_{ii}^\pm(z_a,z_b)}{G_{ii}^\pm(z_a,z_b)}
\end{align}
exists in $\C((z_1))((z_2))\cdots((z_{k-1}))((w))[[\hbar]]$. Furthermore, the limit $L_k^\pm(\underline z,w)=0$ if and only if
 $F_{\theta,\sigma,k}^\pm(w)=0$.
\end{lem}
\begin{proof}We will
 prove the assertion by induction on $k$. The case $k=r$ is trivial and we assume that the assertion hold for all $s>k$.
Set
 \begin{align*}
 \bar{F}_k(w)&=\prod_{\substack{k<a\le r,\ \sigma\inv(k)>\sigma\inv(a)}}F_{ii}^\pm(\bar\theta^\pm(k)w,\bar\theta^\pm(a)w),\\
 \bar{G}_k(w)&=\prod_{\substack{k<a\le r,\ \sigma\inv(k)>\sigma\inv(a)}}G_{ii}^\pm(\bar\theta^\pm(k)w,\bar\theta^\pm(a)w).
 \end{align*}
 Note that if $L^\pm_{k+1}(\underline z,w)=0$, then $L^\pm_k(\underline z,w)=0$ and
$F_{\sigma,\theta,k}^\pm(w)=\bar{F}_k^\pm(w) F_{\sigma,\theta,k+1}^\pm(w)=0$.
Thus it suffices to consider the case that $L^\pm_{k+1}(\underline z,w)\ne 0$ (and so  $F_{\sigma,\theta,k+1}^\pm(w)\ne 0)$).
Assume first that  $\bar{G}_k(w)\ne 0$.
Then we have
\[L_k^\pm(\underline z,w)=0 \Leftrightarrow \frac{\bar{F}_k^\pm(w)}{\bar{G}_k^\pm(w)}=0\Leftrightarrow
\bar{F}_k^\pm(w)=0\Leftrightarrow F_{\sigma,\theta,k}^\pm(w)=0.\]

Assume next that $\bar{G}_k(w)=0$. Then there exists $b\in \bar{r}$ such that
$k<b$, $\sigma^{-1}(k)>\sigma^{-1}(b)$ and $G_{ii}^\pm(\bar\theta^\pm(k)w,\bar\theta^\pm(b)w)=0$.
Recall that $G^\pm_{ii}(z,w)=-F_{ii}^\pm(w,z)$.
Then it follows from \eqref{fii=0}  that either $\theta(b)\prec \theta(k)$ or
$\theta(k)=(c,n),\theta(b)=(-c,n-1)$ with $c\in K^1$.
Since $k<b$, it follows from \eqref{lem:flags} that the former case is impossible.
Recall from  \eqref{fii=0other} that $(c,n-2), (-c,n+1)\in M$ with $\theta(k)\prec'(c,n-2), (-c,n+1)\prec' \theta(b)$.
Take $b_1,b_2\in \bar{r}$ such that $\theta(b_1)=(c,n-2)$ and $\theta(b_2)=(-c,n+1)$.
In view of \eqref{lem:flags} and \eqref{fii=0}, we obtain $k<b_l<b$ and
$F_{ii}^\pm(\bar\theta^\pm(k)w,\bar\theta^\pm(b_l)w)=0=F_{ii}^\pm(\bar\theta^\pm(b_l)w,\bar\theta^\pm(b)w)$ for $l=1,2$.
From the induction assumption $F_{\theta,\sigma,k+1}(w)\ne 0$, it follows that  $\sigma\inv(b_l)<\sigma\inv(b)$ and hence $\sigma\inv(k)>\sigma\inv(b_l)$.
Thus we find that $\bar{F}_k^\pm(w)=0$ and hence $F_{\theta,\sigma,k}^\pm(w)=0$.
Finally, notice that
\begin{align*}
  &F_{ii}^\pm(z_k,\bar\theta^\pm(b_1)w)=(z_k^{d_i}+q^{\pm d_i}\bar\theta^\pm(b)^{d_i}w^{d_i})(z_k^{d_i}-q^{\mp 2 d_i}\bar\theta^\pm(b)^{d_i}w^{d_i}),\\
  &F_{ii}^\pm(z_k,\bar\theta^\pm(b_2)w)=(z_k^{d_i}-q^{\pm 4d_i}\bar\theta^\pm(b)^{d_i}w^{d_i})(z_k^{d_i}+q^{\pm d_i}\bar\theta^\pm(b)^{d_i}w^{d_i}),\\
  &G_{ii}^\pm(z_k,\bar\theta^\pm(b)w)=(q^{\mp d_i}z_k^{d_i}+\bar\theta^\pm(b)^{d_i}w^{d_i})(q^{\pm 2d_i}z_k^{d_i}-\bar\theta^\pm(b)^{d_i}w^{d_i}).
\end{align*}
Thus the limit $L_k^\pm(\underline z,w)$ exists and is zero, as required.
\end{proof}

For convenience, we set $\al_{ij,r}^\pm(M,\theta,w)=0$ when $\theta:\bar r\to M$ is a bijective map such that $\theta\not\in\mathcal F_M$.
The following result is the main result of this subsection.
\begin{prop}\label{prop:al-coeff-rel}
Let $r\in\Z_+$, $\sigma\in S_r$, $M\in\mathscr M_r$ and $\theta\in \mathcal F_M$.
Then we have
\begin{align*}
  \al_{ij,r}^\pm(M,\theta\circ\sigma,w)=\al_{ij,r}^\pm(M,\theta,w)
  \varprodright_{l=1}^r
  \lim_{z_l\to\bar\theta^\pm(l)w}\prod_{\substack{1\le a<b\le r\\ \sigma\inv(a)>\sigma\inv(b)}}
  \frac{F_{ii}^\pm(z_a,z_b)}{G_{ii}^\pm(z_a,z_b)}.
\end{align*}
\end{prop}

\begin{proof}
In view of Lemmas \ref{lem:theta-sigma-Fii-nonzero} and \ref{lem:theta-sigma-Fii-Gii}, it suffices to prove the assertion for the case that $\theta\circ\sigma\in\mathcal F_M$.
For each $1\le a<r$, write $\sigma_a=(a,a+1)\in S_r$.
From the definition of $g$-commutators (see \eqref{eq:def-g-commutator}) and the relations (Q8), it follows that
\begin{align*}
  &F_{ii}^\pm(z_a,z_{a+1})[x_i^\pm(z_1),x_i^\pm(z_2),\dots,x_i^\pm(z_r),x_j^\pm(w)]_g\\
  =&G_{ii}^\pm(z_a,z_{a+1})
  [x_i^\pm(z_{\sigma_a(1)}),x_i^\pm(z_{\sigma_a(2)}),
    \dots,x_i^\pm(z_{\sigma_a(r)}),x_j^\pm(w)]_g.
\end{align*}
By an induction argument, we also have
\begin{align*}
&\prod_{\substack{1\le a<b\le r\\ \sigma\inv(a)>\sigma\inv(b)}}F_{ii}^\pm(z_a,z_b)[x_i^\pm(z_1),x_i^\pm(z_2),\dots,x_i^\pm(z_r),x_j^\pm(w)]_g\\
&\quad=\prod_{\substack{1\le a<b\le r\\ \sigma\inv(a)>\sigma\inv(b)}}G_{ii}^\pm(z_a,z_b)[x_i^\pm(z_{\sigma(1)}),x_i^\pm(z_{\sigma(2)}),\dots,x_i^\pm(z_{\sigma(r)}),x_j^\pm(w)]_g.
\end{align*}
Combining this with Proposition \ref{prop:g-commutator-cal-final} and \eqref{lem:delta-theta-z-invariance}, we obtain:
\begin{align*}
  &\sum_{M\in\mathscr M_r}\sum_{\theta\in\mathcal F_M}
    F_{\theta,\sigma,1}^\pm(w)
    \al_{ij,r}^\pm(M,\theta,w)X_{ij,r}^\pm(M,w)\delta^\pm(M,\theta,\underline z,w)\\
= &\sum_{M\in\mathscr M_r}\sum_{\theta\in\mathcal F_M}
    G_{\theta,\sigma,1}^\pm(w)
    \al_{ij,r}^\pm(M,\theta,w)X_{ij,r}^\pm(M,w)\delta^\pm(M,\theta,\sigma\underline z,w)\\
=&\sum_{M\in\mathscr M_r}\sum_{\theta\in\mathcal F_M}
    G_{\theta,\sigma,1}^\pm(w)
    \al_{ij,r}^\pm(M,\theta,w)X_{ij,r}^\pm(M,w)\delta^\pm(M,\theta\circ\sigma\inv,\underline z,w)\\
=&\sum_{M\in\mathscr M_r}\sum_{\theta\circ\sigma\in\mathcal F_M}
   G_{\theta,\sigma,1}^\pm(w)
    \al_{ij,r}^\pm(M,\theta\circ\sigma,w)X_{ij,r}^\pm(M,w)\delta^\pm(M,\theta,\underline z,w).
\end{align*}
Now the proposition follows from  Lemma  \ref{lem:delta-function-independent} and Lemma \ref{lem:theta-sigma-Fii-Gii}.
\end{proof}

\subsection{Proof of Theorem \ref{prop:Dr-to-normal-ordering}}\label{subsec:qDr-to-normal-ordering}
In this subsection, we complete the proof of Theorem \ref{prop:Dr-to-normal-ordering}.
Throughout this subsection, let $W\in \R_\pm^s, (i,j)\in \mathbb I$, $m\in \Z_+$,
$f^\pm(z,w)\in \C[[\hbar]][z^{d_i},w^{d_i}]$ and $B=(B_0,\dots,B_m)\in (\C[[\hbar]])^{m+1}$ be as in Theorem  \ref{prop:Dr-to-normal-ordering}.

For $0\le r\le m$ and $M\in \mathscr M_r$, we  set
\begin{align*}
  &D_{r,M}^\pm=\sum_{\sigma\in S_m}\sum_{s= r}^m
  \sum_{\substack{J=\{j_1<\cdots<j_r\}\subset\bar s\\
    \{j_{r+1}<\cdots<j_s\}=\bar s\setminus J}}
  \sum_{\theta\in\mathcal F_M}
  (-1)^sB_{s}
  \prod_{1\le a<b\le m}f^\pm(z_{\sigma(a)},z_{\sigma(b)})\\
&\times
  \prod_{\substack{1\le a\le r<b\le s\\ j_a>j_b}}
    g_{ii}(z_{\sigma(j_b)}/z_{\sigma(j_b)})^{\mp 1}
  \prod_{r<b\le s}g_{ji}(z_{\sigma(j_b)}/w)^{\mp 1}
  \al_{ij,r}^\pm(M,\theta,w)X_{ij,r}^\pm(M,w)\\
&\times
  \delta^\pm(M,\theta,\sigma\underline z_J,w)
  x_i^\pm(z_{\sigma(j_{r+1})})\cdots x_i^\pm(z_{\sigma(j_s)})
  x_i^\pm(z_{\sigma(s+1)})\cdots x_i^\pm(z_{\sigma(m_{ij})}).
\end{align*}
Recall the current $D_{ij}^\pm(m,f^\pm,B)$ defined in Section \ref{subsec:aff-qSerre-rel}. We have:

\begin{lem}\label{lem:Dr-to-g-commutator-pre000} As operators on $W$, one has
\[D_{ij}^\pm(m,f^\pm,B)=\sum_{r=0}^{m}\sum_{M\in\mathscr M_r}D_{r,M}^\pm.\]
\end{lem}
\begin{proof} Take a subset $J=\{j_1<\cdots<j_s\}$ of $\bar{r}$ and set $\bar r\setminus J=\{j_1'<\cdots<j_t'\}$.
By definition, it is straightforward to check that
\begin{align*}
&x_i^\pm(z_0)[x_i^\pm(z_{j_1}),x_i^\pm(z_{j_2}),\cdots,x_i^\pm(z_{j_s}),x_j^\pm(w)]_g
    x_i^\pm(z_{j_1'})x_i^\pm(z_{j_2'})\cdots x_i^\pm(z_{j_t'})\\
=&[x_i^\pm(z_0),x_i^\pm(z_{j_1}),x_i^\pm(z_{j_2}),\cdots
    x_i^\pm(z_{j_s}),x_j^\pm(w)]_g
    x_i^\pm(z_{j_1'})x_i^\pm(z_{j_2'})\cdots x_i^\pm(z_{j_t'})\\
 &+\prod_{a\in J}g_{ii}(z_0/z_a)^{\mp 1}g_{ji}(z_0/w)^{\mp 1}
 [x_i^\pm(z_{j_1}),x_i^\pm(z_{j_2}),\cdots,x_i^\pm(z_{j_s}),x_j^\pm(w)]_g\\
 &\quad\quad\times
 x_i^\pm(z_0)x_i^\pm(z_{j_1'})x_i^\pm(z_{j_2'})\cdots x_i^\pm(z_{j_t'}).
\end{align*}
Using this and an induction argument on $r$, we have
\begin{align*}
  x_i^\pm&(z_1)x_i^\pm(z_2)\cdots x_i^\pm(z_r)x_j^\pm(w)
= \sum_{J\subset \bar{r}}\(\prod_{\substack {a>b,\ a\in J,\ b\not\in J}}g_{ii}(z_b/z_a)^{\mp 1}
    \prod_{b\not\in J}g_{ji}(z_b/w)^{\mp 1}\)\\
&\times
    [x_i^\pm(z_{j_1}),x_i^\pm(z_{j_2}),\cdots,x_i^\pm(z_{j_s}),x_j^\pm(w)]_g
    x_i^\pm(z_{j_1'})x_i^\pm(z_{j_2'})\cdots x_i^\pm(z_{j_t'}).
\end{align*}
Then the assertion follows from the above formula and Proposition
\ref{prop:g-commutator-cal-final}.
\end{proof}

For $r\in\N$ with $r\le m$, we set
\begin{align}
\mathscr N_{r}=\{M\in \mathscr M_r\mid p_c^0< -a_{ij}\ \te{for all}\ c\in K^0\}
\quad\te{and}\quad \mathscr N_{r}^c=\mathscr M_r\setminus \mathscr N_{r}.
\end{align}
\begin{lem}\label{lem:G-ij-mult-non-zero}
For each $r\in\N$, $M\in\mathscr N_{r}$ and $\theta\in \mathcal F_M$, the following holds
\begin{align*}
  \lim_{\substack{z_s\to\bar\theta^\pm(s)w\\ 1\le s\le r}}
   \prod_{1\le a\le r}G_{ij}^\pm(z_a,w)\ne 0.
\end{align*}
\end{lem}

\begin{proof}
Notice that
\begin{align*}
 \lim_{\substack{z_s\to\bar\theta^\pm(s)w\\ 1\le s\le r}}
   \prod_{1\le a\le r}G_{ij}^\pm(z_a,w)
  =\prod_{1\le a\le r}
  \(\bar\theta^\pm(a)^{d_{ij}}q_i^{\pm d_{ij}a_{ij}}w^{d_{ij}}-w^{d_{ij}}\).
\end{align*}
Assume that there exist some $a$ such that $\bar\theta^\pm(a)^{d_{ij}}q_i^{\pm d_{ij}a_{ij}}=1$. Then
$\theta(a)=(c,-a_{ij})$ for some $c\in K^0$.
From the definition of $\mathscr M_r$, one has
\begin{align*}
&(c,a_{ij}+2p)\in M\quad\te{for } 0\le p\le -a_{ij}.
\end{align*}
This implies $p_c^0\ge -a_{ij}$, a contradiction with $M\in\mathscr N_{r}$.
\end{proof}

Let $r\in \N$ with $r\le m$.
As usual, we define an $(m,r)$-shuffle to be an element $\sigma\in S_m$ such that
\begin{align}
  \sigma(a)<\sigma(b),\ \sigma(a')<\sigma(b')\quad \mbox{for}\ 1\le a<b\le r< a'<b'\le m.
\end{align}
Denote by $S_{m,r}$ the set of all $(m,r)$-shuffles.
And, for a subset $J$ of $\bar{m}=\{1,\dots,m\}$, denote by
\begin{align*}
S_J=\{\sigma\in S_m\mid \sigma(j)\in J\ \te{and}\ \sigma(j')=j'\ \te{for}\ j\in J, j'\notin J\}.
\end{align*}
Note that for $\sigma\in S_{m,r}$, we have
\begin{align}\label{propertysmr}
\sigma(a)\ge a\ \te{if}\ a\le r,\quad \sigma(b)\le b\ \te{if}\ b>r\quad \te{and}\quad \sigma(c)=c\ \te{if}\ c>\sigma(r).
\end{align}
In particular, we have
\begin{align}\label{lem:S(r,m-r)->sigma(r)-fixed}
  \sigma(\bar s)=\bar s\quad\te{and}\quad
  \sigma(\bar m\setminus\bar s)=\bar m\setminus\bar s
  \quad\te{for }\sigma\in S_{m,r}\ \te{and}\ s\ge\sigma(r).
\end{align}

It is straightforward to see that for $\sigma\in S_m$,
\begin{align}\label{lem:Dr-to-g-commutator-lem0}
&\prod_{1\le a<b\le m}
    \frac{G_{ii}^\pm(z_{\sigma(a)},z_{\sigma(b)})}{G_{ii}^\pm(z_a,z_b)}
=(-1)^{|\sigma|}\prod_{\substack{1\le a<b\le m,\ \sigma\inv(a)>\sigma\inv(b) }}
    \frac{F_{ii}^\pm(z_a,z_b)}{G_{ii}^\pm(z_a,z_b)}.
\end{align}
Using \eqref{propertysmr} and \eqref{lem:Dr-to-g-commutator-lem0}, one can easily check that for
$\sigma\in S_{m,r}$
and $s\ge\sigma(r)$,
\begin{align}
\label{lem:Dr-to-g-commutator-lem1}  &\prod_{\substack {1\le a\le r<b\le s,\ \sigma(a)>\sigma(b)}}
    \frac{G_{ii}^\pm(z_b,z_a)}{F_{ii}^\pm(z_b,z_a)}
=(-1)^{|\sigma|}\prod_{1\le a<b\le m}
    \frac{G_{ii}^\pm(z_{\sigma\inv(a)},z_{\sigma\inv(b)})}
        {G_{ii}^\pm(z_a,z_b)},\\
&\prod_{r<a\le s}\frac{G_{ij}^\pm(z_a,w)}{F_{ij}^\pm(z_a,w)}
= \prod_{1\le a\le r}\frac{1}{G_{ij}^\pm(z_a,w)}
  \prod_{r<b\le m}\frac{1}{F_{ij}^\pm(z_b,w)}\nonumber\\
\label{lem:Dr-to-g-commutator-lem2}&\times
  \prod_{1\le a\le s}G_{ij}^\pm(z_{\sigma\inv(a)},w)
  \prod_{s<b\le m}F_{ij}^\pm(z_{\sigma\inv(a)},w).
\end{align}

For a subset $H$ of $S_m$, we define the polynomial
\begin{align*}
  P^\pm(H,\underline z,w)
  =\sum_{\sigma\in H}(-1)^{|\sigma|}Q^\pm(\underline z\sigma\inv,w),
\end{align*}
where
\begin{align*}
  Q^\pm(\underline z,w)
  =\sum_{r=0}^m (-1)^r B_{r}
  \prod_{1\le a<b\le m}f^\pm(z_a,z_b)G_{ii}^\pm(z_a,z_b)
  \prod_{a=0}^r G_{ij}^\pm(z_a,w)
  \prod_{b=r+1}^m F_{ij}^\pm(z_b,w).
\end{align*}
Note that when $H=S_m$, the polynomials $P^\pm(S_m,\underline z,w)$ coincide with $P^\pm_{ij}(m,f^\pm,B)$ (see Section \ref{subsec:aff-qSerre-rel}). It is straightforward to see that
\begin{align}\label{eq:P-sigma-invariance}
  P^\pm(H,\sigma\underline z,w)=(-1)^{|\sigma|}P^\pm(H\sigma\inv,\underline z,w)\quad
  \te{for }\sigma\in S_m.
\end{align}

We have:

\begin{lem}\label{lem:Dr-to-g-commutator}
Let $r\in \N$ with $r\le m$ and $M\in\mathscr N_r$.
Then
\begin{align*}
D_{r,M}^\pm=&\sum_{\theta\in \mathcal F_M}\sum_{\sigma\in S_{m}}
C_{r}^\pm(\sigma\underline z,w)
P^\pm(S_{m,r},
    \sigma\underline z,w)
\al_{ij,r}^\pm(M,\theta,w)X_{ij,r}^\pm(w)\\
&\times
\delta^\pm(M,\theta,\sigma\underline z,w)
x_i^\pm(z_{\sigma(r+1)})x_i^\pm(z_{\sigma(r+2)})\cdots x_i^\pm(z_{\sigma(m)}),
\end{align*}
where
\begin{align*}
  C_{r}^\pm&(\underline z,w)=
  \prod_{b=r+1}^{m}\frac{1}{F_{ij}^\pm(z_b,w)}
  \prod_{1\le b\le r}\frac{1}{G_{ij}^\pm(z_b,w)}
  \prod_{1\le a<b\le {m}}
    \frac{ 1 }
    {G_{ii}^\pm(z_a,z_b)}.
\end{align*}
\end{lem}

\begin{proof}
We denote by  $T_{r,M}^\pm(\underline z,w)$ the currents:
\begin{align*}
 &\sum_{s= r}^{m}
  \sum_{\substack{J=\{j_1<\cdots<j_r\}\subset\bar s\\
    \{j_{r+1}<\cdots<j_s\}=\bar s\setminus J}}
  \sum_{\theta\in\mathcal F_M}
  (-1)^s B_{s}
  \prod_{1\le a<b\le m}f^\pm(z_{a},z_{b})\\
&\times
  \prod_{\substack{1\le a\le r<b\le s\\ j_a>j_b}}
    g_{ii}(z_{j_b}/z_{j_a})^{\mp 1}
  \prod_{r<b\le s}g_{ji}(z_{j_b}/w)^{\mp 1}
  \al_{ij,r}^\pm(M,\theta,w)X_{ij,r}^\pm(M,w)\\
&\times
  \delta^\pm(M,\theta,\underline z_J,w)
  x_i^\pm(z_{j_{r+1}})\cdots x_i^\pm(z_{j_s})
  x_i^\pm(z_{s+1})\cdots x_i^\pm(z_{m})\\
=&\sum_{s= r}^{m}
  \sum_{\sigma\in S_{s,r}}
  \sum_{\theta\in\mathcal F_M}
  (-1)^s B_{s}
  \prod_{1\le a<b\le m}f^\pm(z_{a},z_{b})\\
&\times
  \prod_{\substack{1\le a\le r<b\le s\\ \sigma(a)>\sigma(b)}}
    g_{ii}(z_{\sigma(b)}/z_{\sigma(a)})^{\mp 1}
  \prod_{r<b\le s}g_{ji}(z_{\sigma(b)}/w)^{\mp 1}
  \al_{ij,r}^\pm(M,\theta,w)X_{ij,r}^\pm(M,w)\\
&\times
  \delta^\pm(M,\theta,\underline z_{\sigma(\bar r)},w)
  x_i^\pm(z_{\sigma({r+1})})\cdots x_i^\pm(z_{\sigma(s)})
  x_i^\pm(z_{s+1})\cdots x_i^\pm(z_{m}).
\end{align*}
By using  \eqref{lem:S(r,m-r)->sigma(r)-fixed},  $T_{r,M}^\pm(\underline z,w)$ can be rewritten as:
\begin{align*}
&\sum_{\sigma\in S_{m,r}}\sum_{s=\sigma(r)}^{m}
\sum_{\theta\in\mathcal F_M}
  (-1)^s B_{s}
  \prod_{1\le a<b\le m}f^\pm(z_{a},z_{b})\\
&\times
  \prod_{\substack{1\le a\le r<b\le s\\ \sigma(a)>\sigma(b)}}
    g_{ii}(z_{\sigma(b)}/z_{\sigma(a)})^{\mp 1}
  \prod_{r<b\le s}g_{ji}(z_{\sigma(b)}/w)^{\mp 1}
  \al_{ij,r}^\pm(M,\theta,w)X_{ij,r}^\pm(M,w)\\
&\times
  \delta^\pm(M,\theta,\underline z_{\sigma(\bar r)},w)
  x_i^\pm(z_{\sigma({r+1})})\cdots x_i^\pm(z_{\sigma(m)}).
\end{align*}
This together with the fact
$g_{ji}(z/w)^{\mp 1}
  ={G_{ij}^\pm(z,w)}/{F_{ij}^\pm(z,w)}$ gives
\begin{align*}
T_{r,M}^\pm(\underline z,w)=&\sum_{\sigma\in S_{m,r}}\sum_{s=\sigma(r)}^{m}
\sum_{\theta\in\mathcal F_M}
  (-1)^s B_{s}
  \prod_{1\le a<b\le m}f^\pm(z_{a},z_{b})\\
&\times
  \prod_{\substack{1\le a\le r<b\le s\\ \sigma(a)>\sigma(b)}}
    \frac{G_{ii}^\pm(z_{\sigma(b)},z_{\sigma(a)})}
        {F_{ii}^\pm(z_{\sigma(b)},z_{\sigma(a)})}
  \prod_{r<b\le s}
    \frac{G_{ij}^\pm(z_{\sigma(b)},w)}
        {F_{ij}^\pm(z_{\sigma(b)},w)}
  \al_{ij,r}^\pm(M,\theta,w)X_{ij,r}^\pm(M,w)\\
&\times
  \delta^\pm(M,\theta,\underline z_{\sigma(\bar r)},w)
  x_i^\pm(z_{\sigma({r+1})})\cdots x_i^\pm(z_{\sigma(m)}).
\end{align*}
Therefore, it follows from Lemma \ref{lem:G-ij-mult-non-zero}, \eqref{lem:Dr-to-g-commutator-lem1}
and \eqref{lem:Dr-to-g-commutator-lem2} that
\begin{align*}
  D_{r,M}^\pm=&\sum_{\sigma_1\in S_{m}}
    T_{r,M}^\pm(\sigma_1\underline z,w)
  =\sum_{\sigma_1\in S_{m}}\sum_{\sigma\in S_{m,r}}\sum_{s=\sigma(r)}^{m}
\sum_{\theta\in\mathcal F_M}
  (-1)^s B_{s} \\
&\times
  \prod_{1\le a<b\le m}
  f^\pm(\sigma_1(z_{\sigma^{-1}(a)}),\sigma_1(z_{\sigma^{-1}(b)}))
  \prod_{\substack{1\le a\le r<b\le s\\ \sigma(a)>\sigma(b)}}
    \frac{G_{ii}^\pm(\sigma_1(z_{b}),\sigma_1(z_{a}))}
        {F_{ii}^\pm(\sigma_1(z_{b}),\sigma_1(z_{a}))}
    \\
&\times
  \prod_{r<b\le s}
    \frac{G_{ij}^\pm(\sigma_1(z_{b}),w)}
        {F_{ij}^\pm(\sigma_1(z_{b}),w)}
  \al_{ij,r}^\pm(M,\theta,w)X_{ij,r}^\pm(M,w)
  \delta^\pm(M,\theta,\sigma_1\underline z_{\bar r},w)\\
&\times
  x_i^\pm(\sigma_1(z_{r+1}))\cdots x_i^\pm(\sigma_1(z_{m}))\\
=&\sum_{\sigma_1\in S_{m}}\sum_{\sigma\in S_{m,r}}\sum_{s=\sigma(r)}^{m}
\sum_{\theta\in\mathcal F_M}
  (-1)^s B_{s}(-1)^{|\sigma|}\\
&\times
  \prod_{1\le a<b\le m}
  f^\pm(\sigma_1(z_{\sigma^{-1}(a)}),\sigma_1(z_{\sigma^{-1}(b)}))
  \prod_{1\le a<b\le m}
    \frac{G_{ii}^\pm(\sigma_1(z_{\sigma^{-1}(a)}),\sigma_1(z_{\sigma^{-1}(b)}))}
        {G_{ii}^\pm(\sigma_1(z_{a}),\sigma_1(z_{b}))}
    \\
&\times
  \prod_{1\le a\le r}\frac{1}{G_{ij}^\pm(\sigma_1(z_a),w)}
  \prod_{r<a\le m}\frac{1}{F_{ij}^\pm(\sigma_1(z_a),w)}
  \prod_{1\le a\le s}G_{ij}^\pm(\sigma_1(z_{\sigma^{-1}(a)}),w)\\
&\times
  \prod_{s<b\le m}F_{ij}^\pm(\sigma_1(z_{\sigma^{-1}(b)}),w)
  \al_{ij,r}^\pm(M,\theta,w)X_{ij,r}^\pm(M,w)
  \delta^\pm(M,\theta,\sigma_1\underline z_{\bar r},w)\\
&\times
  x_i^\pm(\sigma_1(z_{r+1}))\cdots x_i^\pm(\sigma_1(z_{m})).
\end{align*}
Notice that
\begin{align*}
  F_{ij}^\pm(z_{r},w)
  \delta^\pm(M,\theta,\underline z,w)=0\quad\te{for  }
  M\in\mathscr M_r,\,
  \theta\in\mathcal F_M,
\end{align*}
and that
\[\te{$F_{ij}^\pm(z_r,w)$ divides $\displaystyle\prod_{s<b\le m}F_{ij}^\pm(z_{\sigma\inv(b)},w)$
for $s<\sigma(r)$}.\]
So we obtain
\begin{align*}
D_{r,M}^\pm=&\sum_{\sigma_1\in S_{m}}\sum_{\sigma\in S_{m,r}}\sum_{s=0}^{m}
\sum_{\theta\in\mathcal F_M}
  (-1)^s B_{s} (-1)^{|\sigma|}\\
&\times
  \prod_{1\le a<b\le m}
f^\pm(\sigma_1(z_{\sigma^{-1}(a)}),\sigma_1(z_{\sigma^{-1}(b)}))
  \prod_{1\le a<b\le m}
    \frac{G_{ii}^\pm(\sigma_1(z_{\sigma^{-1}(a)}),\sigma_1(z_{\sigma^{-1}(b)}))}
        {G_{ii}^\pm(\sigma_1(z_{a}),\sigma_1(z_{b}))}
    \\
&\times
  \prod_{1\le a\le r}\frac{1}{G_{ij}^\pm(\sigma_1(z_a),w)}
  \prod_{r<a\le m}\frac{1}{F_{ij}^\pm(\sigma_1(z_a),w)}
  \prod_{1\le a\le s}G_{ij}^\pm(\sigma_1(z_{\sigma^{-1}(a)}),w)\\
&\times
  \prod_{s<b\le m}F_{ij}^\pm(\sigma_1(z_{\sigma^{-1}(b)}),w)
  \al_{ij,r}^\pm(M,\theta,w)X_{ij,r}^\pm(M,w)
  \delta^\pm(M,\theta,\sigma_1\underline z_{\bar r},w)\\
&\times
  x_i^\pm(\sigma_1(z_{r+1}))\cdots x_i^\pm(\sigma_1(z_{m}))\\
=&\sum_{\theta\in \mathcal F_M}\sum_{\sigma_1\in S_{m}}
C_{r}^\pm(\sigma_1\underline z,w)
P^\pm(S_{m,r},
    \sigma_1\underline z,w)
\al_{ij,r}^\pm(M,\theta,w)X_{ij,r}^\pm(w)\\
&\times
\delta^\pm(M,\theta,\sigma_1\underline z,w)
x_i^\pm(\sigma_1(z_{r+1}))x_i^\pm(\sigma_1(z_{r+2}))\cdots x_i^\pm(\sigma_1(z_{m})),
\end{align*}
as desired.
\end{proof}

It is straightforward  to see that $\prod_{1\le a<b\le r} G_{ii}^\pm(z_a,z_b)$ divides
$P^\pm(S_{m,r},
    \underline z,w)$.
Combining this with Lemma \ref{lem:delta-theta-z-invariance},
one gets the well-defined current
\begin{align}\label{lem:CP-exists}
  C_r^\pm(\underline z,w)P^\pm(S_{m,r},
    \underline z,w)\delta^\pm(M,\theta,\sigma\underline z,w),
\end{align}
where $0\le r\le m$, $M\in\mathscr M_r$, $\theta\in\mathcal F_M$ and $\sigma\in S_r$.
Then we have

\begin{lem}\label{lem:Dr-pre001}
Let $0\le r\le m$, $M\in\mathscr N_r$ and
 $\theta_M\in\mathcal F_M$.
Then
\begin{align*}
  {D}_{r,M}^\pm=&
  \sum_{\sigma\in S_{m}}
  C_r^\pm(\sigma\underline z,w)
  P^\pm(S_{m,r}S_r,
    \sigma\underline z,w)
        \al_{ij,r}^\pm(M,\theta_M,w)X_{ij,r}^\pm(M,w)\\
&\times
    x_i^\pm(\sigma(z_{r+1}))\cdots x_i^\pm(\sigma(z_{m}))
        \delta^\pm(M,\theta_M, \sigma\underline z,w).
\end{align*}
\end{lem}

\begin{proof}
From  Lemma \ref{lem:Dr-to-g-commutator}, it follows that
\begin{align*}
  {D}_{r,M}^\pm=&\sum_{\sigma\in S_{m}}
  \sum_{\theta\in\mathcal F_M}
  C_r^\pm(\sigma\underline z,w)
  P^\pm(S_{m,r},
    \sigma\underline z,w)
        \al_{ij,r}^\pm(M,\theta,w)X_{ij,r}^\pm(M,w)\\
&\times
    x_i^\pm(\sigma(z_{r+1}))\cdots x_i^\pm(\sigma(z_{m})))
        \delta^\pm(M,\theta, \sigma\underline z,w)\\
=&\sum_{\sigma\in S_{m}}
  \sum_{\theta\in\mathcal F_M}
  C_r^\pm(\sigma\underline z,w)
  P^\pm(S_{m,r},
    \sigma\underline z,w)
        \al_{ij,r}^\pm(M,\theta,w)X_{ij,r}^\pm(M,w)\\
&\times
    x_i^\pm(\sigma(z_{r+1}))\cdots x_i^\pm(\sigma(z_{m}))
        \delta^\pm(M,\theta_M, \sigma\theta\inv\theta_M\underline z,w)\\
=&\sum_{\sigma\in S_{m}}
  \sum_{\theta\in\mathcal F_M}
  C_r^\pm(\sigma\theta_M^{-1}\theta \underline z,w)
  P^\pm(S_{m,r},
    \sigma\theta_M\inv\theta\underline z,w)
        \al_{ij,r}^\pm(M,\theta_M\theta_M\inv\theta,w)\\
&\times
    X_{ij,r}^\pm(M,w)x_i^\pm(\sigma(z_{r+1}))\cdots x_i^\pm(\sigma(z_{m}))
        \delta^\pm(M,\theta_M, \sigma\underline z,w)\\
=&\sum_{\sigma\in S_{m}}
  \sum_{\sigma_1\in S_r}
  C_r^\pm(\sigma\sigma_1 \underline z,w)
  P^\pm(S_{m,r},
    \sigma\sigma_1\underline z,w)
        \al_{ij,r}^\pm(M,\theta_M\circ\sigma_1,w)\\
&\times
    X_{ij,r}^\pm(M,w)x_i^\pm(\sigma(z_{r+1}))\cdots x_i^\pm(\sigma(z_{m}))
        \delta^\pm(M,\theta_M, \sigma\underline z,w)
\end{align*}
where the second equation follows from Lemma \ref{lem:delta-theta-z-invariance} and \eqref{lem:CP-exists}, and the last one follows from \eqref{lem:CP-exists}
and the convention that
\begin{align*}
  \al_{ij,r}^\pm(M,\theta_M\circ\sigma_1,w)=0\quad\te{if }\theta_M\circ\sigma_1\not\in \mathcal F_M.
\end{align*}
From \eqref{lem:Dr-to-g-commutator-lem0} and Proposition \ref{prop:al-coeff-rel}, we have
\begin{align}\label{eq:Dr-pre001-temp1}
  \al_{ij,r}^\pm(M,&\theta_M\circ\sigma_1,w)
  \delta^\pm(M,\theta_M, \underline z,w)
  =(-1)^{|\sigma_1|}
  \prod_{1\le a<b\le m_{ij}}\frac{G_{ii}^\pm(z_{\sigma_1(a)},z_{\sigma_1(b)})}
    {G_{ii}^\pm(z_a,z_b)}\nonumber\\
&\times
    \al_{ij,r}^\pm(M,\theta_M,w)
    \delta^\pm(M,\theta_M, \underline z,w).
\end{align}
Notice that
\begin{align*}
  C_r^\pm(\sigma_1\underline z,w)
  \prod_{1\le a<b\le m_{ij}}\frac{G_{ii}^\pm(z_{\sigma_1(a)},z_{\sigma_1(b)})}
    {G_{ii}^\pm(z_a,z_b)}
  =C_r^\pm(\underline z,w).
\end{align*}
Then one can conclude from \eqref{eq:Dr-pre001-temp1} that
\begin{align*}
  &C_r^\pm(\sigma_1\underline z,w)
  \al_{ij,r}^\pm(M,\theta_M\circ \sigma_1,w)
   P^\pm(S_{m,r},
    \sigma_1\underline z,w)
  \delta^\pm(M,\theta_M, \underline z,w)\\
=&(-1)^{|\sigma_1|}C_r^\pm(\underline z,w)
    \al_{ij,r}^\pm(M,\theta_M,w)
    P^\pm(S_{m,r},
    \sigma_1\underline z,w)
    \delta^\pm(M,\theta_M, \underline z,w).
\end{align*}
Combining these with the equation \eqref{eq:P-sigma-invariance}, we have that
\begin{align*}
  D_{r,M}^\pm
=&\sum_{\sigma\in S_{m}}
  \sum_{\sigma_1\in S_r}
  C_r^\pm(\sigma\underline z,w)
  (-1)^{|\sigma_1|}
  P^\pm(S_{m,r},
    \sigma\sigma_1\underline z,w)
        \al_{ij,r}^\pm(M,\theta_M,w)\\
&\times
    X_{ij,r}^\pm(M,w)x_i^\pm(\sigma(z_{r+1}))\cdots x_i^\pm(\sigma(z_{m}))
        \delta^\pm(M,\theta_M, \sigma\underline z,w)\\
=&\sum_{\sigma\in S_{m}}
  C_r^\pm(\sigma\underline z,w)
  P^\pm(S_{m,r}S_r,
    \sigma\underline z,w)
        \al_{ij,r}^\pm(M,\theta_M,w)\\
&\times
    X_{ij,r}^\pm(M,w)x_i^\pm(\sigma(z_{r+1}))\cdots x_i^\pm(\sigma(z_{m}))
        \delta^\pm(M,\theta_M, \sigma\underline z,w).
\end{align*}
Therefore, we complete the proof of the lemma.
\end{proof}

\begin{prop}\label{prop:general-D-r-M=0}
If
$P^\pm(S_m,\underline z,w)=0$,
then $D_{r,M}^\pm=0$ on $W$ for all  $0\le r\le m$ and $M\in\mathscr N_{r}$.
\end{prop}
\begin{proof}
From the assumption $P^\pm(S_m,\underline z,w)=0$, we get
\begin{align*}
  P^\pm(S_{m,r}S_r,
    \underline z,w)=\sum_{1\ne\sigma_1\in S_{\bar m\setminus\bar r}}
    (-1)^{|\sigma_1|+1}P^\pm(S_{m,r}S_r,
    \underline z \sigma_1,w).
\end{align*}
Then it follows from Lemma \ref{lem:Dr-pre001} that
\begin{align}
  {D}_{r,M}^\pm=&
  \sum_{\sigma\in S_{m}}
  C_r^\pm(\sigma\underline z,w)
  P^\pm(S_{m,r}S_r,
    \sigma\underline z,w)
        \al_{ij,r}^\pm(M,\theta_M,w)X_{ij,r}^\pm(M,w)\nonumber\\
&\times
    x_i^\pm(z_{\sigma(r+1)})\cdots x_i^\pm(z_{\sigma(m)})
        \delta^\pm(M,\theta_M, \sigma\underline z,w)\nonumber\\
=&\sum_{\sigma\in S_{m}}
  \sum_{1\ne\sigma_1\in S_{\bar m\setminus\bar r}}
  (-1)^{|\sigma_1|+1}
  C_r^\pm(\sigma\underline z,w)
  P^\pm(S_{m,r}S_r,
    \sigma\underline z\sigma_1,w)\nonumber\\
 &\times
    \al_{ij,r}^\pm(M,\theta_M,w)X_{ij,r}^\pm(M,w)
    x_i^\pm(\sigma(z_{r+1}))\cdots x_i^\pm(\sigma(z_{m}))
        \delta^\pm(M,\theta_M, \sigma\underline z,w).\nonumber
\end{align}
Notice that
$\displaystyle\prod_{\substack{r<a<b\le m\\ \sigma_1\inv(a)>\sigma_1\inv(b)}}F_{ii}^\pm(z_a,z_b)$
divides $P_m^\pm(S_{m,r}S_r,
    \underline z\sigma_1,w)$.
We deduce from (Q8) that
\begin{align*}
&C_r^\pm(\underline z,w)
  P^\pm(S_{m,r}S_r,
    \underline z\sigma_1,w)
    x_i^\pm(z_{r+1})\cdots x_i^\pm(z_{m})\\
=&(-1)^{|\sigma_1|}C_r^\pm(\sigma_1\underline z,w)
    P^\pm(S_{m,r}S_r,
    \underline z\sigma_1,w)\\
&\quad\times
    x_i^\pm(\sigma_1(z_{r+1}))\cdots x_i^\pm(\sigma_1(z_{m})).
\end{align*}
Combining these with the fact $\underline z\sigma_1=\sigma_1\underline z$, we get that
\begin{align}
{D}_{r,M}^\pm=&-
  \sum_{\sigma\in S_{m}}
  \sum_{1\ne\sigma_1\in S_{\bar m\setminus\bar r}}
  C_r^\pm(\sigma\sigma_1\underline z,w) P^\pm(S_{m,r}S_r,
    \sigma\sigma_1\underline z,w)\nonumber\\
&\times
        \al_{ij,r}^\pm(M,\theta_M,w)X_{ij,r}^\pm(M,w)
        \delta^\pm(M,\theta_M, \sigma\underline z,w)\nonumber\\
&\times x_i^\pm(\sigma\sigma_1(z_{r+1}))\cdots x_i^\pm(\sigma\sigma_1(z_{m}))\nonumber\\
=&
  (1-(m-r)!)
  \sum_{\sigma\in S_{m}}
  C_r^\pm(\sigma\underline z,w) P^\pm(S_{m,r}S_r,
    \sigma\underline z,w)\nonumber\\
&\times
        \al_{ij,r}^\pm(M,\theta_M,w)X_{ij,r}^\pm(M,w)
        \delta^\pm(M,\theta_M, \sigma\underline z,w)\nonumber\\
&\times x_i^\pm(\sigma(z_{r+1}))\cdots x_i^\pm(\sigma(z_{m}))\nonumber\\
=&(1-(m-r)!)D_{r,M}^\pm.\nonumber
\end{align}
Therefore, we have that
\begin{align*}
  {D}_{r,M}^\pm=0\quad\te{for }0\le r<m.
\end{align*}
For the case $r=m$, it follows from the assumption of the proposition that ${D}_{m,M}^\pm=0$.
This finishes  the proof of proposition.
\end{proof}

\textbf{Proof of Theorem \ref{prop:Dr-to-normal-ordering}:}
First, in the setting of Theorem \ref{prop:Dr-to-normal-ordering}, one can conclude from Lemma \ref{lem:Dr-to-g-commutator-pre000} and
Proposition \ref{prop:general-D-r-M=0} that
\begin{align}\label{dijpfb=drm}
 D_{ij}^\pm(m,f^\pm,B)=\sum_{r=m_{ij}}^m \sum_{M\in \mathscr N_r^c} D^\pm_{r,M}.
\end{align}
Assume now that  the relations \eqref{reforq10} hold.
Let  $M\in \mathscr N_r^c$ with $r=m_{ij},\dots,m$.
Then there exists $l=0,\dots,d_{ij}/d_i-1$ such that $[l]$ is a subset of $M$,
where
\begin{align*}
  [l]=A_{\xi_{d_{ij}/d_i}^l,-a_{ij}}^0=\set{(\xi_{d_{ij}/d_i}^l,a_{ij}+2p)}{0\le p\le -a_{ij}}.
\end{align*}
One notices that
\begin{align*}
\:\prod_{b\in [l]}x_i^\pm(\tau^\pm(b)w)x_j^\pm(w)\;
=\:x_i^\pm(q_i^{a_{ij}}\xi_{d_{ij}}^l w)
    x_i^\pm(q_i^{a_{ij}+2}\xi_{d_{ij}}^l w)
    \cdots x_i^\pm(q_i^{-a_{ij}}\xi_{d_{ij}}^l w)x_j^\pm(w)\;.
\end{align*}
This
implies that $\wt{X}_{ij,r}(M,\check M,w,\underline w)=0$ and hence $X_{ij,r}^\pm(M,w)=0$.
Then by definition  we have  $D_{r,M}^\pm=0$.
Thus it follows from \eqref{dijpfb=drm} that  $D_{ij}^\pm(m,f^\pm,B)=0$,
which proves the first assertion in Theorem \ref{prop:Dr-to-normal-ordering}.

 Assume next that   $m=m_{ij}$, $f^\pm(w,q_i^{\pm 2n}w)\ne 0$ for $n=1,\dots,-a_{ij}$  and $D_{ij}^\pm(m,f^\pm,B)=0$.
Note that with $m=m_{ij}$ we have
 \begin{align*}
 \mathscr N_{m}^c&=\{[l]\mid 0\le l\le d_{ij}/d_i-1\},\\
 \mathcal F_{[l]}&=\{\theta_l: p\mapsto (\xi_{d_{ij}/d_i}^l,-a_{ij}+2-2p)\}.
 \end{align*}
Then  by applying \eqref{dijpfb=drm}   we have
\begin{align*}
  D_{ij}^\pm(m,f^\pm,B)=&(-1)^{m_{ij}}\sum_{\sigma\in S_{m_{ij}}}
  \sum_{l=0}^{d_{ij}/d_i-1}\prod_{1\le a<b\le m_{ij}}
    f_{ij}^\pm(z_{\sigma(a)},z_{\sigma(b)})
    \\
  &\times\al_{ij,m_{ij}}^\pm([l],\theta_l,w)
  X_{ij,m_{ij}}^\pm([l],w)
\delta^\pm([l],\theta_l,\sigma\underline z,w).
\end{align*}
From Lemma \ref{lem:delta-function-independent} we get that
\begin{align*}
   \prod_{1\le a<b\le m_{ij}}
    f_{ij}^\pm(z_{\sigma(a)},z_{\sigma(b)})
    \al_{ij,m_{ij}}^\pm([l],\theta_l,w)
  X_{ij,m_{ij}}^\pm([l],w)
    \delta^\pm([l],\theta_l,\underline z,w)= 0.
\end{align*}
Thus from the condition $f^\pm(w,q_i^{\pm 2n}w)\ne 0$ for $n=1,\dots,-a_{ij}$, it follows that
\begin{align*}
  X_{ij,m_{ij}}^\pm([l],w)=0\quad\te{for } 0\le l<d_{ij}/d_i.
\end{align*}
Since
\begin{align*}
  X_{ij,m_{ij}}^\pm([l],w)=\:x_i^\pm(q_i^{a_{ij}}\xi_{d_{ij}}^l w)
    x_i^\pm(q_i^{a_{ij}+a_{ii}}\xi_{d_{ij}}^l w)
    \cdots x_i^\pm(q_i^{-a_{ij}}\xi_{d_{ij}}^l w)x_j^\pm(w)\;,
\end{align*}
 we complete the proof of Theorem \ref{prop:Dr-to-normal-ordering}.

\section*{Acknowledgment}

The authors would like to thank the  referee for his/her valuable comments that greatly improved the exposition of the paper.
In particular, we are grateful to the referee for pointing out an error in Lemma \ref{lem:LC1aff} in the original version.

\end{document}